\documentclass{amsart}
\usepackage[utf8]{inputenc}

\title[The fractional-logarithmic Laplacian]{The fractional-logarithmic Laplacian:\\ fundamental properties and eigenvalues}

\author{Huyuan Chen}
\address{Huyan Chen, Center for Mathematics and Interdisciplinary Sciences, Fudan University, 
Shanghai 200433, China;\newline 
\indent Huyan Chen, Shanghai Institute for Mathematics and Interdisciplinary Sciences, Shanghai 200433, China.}
\email{chenhuyuan@yeah.net}

\author{Rui Chen}
\address{Rui Chen, School of Mathematical Sciences, Fudan University, Shanghai 200433, China;\newline \indent Rui Chen, Brandenburg University of Technology Cottbus–Senftenberg, Platz der Deut\-schen Einheit 1, 03046 Cottbus, Germany}
\email{chenrui23@m.fudan.edu.cn}

\author{Daniel Hauer}
\address{Daniel Hauer, Brandenburg University of Technology Cottbus–Senftenberg, Platz der Deutschen Einheit 1, 03046 Cottbus, Germany;\newline 
\indent Daniel Hauer, School of Mathematics and Statistics, The University of Sydney, Sydney, NSW, 2006, Australia }
\email{daniel.hauer@b-tu.de}

\usepackage{amssymb}
\usepackage{amsfonts,dsfont}
\usepackage{amsmath}
\usepackage{amsthm}
\usepackage{cite}
\usepackage{enumitem}
\usepackage{mathrsfs}

\usepackage[
  hyperfootnotes=false,
  colorlinks=true,
  linkcolor=blue,
  citecolor=blue,
  urlcolor=blue,
  pdftitle={The fractional-logarithmic Laplacian},
  pdfauthor={Huyuan Chen, Rui Chen and
Daniel Hauer}
]{hyperref}

\newcommand{\s}{s}

\usepackage{xcolor}

\newcommand{\R}{\mathbb{R}}

\newcommand{\N}{\mathbb{N}}

\newcommand{\cuad}{{\sqcap\kern-.68em\sqcup}}
\newcommand{\abs}[1]{\mid #1 \mid}
\newcommand{\norm}[1]{\|#1\|}

\newtheorem{theorem}{Theorem}[section]
\newtheorem{proposition}{Proposition}[section]
\newtheorem{lemma}{Lemma}[section]

\newtheorem{definition}{Definition}[section]
\newtheorem{notation}{Notation}[section]

\newtheorem{remark}{Remark}[section]

\newcommand{\bremark}{\begin{remark} \em}
	\newcommand{\eremark}{\end{remark} }

\renewcommand{\div}{\,{\rm div}\,}
\renewcommand{\phi}{\varphi}

\newcommand{\cE}{{\mathcal E}}
\newcommand{\cF}{{\mathcal F}}

\newcommand{\cH}{{\mathcal H}}

\newcommand{\cL}{{\mathcal L}}
\newcommand{\cM}{{\mathcal M}}
\newcommand{\cN}{{\mathcal N}}

\newcommand{\cS}{{\mathcal S}}

\DeclareMathOperator*{\PV}{{\rm p.v.}}
\DeclareMathOperator*{\Log}{Log}

\allowdisplaybreaks



\begin{document}

\begin{abstract}
    In this paper, we introduce the fractional--logarithmic Laplacian \( (-\Delta)^{s+\Log} \), defined as the derivative of the fractional Laplacian $(-\Delta)^t$ at \(t=s\). It is a singular integral operator with Fourier symbol \( |\xi|^{2s}(\ln|\xi|^2) \) and we derive its pointwise integral representation
    \begin{displaymath}
    (-\Delta)^{s+\Log}u(x):=
c_{n,s}\,\mathcal \PV\int_{\mathbb{R}^n} \frac{u(x)-u(y)}{|x-y|^{n+2s}}\bigl(-2\ln|x-y|\bigr)\,dy+b_{n,s}(-\Delta)^s u(x),        
    \end{displaymath}
where \(c_{n,s}\) is the normalization constant of the fractional Laplacian, and  $b_{n,s}:=\frac{d}{ds}\ln c_{n,s}.$ We also establish other equivalent formulations of \( (-\Delta)^{s+\Log}\), including its definition via spectral calculus, and its characterization through an extension problem.

We develop the functional-analytic framework for studying the fractional--logarithmic operator \((-\Delta)^{s+\Log}\) both on \(\R^n\) and on bounded domains with continuous boundary. We 
introduce the natural corresponding energy spaces and establish Sobolev and Poincaré embeddings; in particular, we obtain
a compact embedding at the critical exponent $2_s^*=\frac{2n}{n-2s}$ when $n>2s$, a phenomenon that differs from the classical and fractional Sobolev settings. 

We further study the Poisson problem, proving existence and \(L^\infty\)-regularity results. We then investigate the Dirichlet eigenvalue problem and establish qualitative spectral properties. Finally, we derive Weyl-type asymptotics for the eigenvalue counting function and for the $k$-th Dirichlet eigenvalue, showing that the high-frequency behaviour combines the fractional Weyl scaling with the logarithmic growth factor, thus interpolating between the fractional Laplacian and the logarithmic Laplacian.
\end{abstract}

\maketitle

	\tableofcontents

	\setcounter{equation}{0}
	\section{Introduction and Main results}
	In recent years, there has been a significant increase of interest in elliptic and parabolic boundary-value problems, as well as eigenvalue problems, involving hypersingular integral operators (see, for instance, \cite{bucur-valdinoci,kuusi-palatucci,Molica-Bisci-Radulescu-Servadei,Stinga}). In this context, the \emph{fractional Laplacian} $(-\Delta)^s$, for $0<s<1$, serves as a fundamental prototype and is defined by
	\begin{equation}\label{def:frac-Lap}
		(-\Delta)^s u(x)
		= c_{n,s}\,\PV \int_{\mathbb{R}^n}
		\frac{u(x)-u(y)}{|x-y|^{n+2s}}\,dy,
	\end{equation}
	for every $x\in \mathbb{R}^n$ and $u\in C_c^{2}(\mathbb{R}^n)$.
	In~\eqref{def:frac-Lap}, the normalization constant
	\begin{equation}\label{norm const1}
		c_{n,s}:=2^{2s}\pi^{-\frac n2}s\,\frac{\Gamma\!\left(\frac{n+2s}{2}\right)}{\Gamma(1-s)}
	\end{equation}
	is chosen so that $(-\Delta)^s$ has Fourier symbol $|\xi|^{2s}$ for $\xi\in\mathbb{R}^n$ (cf.~\cite[Chapter~3]{EGE}), where $\Gamma$ denotes the Gamma function. With this normalization one also has the small-order limit $\lim_{s\to0^+}(-\Delta)^s u(x)=u(x)$
	for $u\in C^2_c(\mathbb{R}^n)$ and $x\in\mathbb{R}^n$.
	
	Motivated by this limit, Weth and the first author \cite{CW18} introduced the \emph{logarithmic Laplacian} $ (-\Delta)^{\Log}$ as the first-order correction term in the expansion as $s\to0^+$. More precisely, for $u\in C^\alpha_c(\mathbb{R}^n)$ with some $\alpha>0$,  $ (-\Delta)^{\Log}u$ is uniquely defined by
	\begin{equation}\label{eq:def-log-Laplace}
		(-\Delta)^s u(x)=u(x)+ s\, (-\Delta)^{\Log}u(x)+o(s)\ \ \text{in}\ \  L^p(\mathbb{R}^n),\ 1<p\le \infty,\ \s\rightarrow 0^+
	\end{equation}
	or equivalently,
	\[
	 (-\Delta)^{\Log}u=\big.\frac{d}{ds}(-\Delta)^s u\big|_{s=0}\quad\text{in}\quad L^p(\mathbb{R}^n),\ 1<p\le \infty.
	\]
	Moreover, they proved (see \cite[Theorem~1.1]{CW18}) that $ (-\Delta)^{\Log}$ has the Fourier symbol $2\ln|\xi|$ for $\xi\in\mathbb{R}^n\setminus\{0\}$, and admits the integral representation
	\[ (-\Delta)^{\Log}u(x)=c_n
	\PV \int_{B_1(0)}\frac{u(x)-u(x+y)}{|y|^{n}}dy
	- c_n\int_{\mathbb{R}^n\setminus B_1(0)}\frac{u(x+y)}{|y|^{n}}dy
	+\rho_nu(x)\]
	for every Dini-continuous function $u:\mathbb{R}^n\to\mathbb{R}$ belonging to the tail space
	\begin{displaymath}
	L^1_0(\mathbb{R}^n)
	:= \Bigl\{ u \in L^1_{\mathrm{loc}}(\mathbb{R}^n)\;\Big\vert\;
	\int_{\mathbb{R}^n}(1+|x|)^{-n}|u(x)|\,dx < \infty\Bigr\}.
	\end{displaymath}
	Here, and throughout this paper, we write $B_1(0)$ to denote the open unit ball in $\mathbb{R}^n$ centred at the origin, and set
	\begin{displaymath}
	c_n=\frac{\Gamma(n/2)}{\pi^{n/2}}=\frac{2}{|\mathbb{S}^{n-1}|}
    \quad\text{ and }
	\quad
	\rho_n :=2\ln 2+\psi(n/2)-\gamma,    
	\end{displaymath}
	where $\gamma=-\Gamma'(1)$ is the Euler--Mascheroni constant, $\psi=\Gamma'/\Gamma$ is the Digamma function, and $|\mathbb{S}^{n-1}|$ denotes the $(n-1)$-dimensional surface measure of the unit sphere in $\mathbb{R}^n$.
	
	More recently, the first and third authors established in joint work with Weth \cite{CDW} that the logarithmic Laplacian $ (-\Delta)^{\Log}$ admits an \emph{extension property}, and this one is quite different from the Caffarelli-Silvestre extension for the fractional Laplacian, see~\cite{Caffarelli-Silvestre}. Specifically, they showed that for every $u\in L^1_0(\mathbb{R}^n)$,
	\begin{displaymath}
	 (-\Delta)^{\Log} u
	= 2(\ln 2-\gamma)u
	- 2\lim_{t\to 0^+}\Bigl(w_u(\cdot,t)+u\ln t\Bigr)
	\quad \text{in }\mathscr{D}'(\mathbb{R}^n),
	\end{displaymath}
	where
\[
w_u(x,t):=\frac{c_n}{2}\int_{\mathbb R^n}
\frac{u(y)}{(|x-y|^2+t^2)^{n/2}}\,dy
\]
is the extension constructed in \cite[Theorem~1.2]{CDW}, which solves the Neumann problem
	\begin{displaymath}
	\begin{cases}
		-  \div_{(x,t)}(t\,\nabla w_u)=0 & \text{in }\mathbb{R}^{n+1}_+,\\[2pt]
		-\displaystyle\lim_{t\to0^+} t\,\partial_t w_u(\cdot,t)=0 & \text{on }\mathbb{R}^n,
	\end{cases}
	\qquad
	\mathbb{R}^{n+1}_+:=\mathbb{R}^n\times(0,\infty).
	\end{displaymath}
	
	Through various applications, it has become clear that it is useful to understand the logarithmic Laplacian $ (-\Delta)^{\Log}$ and its fundamental properties. For instance, it appears in the study of the asymptotics of the Dirichlet eigenvalue problem for $(-\Delta)^s$ as $s\to0^+$ (see \cite{CW18,FeuJa}); in $s$-dependent nonlinear Dirichlet problems in the small-order regime (see \cite{angeles-saldana,hs.saldana}), motivated in particular by order-dependent optimization problems with small optimal order (see, e.g., \cite{antil.bartels,sprekels.valdinoci,pellacci.verzini} for applications in image processing and population dynamics); and in the geometric context of the $0$-fractional perimeter (see \cite{DNP}). Recently, the second author introduced in \cite{ChenR} a definition of $ (-\Delta)^{\Log}$ on any complete Riemannian manifold using spectral-theoretic tools, together with an explicit integral representation. More recently, the \emph{conformal logarithmic Laplacian} on the unit sphere has been studied by Fern\'andez and Salda\~na \cite{FS2025}.
	\medskip
	
	Returning to the expansion~\eqref{eq:def-log-Laplace}, it is natural to study the first-order derivative
	\[
	\big.\frac{d}{dt}(-\Delta)^{t}u\big|_{t=s}
	\]
	at an arbitrary order $0<s<1$ for some suitable function $u$.  The following proposition establishes the existence of the derivative provided that $u\in C_c^2(\mathbb{R}^n)$.

    \begin{proposition}
\label{prop:dt-fraclap}
Let $0<s<1$ and $u\in C_c^2(\R^n)$. Then for every fixed $x\in\R^n$, the map $t\mapsto (-\Delta)^t u(x)$ is $C^1$ on $(0,1)$, and
\begin{align}
\label{eq:dt-fraclap}
\left.\frac{d}{dt}(-\Delta)^t u(x)\right|_{t=s}
&=
c_{n,s}\,\mathcal{L}_1 u(x) + b_{n,s}\,(-\Delta)^s u(x),
\end{align}
	where $c_{n,s}$ is given by~\eqref{norm const1},
			\begin{equation}\label{eq:const-bns}
				b_{n,s}:=\frac{d}{ds}\ln c_{n,s}= \ln 4 + \frac1s +\psi(1-s) +\psi\!\big(\frac{n+2s}{2}\big)
			\end{equation}
    and
			\begin{equation}\label{liiner}
			    \mathcal{L}_1u(x)
				:= \PV\int_{\mathbb{R}^n} \frac{u(x)-u(y)}{|x-y|^{n+2s}}\bigl(-2\ln|x-y|\bigr)\,dy,
			\end{equation}
			where $\psi$ denotes the Digamma function, i.e., $\psi(z):=\frac{\Gamma'(z)}{\Gamma(z)}$.
\end{proposition}

This motivates the following definition.

	\begin{definition}\label{def:1}
		Let $0<s<1$. Then, we call the operator $(-\Delta)^{s+\Log}$ defined by
		\[
		(-\Delta)^{s+\Log}u(x):=\big.\frac{d}{dt}(-\Delta)^{t}u(x)\big|_{t=s},
		\quad x\in\mathbb{R}^n
		\]
		for every $u\in C_c^2(\mathbb{R}^n)$  \emph{the fractional-logarithmic Laplacian} $(-\Delta)^{s+\Log}$.
	\end{definition}
	
This type of operator also appears naturally in the study of the
differentiability of nonlocal-to-local transitions. In particular, in the work
of Jarohs, Salda{\~n}a, and Weth~\cite{JarohsSaldanaWeth2025} on fractional
Poisson problems depending on the fractional order, operators of the form
\(L_\Delta(-\Delta)^s\) arise in the corresponding differentiation process.
Their identities also include the endpoint case \(s=1\), where an additional
term appears in the relation between the logarithmic Laplacian and the
fractional Laplacian. This provides another motivation for the study of
fractional-logarithmic Laplacian and connects the present operator with
nonlocal-to-local limiting problems.

The aim of this paper is to study rigorously fundamental properties of the fractional-logarithmic Laplacian $(-\Delta)^{s+\Log}$ on $\mathbb{R}^n$. Note, here and throughout this paper, we denote by $\mathcal{F}$ or $\widehat{\cdot}$ the Fourier transform \[ \mathcal{F}[u](\xi)=\widehat{u}(\xi) =\frac1{(2\pi)^{n/2}}\int_{\mathbb{R}^n}e^{-i\xi\cdot x}\,u(x)\,dx, \qquad \xi\in\mathbb{R}^n, \] for $u\in L^{1}(\mathbb{R}^n)$. We assign the value $0$ to $|\xi|^{2s}\ln|\xi|^2$ at $\xi=0$ whenever $s>0$. We also denote the Lizorkin space (see \cite[Section 5.1.2]{MR0730762})
\[
\mathcal{Z}(\mathbb{R}^n)
:=\Bigl\{\varphi\in\mathscr{S}(\mathbb{R}^n)\,\big\vert\, \partial^\alpha\widehat{\varphi}(0)=0\ \text{for all }\alpha\in\mathbb{N}_0^{\,n}\Bigr\},
\]
as well as the corresponding topological dual $\mathcal{Z}'(\mathbb{R}^n)$, i.e., the space of all continuous linear
functionals on $\mathcal{Z}(\mathbb{R}^n)$, where $\mathscr S(\mathbb R^n)$ and $\mathscr S'(\mathbb R^n)$ denote, respectively, the Schwartz space on $\mathbb R^n$ and the space of tempered distributions on $\mathbb R^n$. Within this framework, for every symbol $a(\xi)$ of the class
\begin{displaymath}
    \Bigl\{|\xi|^{2s},\ \ln|\xi|^2,\ |\xi|^{2s}\ln|\xi|^2 \Bigr\},
\end{displaymath}
the associated singular pseudo-differential operator $A(D)$ defines a continuous linear map $A(D) : \mathcal Z'(\R^n)\to\mathcal Z'(\R^n)$ given by
\begin{displaymath}
\langle A(D)u,\varphi\rangle_{\mathcal Z',\mathcal Z}
:=\Big\langle u,\ \mathcal F^{-1}\!\big(a(\xi)\,\widehat{\varphi}(\xi)\big)\Big\rangle_{\mathcal Z',\mathcal Z}    
\end{displaymath}
for every $u\in\mathcal Z'(\R^n)$ and
$\varphi\in\mathcal Z(\R^n)$. In particular, the operators $(-\Delta)^s$, $ (-\Delta)^{\Log}$ and $(-\Delta)^{s+\Log}$ are well-defined continuous linear operators on $\mathcal Z'(\R^n)$.\medskip

In the following, we present several alternative representations and basic properties of the fractional-logarithmic Laplacian $(-\Delta)^{s+\Log}$. Depending on the context, one may choose the formulation that is most suitable for the problem under consideration.

	\begin{theorem}\label{pr 1.1}
		Let $s\in(0,1)$. Then the following statements hold.
		\begin{enumerate}[label=\normalfont\textup{(\roman*)}]
			\item\label{pr 1.1:claim 1} (\textrm{Integral representation})
			For every $u\in C_c^2(\mathbb{R}^n)$ and $x\in\mathbb{R}^n$,
			\[(-\Delta)^{s+\Log} u(x)
				= c_{n,s}\,\mathcal{L}_1 u(x) + b_{n,s}\,(-\Delta)^s u(x),\]
			where $c_{n,s}, b_{n,s},\mathcal{L}_1$ are defined in~\eqref{norm const1}, \eqref{eq:const-bns} and \eqref{liiner}, respectively.

			\item\label{pr 1.1:claim 2} (\textrm{Kernel representation})
			For every $u\in C_c^{2}(\mathbb{R}^n)$ and $x\in\mathbb{R}^n$,
		\[(-\Delta)^{s+\Log} u(x)
				= \PV\int_{\mathbb{R}^n} \bigl(u(x)-u(y)\bigr)\,{\bf K}_{s+\Log}(|x-y|)\,dy,\]
			where
			\begin{equation}\label{def fract+log-an2}
				{\bf K}_{s+\Log}(r)
				:= c_{n,s}\,\bigl(b_{n,s}-2\ln r \bigr)\,r^{-n-2s},
				\quad r>0.
			\end{equation}

            \item\label{pr 1.1:claim 3} (\textrm{Fourier symbol})
			The fractional-logarithmic Laplacian $(-\Delta)^{s+\Log}$ has the singular Fourier symbol $|\xi|^{2s}\,(2\ln|\xi|)$ for $\xi\in\mathbb{R}^n$. More precisely,
			\[
			\mathcal{F}\bigl[(-\Delta)^{s+\Log} u\bigr](\xi)
			= |\xi|^{2s}\,\bigl(2\ln|\xi|\bigr)\,\widehat{u}(\xi),
			\quad \xi\in\mathbb{R}^n
			\]
			for every $u\in C_c^2(\mathbb{R}^n)$.

			\item\label{pr 1.1:claim 4} (\textrm{Representation as a composition})
			For every $u\in \mathcal{Z}'(\mathbb{R}^n)$,
\begin{displaymath}
(-\Delta)^{s+\Log}u
=
(-\Delta)^s\bigl( (-\Delta)^{\Log}u\bigr)
=
 (-\Delta)^{\Log}\bigl((-\Delta)^s u\bigr) \quad \text{in}\ \, \mathcal{Z}'(\mathbb{R}^n).   
\end{displaymath}
    
			\item\label{pr 1.1:claim 5} (\textrm{Stability})
			For every $u\in C_c^2(\mathbb{R}^n)$,
			\[(-\Delta)^{s+\Log}u \to  (-\Delta)^{\Log}u
				\quad\ \text{in }\, L^{\infty}(\mathbb{R}^n)\ \ \text{as }s\to0^+.\]
		\end{enumerate}
	\end{theorem}
	
We make the following remark concerning Theorem~\ref{pr 1.1}.

	\begin{remark}\label{rem:1}
    {\normalfont Here, we want to mention the following observations.
	\begin{enumerate}[label=\normalfont\textup{(\roman*)}]
			\item The constant $b_{n,s}$ defined in~\eqref{eq:const-bns} satisfies the following limits
			\begin{displaymath}
			  \lim_{s\to 0^+}b_{n,s}=\infty\quad\text{and}\quad 
               \lim_{s\to 1^-}b_{n,s}=-\infty.  
			\end{displaymath}
			Moreover, \(b_{n,s}\) is monotonically decreasing with respect to \(s\in (0,1)\), and there exists an  \(s_0\in\left(\tfrac12,1\right)\) such that \(b_{n,s}\ge 0\) for every \(s\in(0,s_0)\).
		\item The kernel \({\bf K}_{s+\Log}\) appearing in~\eqref{def fract+log-an2} is not nonnegative in general.
		\end{enumerate}}
	\end{remark}
	
We now describe the fractional-logarithmic Laplacian \((-\Delta)^{s+\Log}\)
from the viewpoint of functional calculus. We refer to~\cite{Conway1990} for
the basic spectral-theoretic framework used below. Let $-\Delta$ be the nonnegative self-adjoint Laplacian, by the spectral theorem, there exists a spectral resolution
$E(\cdot)$ on $[0,\infty)$ such that
\[ -\Delta= \int_{[0,\infty)} \lambda\, dE(\lambda).\]
For \(u,v\in L^2(\mathbb R^n)\), we write $E_{u,v}(B):=\langle E(B)u,v\rangle_{L^2(\mathbb R^n)}$
for the associated complex measure; in particular, \(E_{u,u}\) is a finite
positive measure. It is well known that the fractional Laplacian $(-\Delta)^s$ defined by the singular integral operator~\eqref{def:frac-Lap} can alternatively be defined via the spectral calculus as the mapping
\[
(-\Delta)^s: H^{2s}(\mathbb R^n)
\longrightarrow L^2(\mathbb R^n)\text{ given by }
(-\Delta)^s
=
\int_{[0,\infty)}\lambda^s\,dE(\lambda),
\]
where we refer to the space
\[
H^{2s}(\mathbb R^n)
:=
\Bigl\{
u\in L^2(\mathbb R^n)\,\Big\vert\,
\int_{[0,\infty)} \lambda^{2s}\,dE_{u,u}(\lambda)<\infty
\Bigr\},
\]
as the \emph{spectral fractional Sobolev space}. We provide more
details in Section~\ref{spectral}.\medskip 

Further, due to the well-known limit
	\begin{displaymath}
	\lim_{t\to s}\frac{\lambda^{t}-\lambda^s}{t-s}
	=\lambda^s\,\ln\lambda
	\end{displaymath}
	for every given $\lambda>0$ and $s\in (0,1)$, it is natural to define the fractional--logarithmic power
	of $-\Delta$ by
\[(-\Delta)_{\mathrm{spec}}^{s+\Log}
		:=\int_{(0,\infty)} \lambda^s\ln\lambda\, dE(\lambda)
\]
and the \emph{spectral fractional--logarithmic Sobolev space} by 
\begin{equation}\label{domainlo}
    H^{2s+\Log}(\R^n):=\Bigl\{u\in L^2(\R^n)\,\Big\vert\,
	\int_{(0,\infty)}\lambda^{2s}\ln^2\lambda\, dE_{u,u}(\lambda)<\infty\Bigr\}.
\end{equation}

\begin{remark}
   {\normalfont It is easy to see that for every $s\in (0,1)$, \(H^{s+\Log}(\R^n)\) is a linear subspace of the spectral fractional Sobolev space $H^{s}(\R^n)$.}
\end{remark}

Next, we intend to compare the two operators $(-\Delta)_{\mathrm{spec}}^{s+\Log}$ and $(-\Delta)^{s+\Log}$. We make this precise in the following proposition.

	\begin{proposition}
		\label{prop:derivative-frac-power}
		Let $s\in(0,1)$ and assume that $u\in H^{2s_0}(\R^n)$ for some $s_0>s.$
		Then, 
		\[
		\lim_{t\to s}\Bigl\|
		\frac{(-\Delta)^t-(-\Delta)^s}{t-s}\,u - (-\Delta)_{\mathrm{spec}}^{s+\Log}u
		\Bigr\|_{L^2(\R^n)}=0.
		\]
	\end{proposition}

Thanks to Proposition \ref{prop:derivative-frac-power}, the definition of the fractional--logarithmic operator $(-\Delta)_{\mathrm{spec}}^{s+\Log}$ via the spectral measure coincides with the operator $(-\Delta)^{s+\Log}$ given in Definition~\ref{def:1} for functions $u\in C_c^2(\mathbb{R}^n)$.\medskip

Further, since the spectral measure $E_{f,f}$ of the negative
Laplacian $-\Delta$ on $L^{2}(\R^n)$ can be expressed by the Fourier
transform (see, e.g., \cite[Chapter 7]{Teschl2014})
\[
dE_{f,f}(\lambda)
=F'(\lambda)\,d\lambda
=\frac12\,\lambda^{\frac n2-1}
\int_{\mathbb{S}^{n-1}}\bigl|\widehat f(\sqrt\lambda\,\omega)\bigr|^2\,d\omega\,d\lambda,
\]
we obtain the following characterization and representation of the fractional--logarithmic operator $(-\Delta)^{s+\Log}$.

\begin{theorem}\label{thm:fourspec}
Let $s>0$.

\begin{enumerate}[label=\normalfont\textup{(\roman*)}]
\item The spectral fractional--logarithmic Sobolev space $H^{2s+\Log}(\R^n)$ defined in \eqref{domainlo} admits the following characterization
\[
H^{2s+\Log}(\mathbb{R}^n)
=\Bigl\{u\in L^2(\mathbb{R}^n)\,\Big|\,
\int_{\mathbb{R}^n}|\xi|^{4s}\,\ln^2\!\bigl(|\xi|^2\bigr)\,|\widehat u(\xi)|^2\,d\xi<\infty\Bigr\}.
\]

\item For every $u\in H^{2s+\Log}(\mathbb{R}^n)$, $(-\Delta)^{s+\Log}u$ is well-defined
as an $L^2(\mathbb{R}^n)$ function and satisfies
\begin{equation}\label{eq:fourier-symbol-frac-log}
(-\Delta)^{s+\Log}u
=\mathcal F^{-1}\!\Big(|\xi|^{2s}\ln(|\xi|^2)\,\widehat u(\xi)\Big)
\quad\text{in }\quad L^2(\mathbb{R}^n).
\end{equation}
\end{enumerate}
\end{theorem}

	Next we turn to the \emph{extension property} on the half-space $\mathbb{R}^{n+1}_{+}:=\mathbb{R}^{n}\times (0,\infty)$ of the fractional-logarithmic Laplacian $(-\Delta)^{s+\Log}$. For this, we need to introduce the following notation. For a given $u\in C_c^\infty(\R^n)$, let
\begin{equation}\label{pot 2.0}
w_s(x,t):=p_{n,s}\int_{\R^n}\frac{t^{2s}u(y)}{(|x-y|^2+t^2)^{\frac{n+2s}{2}}}\,dy,
\qquad (x,t)\in\R^{n+1}_+,
\end{equation}
where
\[
p_{n,s}:=\Big(\displaystyle\int_{\R^n}\frac{1}{(|z|^2+1)^{\frac{n+2s}{2}}}\,dz\Big)^{-1}=\pi^{-\frac n2} \frac{\Gamma(\frac{n+2s}2)}{\Gamma(s) }.
\]
Then $w_s$ is the solution to the weighted Dirichlet problem
\[
\begin{cases}
\div_{(x,t)}\!\big(t^{1-2s}\nabla_{(x,t)}w_s\big)=0 & \text{in }\R^{n+1}_+,\\[2.5mm]
 \phantom{ -------\!   }
w_s(\cdot,0)=u & \text{on }\R^n
\end{cases}
\]
(see~\cite{CDW}). Further, set
	\begin{equation}\label{pot 2.11}
		v_s (x,t):=p_{n,s}\ \partial_s(\frac{w_s}{p_{n,s}})=p_{n,s}\  t^{2s}  \int_{\R^n} \frac{ 2\ln t-\ln(|x-y|^2+t^2)   }{(|x-y|^2+t^2)^{\frac{n+2s}{2}}}u(y)dy 
	\end{equation}
    and
\begin{equation}\label{dsbs}
    {\bf d}_{s}:=\frac{c_{n,s}}{p_{n,s}}=2^{2\s}   \frac{ \s \ \Gamma(\s) }{ \Gamma(1-s)  },\quad    b_1:=p_{n,s} \int_{\R^n} \frac{  -\ln(|z|^2+1)   }{(|z|^2+1)^{\frac{n+2s}{2}}} dz. 
\end{equation}

Now, with these preliminaries in mind, the following theorem provides a characterization of the fractional-logarithmic Laplacian $(-\Delta)^{s+\Log}$ through an elliptic \emph{extension problem} on the upper half-space.

\begin{theorem}\label{pr 2.1}
	Let \(s\in(0,1)\), let \(w_s\), \(v_s\), \({\bf d}_s\), and \(b_1\)
	be defined as in~\eqref{pot 2.0}, \eqref{pot 2.11}, and~\eqref{dsbs},
	respectively, and let \(u\in C_c^\infty(\mathbb R^n)\). Then \(v_s\)
	solves the inhomogeneous Dirichlet problem
	\begin{equation}\label{eq 1.1-ext}
		\begin{cases}
			\displaystyle
			\operatorname{div}_{(x,t)}
			\!\bigl(t^{1-2s}\nabla_{(x,t)}v_s\bigr)
			=
			2t^{-2s}\partial_t w_s
			& \text{in } \mathbb R^{n+1}_+,
             \\[2.5mm]
 \phantom{ -------\!\!   }
			\displaystyle
			v_s(\cdot,0)=b_1u
			& \text{on } \mathbb R^n.
		\end{cases}
	\end{equation}
	Moreover,
	\begin{align*}
		(-\Delta)^{s+\Log}u(x)
		=
		-\lim_{t\to0^+}{\bf d}_s
		\Big[
		&\bigl(b_{n,s}-2\ln t\bigr)t^{-2s}
		\bigl(w_s(x,t)-w_s(x,0)\bigr)  \\
		&\hspace{2.3cm}
		+t^{-2s}
		\bigl(v_s(x,t)-v_s(x,0)\bigr)
		\Big].
	\end{align*}
\end{theorem}

Next, we study the fractional--logarithmic operator
\((-\Delta)^{s+\Log}\) through its singular integral kernel. Here and in what follows, we use the notation $a_\pm:=\max\{0,\pm a\}$ to define the \emph{positive}, respectively, \emph{negative} part of $a$. 

\begin{definition}
    \label{def:kernel}
    For $0<s<1$, let ${\bf k}_{s+\Log,\pm}: \R^n\setminus\{0\}\to [0,\infty)$ be defined by
    \[
	{\bf k}_{s+\Log,\pm}(z):=c_{n,s}\,|z|^{-n-2s}\,(-\ln|z|)_{\pm},\qquad z\in\R^n\setminus\{0\}.
	\]
    Then, the \emph{fractional--logarithmic kernel} ${\bf K}_{s+\Log}: \R^n\setminus\{0\}\to\R$ is defined by
	\[
	{\bf K}_{s+\Log}(z)
	= c_{n,s}b_{n,s}\,|z|^{-n-2s}+2{\bf k}_{s+\Log,+}(z)-2{\bf k}_{s+\Log,-}(z)
	\]
    for every $z\in\R^n\setminus\{0\}$.
\end{definition}

Associated to the positive part of the
fractional--logarithmic kernel, we introduce, now, the 
\emph{Gagliardo-type fractional--logarithmic Sobolev space}.
    
\begin{definition}
    \label{def:Gagliardo-fractional-log-sobolev-space}
    For $0<s<1$, we call the function space 
    \[
	\mathcal H^{s+\Log}(\R^n)
	:=\Bigl\{
	u\in L^2(\R^n)\,\Big|\, 
	\iint_{\R^n\times\R^n}\bigl(u(x)-u(y)\bigr)^2\,
	{\bf k}_{s+\Log,+}(x-y)\,dx\,dy<\infty
	\Bigr\}.
	\]
    the \emph{Gagliardo-type fractional--logarithmic Sobolev space}, and we equip this space with the norm
    \begin{equation}
        \label{eq:def-frac-loc-norm}
        \norm{u}_{\mathcal H^{s+\Log}}:=\sqrt{\|u\|_{L^2(\R^n)}^{\,2}+[u]_{s+\Log,+}^2},
    \end{equation}
    where the Gagliardo-type seminorm $[\cdot]_{s+\Log,+}$ is given by
	\begin{equation}
    \label{def:gagliardo-seminorm}
	[u]_{s+\Log,+}^2
	:=\iint_{\R^n\times\R^n}\bigl(u(x)-u(y)\bigr)^2\,
	{\bf k}_{s+\Log,+}(x-y)\,dx\,dy
	\end{equation}
    for every $u\in \mathcal H^{s+\Log}(\R^n)$.
\end{definition}

It is not difficult to see that the norm $\norm{\cdot}_{\mathcal H^{s+\Log}}$ is induced by the inner product
\begin{equation}
    \label{eq:inner-prod-fraglog} 
		\begin{split}
		    \langle u,&v\rangle_{\mathcal H^{s+\Log}}
			:=\int_{\R^n}u(x)v(x)\,dx\\
			&\quad + \iint_{\R^n\times\R^n}
			\bigl(u(x)-u(y)\bigr)\bigl(v(x)-v(y)\bigr)\,
			{\bf k}_{s+\Log,+}(x-y)\,dx\,dy,
		\end{split}
\end{equation}
for every $u$, $v\in \mathcal H^{s+\Log}(\R^n)$. In particular, we show in Proposition \ref{prop:Hslog-Hilbert} below that the Gagliardo-type fractional--logarithmic Sobolev space $\mathcal H^{s+\Log}(\R^n)$ equipped with the inner product $\langle \cdot,\cdot\rangle_{\mathcal H^{s+\Log}}$ defines a Hilbert space. Next, we clarify the relationship between the spectral fractional--logarithmic Sobolev space and the Gagliardo-type fractional--logarithmic Sobolev space \(\mathcal H^{s+\Log}(\R^n)\).

\begin{proposition}\label{prop:Hslog-vs-form}
Let \(s\in(0,1)\). Then the embedding 
\begin{displaymath}
 H^{s+\Log}(\mathbb R^n)
\hookrightarrow
\mathcal H^{s+\Log}(\mathbb R^n)
\end{displaymath}
is  continuous. 
In particular, the inclusion is strict in the sense that there exists
(at least) a function \(\varphi\in \mathcal H^{s+\Log}(\mathbb R^n)\) that does not belong to
\(H^{s+\Log}(\mathbb R^n)\).
\end{proposition}

Here, the spectral fractional--logarithmic Sobolev space \(H^{2s+\Log}(\mathbb R^n)\) is the operator domain of the
fractional--logarithmic Laplacian \((-\Delta)^{s+\Log}\), whereas
\(\mathcal H^{s+\Log}(\mathbb R^n)\) is the domain of the associated
quadratic form \(\mathcal E_{+}\) defined in~\eqref{def:E} below. 
This is consistent with the classical operator-theoretic principle that
the operator domain of a nonnegative self-adjoint operator is continuously
embedded into its form domain. However, it is worth mentioning that the fractional--logarithmic Laplacian \((-\Delta)^{s+\Log}\) is
not necessarily a positive operator, but 
\((-\Delta)^{s+\Log}\) is bounded from below, and the same
domain inclusion remains valid for the associated closed quadratic form (cf.~\cite{MR839450}).\medskip

Next, we want to consider the fractional--logarithmic Laplacian 
\((-\Delta)^{s+\Log}\) on an open subset $\Omega$ of $\R^n$ equipped
with homogeneous Dirichlet boundary conditions in the nonlocal sense.

\begin{definition}
 Let $\Omega$ be an open subset of $\R^n$ and $0<s<1$. Then
    \[
	\mathcal H^{s+\Log}_{0}(\Omega)
	:=\Bigl\{
	u\in \mathcal H^{s+\Log}(\R^n)\,\Big|\,\ u=0\ \text{a.e. in }\mathbb R^n\setminus\Omega \Bigr\}.
	\]
\end{definition}

In the following, we always assume that $\Omega\subseteq\mathbb{R}^n$ is a bounded domain with continuous boundary $\partial\Omega$, and we identify $L^p(\Omega)$ with the linear subspace of functions $f\in L^p(\R^n)$ satisfying $f=0$ a.e. on $\R^n\setminus\Omega$.\medskip 
    
In this framework, we first establish a classical Poincar\'e inequality associated with the fractional--logarithmic Laplacian.

\begin{proposition}\label{prop:Poincare-Hslog-plus11}
   Let $0<s<1$ and $\Omega$ be a bounded domain in $\R^n$. Then, there exists a constant
		$C=C(n,s,\Omega)>0$ such that
		\begin{equation}\label{eq:Poincare-Hslog-plus11}
			\norm{u}_{L^2(\Omega)}\le  C [u]_{s+\Log,+}
		\end{equation}
     for every $u\in \mathcal H^{s+\Log}_0(\Omega)$. In particular, the Gagliardo-type semi-norm $[\cdot]_{s+\Log,+}$ defined in~\eqref{def:gagliardo-seminorm} is a norm on $\mathcal H^{s+\Log}_0(\Omega)$, which is equivalent to the standard norm on $\mathcal H^{s+\Log}_0(\Omega)$. 
	\end{proposition}

Clearly, $\mathcal H^{s+\Log}_{0}(\Omega)$ equipped with the norm $\norm{\cdot}_{\mathcal H^{s+\Log}}$ defined in~\eqref{eq:def-frac-loc-norm} is a closed subspace of $\mathcal H^{s+\Log}(\R^n)$. But under the assumption of Proposition~\ref{prop:Poincare-Hslog-plus11}, it is sufficient to equip $\mathcal H^{s+\Log}_{0}(\Omega)$ in the following with the semi-norm $[\cdot]_{s+\Log,+}$. For this, we employ the following notation.

\begin{notation}
    In the following, we set
	\[
	\|u\|_{\mathcal H^{s+\Log}_{0}(\Omega)}:=[u]_{s+\Log,+}.
	\]
\end{notation}    

Next, we want to compare the properties of the 
fractional--logarithmic space $\mathcal H^{s+\Log}_{0}(\Omega)$ and
the classic fractional Sobolev space $\mathcal H^{s}_{0}(\Omega)$. For completeness, we briefly recall the definition of the fractional Sobolev space
\[
\mathcal H^s(\mathbb R^n)
:=
\left\{
u\in L^2(\mathbb R^n)
\,:\,
[u]_s<\infty
\right\},
\]
where $[u]_s
:=
\langle u,u\rangle_s^{1/2},$
and
\[
\langle u,v\rangle_s
:=
\iint_{\mathbb R^n\times\mathbb R^n}
\frac{(u(x)-u(y))(v(x)-v(y))}
{|x-y|^{n+2s}}
\,dx\,dy.
\]
Equipped with the inner product
\[
\langle u,v\rangle_{\mathcal H^s(\mathbb R^n)}
:=
\langle u,v\rangle_{L^2(\mathbb R^n)}
+
\langle u,v\rangle_s,
\]
the space \(\mathcal H^s(\mathbb R^n)\) is a Hilbert space. Furthermore,
\[
\mathcal H_0^s(\Omega)
:=
\left\{
u\in\mathcal H^s(\mathbb R^n)
\,:\,
u=0\ \text{a.e. in }\mathbb R^n\setminus\Omega
\right\}
\]
is a closed linear subspace of \(\mathcal H^s(\mathbb R^n)\). If \(\Omega\) is a bounded domain, then the seminorm \([\cdot]_s\) defines a norm on \(\mathcal H_0^s(\Omega)\) that is equivalent to the norm inherited from \(\mathcal H^s(\mathbb R^n)\). Accordingly, for \(u\in\mathcal H_0^s(\Omega)\), we write $\|u\|_{\mathcal H_0^s(\Omega)}
:=
[u]_s.$

Further, the following \emph{Sobolev embedding} is well known (see, e.g. \cite[Lemma~1.28 \& Lemma~1.31]{Molica-Bisci-Radulescu-Servadei} or \cite{EGE}):
\begin{enumerate}
\item[(i)] (Sobolev embedding) Let $0<s<1$ such that $2s<n$ and set $2_s^*:=\frac{2n}{n-2s}$. Then, for every $q\in[1,2_s^*]$, one has the continuous embedding
\begin{equation}
 \label{eq:frac-Sobolev-embed}
\mathcal H^{s}_{0}(\Omega)\hookrightarrow L^q(\Omega).
\end{equation}

\item[(ii)] (Compact embedding) Let $0<s<1$ such that $2s<n$, and assume that $\Omega$ has continuous boundary $\partial\Omega$. Then, the embedding~\eqref{eq:frac-Sobolev-embed} is compact for every $q\in[1,2_s^*)$.
\end{enumerate}

In this context, there are two important properties of the Gagliardo-type fractional--logarithmic Sobolev space \(\mathcal H^{s+\Log}_{0}(\Omega)\) to mention: first, the space  \(\mathcal H^{s+\Log}_{0}(\Omega)\) is an \emph{intermediate space} between \(\mathcal H^{s+\varepsilon}_{0}(\Omega)\) and \(\mathcal H^{s}_{0}(\Omega)\) for every $\varepsilon\in(0,1-s)$, and secondly, the Sobolev-type embedding
\begin{equation}
    \label{eq:sobolev-embedd-frac-log}
    \mathcal H^{s+\Log}_{0}(\Omega)\hookrightarrow L^{q}(\Omega)
\end{equation}
remains compact for the critical exponent $q=2_s^*$. 

	\begin{proposition}\label{embedding}
    Let $s\in (0,1)$ such that $2s<n$, and $\Omega$ be a bounded domain in $\R^n$. Then, the following statements hold.
	\begin{enumerate}[label=\normalfont\textup{(\roman*)}]
			\item For  $\varepsilon\in(0,1-s)$, one has the continuous embeddings
			\begin{equation}
                \label{eq:embeddings-of-frac-log}
			\mathcal H^{s+\varepsilon}_{0}(\Omega)\hookrightarrow \mathcal H^{s+\Log}_{0}(\Omega)\hookrightarrow \mathcal H^{s}_{0}(\Omega).
			\end{equation}
			
			\item If, in addition, the boundary $\partial\Omega$ of $\Omega$ is continuous, then embedding~\eqref{eq:sobolev-embedd-frac-log} is compact for every $q\in[1,2_s^*]$.
		\end{enumerate}
	\end{proposition}
	
We emphasize that statement (ii) in Proposition~\ref{embedding}, the compact embedding 
$\mathcal H^{s+\Log}_{0}(\Omega)\hookrightarrow L^{2_{s}^*}(\Omega)$ is remarkably different to the classical Sobolev and fractional Sobolev space case (as mentioned above). The main reason for this is the stronger singular behavior near $x=0$ of the logarithmic factor in the positive part ${\bf k}_{s+\Log,+}$ of the
fractional--logarithmic kernel. We outline this fact in Lemma~\ref{lem:Es-controlled} (below), and thanks to this lemma, we can construct the compactness argument.

\begin{remark}
	{\normalfont On the other hand, the fact that the embedding~\eqref{eq:sobolev-embedd-frac-log} is compact for every $q\in[1,2_s^*]$ 
	indicates that \(L^{2_s^\ast}(\Omega)\) may not be the optimal critical target
	space for the logarithmically strengthened fractional energy space $\mathcal H^{s+\Log}_{0}(\Omega)$. It would be
	natural and important to investigate the optimal critical target space as,
	for instance, in the Orlicz space framework. This is in line with related studies on
	Sobolev-type embeddings for spaces with logarithmic smoothness; see, for
	example, Opic and Trebels~\cite{OpicTrebels2000}. It is also analogous to
	recent developments for the logarithmic Laplacian, where critical embeddings
	for the associated energy spaces are studied in terms of Orlicz spaces; see
	Arora, Giacomoni, and Vaishnavi~\cite{AroraGiacomoniVaishnavi2025}.}
\end{remark}

 Next, we introduce the bilinear form associated with the fractional--logarithmic operator \((-\Delta)^{s+\Log}\) on $\Omega$ equipped with nonlocal homogeneous Dirichlet boundary conditions.

\begin{definition}
Let \(0<s<1\), and let \(\Omega\subset\mathbb R^n\) be a bounded domain. We define the bilinear forms $\mathcal E_{\pm}:
C_c^\infty(\Omega)\times C_c^\infty(\Omega)
\longrightarrow \mathbb R$
by
\begin{equation}\label{def:E}
\mathcal E_{\pm}(u,w)
:=
\iint_{\mathbb R^n\times\mathbb R^n}
(u(x)-u(y))(w(x)-w(y))
\,{\bf k}_{s+\Log,\pm}(x-y)
\,dx\,dy
\end{equation}
for every \(u,w\in C_c^\infty(\Omega)\). We further define $\mathcal E_s:
C_c^\infty(\Omega)\times C_c^\infty(\Omega)
\longrightarrow \mathbb R$
by
\begin{equation}\label{def:Es}
\mathcal E_s(u,w)
:=
\iint_{\mathbb R^n\times\mathbb R^n}
\frac{(u(x)-u(y))(w(x)-w(y))}
{|x-y|^{n+2s}}
\,dx\,dy
\end{equation}
for every \(u,w\in C_c^\infty(\Omega)\). We refer to \(\mathcal E_s\) as the bilinear form associated with the fractional Laplacian.

The bilinear form associated with the fractional--logarithmic operator
\((-\Delta)^{s+\Log}\) is then defined by
\begin{equation}\label{enerln}
\mathcal E_{s+\Log}(u,w)
:=
\mathcal E_{+}(u,w)
-
\mathcal E_{-}(u,w)
+
\frac{b_{n,s}c_{n,s}}{2}\,
\mathcal E_s(u,w)
\end{equation}
for every \(u,w\in C_c^\infty(\Omega)\).
\end{definition}

 Clearly, one has that
  \begin{displaymath}
	\|u\|_{\mathcal H^{s+\Log}_{0}(\Omega)}
	=[u]_{s+\Log,+}
	=\sqrt{\mathcal E_{+}(u,u)}\qquad
    \text{for every $u\in \mathcal H^{s+\Log}_{0}(\Omega)$.}
  \end{displaymath}

The following proposition provides a refined analysis of the bilinear forms $\mathcal E_{\pm}$, $\mathcal E_{s}$, and $\mathcal E_{s+\Log}$ introduced above. In the following, we write $\operatorname{diam}(\Omega)$ to denote the \emph{diameter} 
$\operatorname{diam}(\Omega):=\sup_{x\, y\in \Omega}|x-y|$ of $\Omega$.

	\begin{proposition}\label{embedding-1}
    Let $0<s<1$ and $\Omega$ be a bounded domain. Then, the following statements are true.
		\begin{enumerate}[label=\normalfont\textup{(\roman*)}]
			\item  The negative part satisfies the estimate
			\begin{equation}
            \label{eq:estimate-of-E-}
			0\le \mathcal E_{-}(u,u)\le \frac{1}{s^2}\,c_{n,s}\,|\mathbb S^{n-1}|\,
			\|u\|_{L^2(\Omega)}^2
			\end{equation}
            for every $u\in L^2(\Omega)$, 
            	where the constant $c_{n,s}$ is given by~\eqref{norm const1}. 
                
            \item The bilinear form $\mathcal E_{s+\Log} : \mathcal H^{s+\Log}_{0}(\Omega) \times \mathcal H^{s+\Log}_{0}(\Omega)\to \R$ given by~\eqref{enerln} is well-defined and continuous.
			
			\item If $\Omega$ has diameter $ \operatorname{diam}(\Omega)<1$, then 
		\begin{equation}\label{eq:Eplus-minus-lower}
            \begin{split}
               &\mathcal E_{+}(u,u)-\mathcal E_{-}(u,u)\\
			    &\qquad\ge
				\frac{1}{s}\,c_{n,s}\,|\mathbb S^{n-1}|\,\operatorname{diam}(\Omega)^{-2s}
				\Bigl(\ln \frac{1}{\operatorname{diam}(\Omega)}-\frac{1}{2s}\Bigr)\,
				\|u\|_{L^2(\Omega)}^{\,2} 
            \end{split}
			\end{equation}
            for every $u\in \mathcal H^{s+\Log}_0(\Omega)$, where $c_{n,s}$ denotes the universal constant defined in~\eqref{norm const1}. 
			In particular, if $\operatorname{diam}(\Omega)\le e^{-1/(2s)}$, then
			\[ 
			\mathcal E_{+}(u,u)-\mathcal E_{-}(u,u)\ge 0.
			\]
			\item Assume that $\operatorname{diam}(\Omega)<e^{-1/(2s)}$ and the constant $b_{n,s}$ given by~\eqref{eq:const-bns} satisfies $b_{n,s}\ge0$. Then the bilinear form $\mathcal E_{s+\Log}$ defined by~\eqref{enerln} is positive definite. Moreover, for every $u\in \mathcal H^{s+\Log}_0(\Omega)$, one has that $|u|\in \mathcal H^{s+\Log}_0(\Omega)$ and
			\begin{equation}\label{eq:modulus-invariance}
				\mathcal E_{s+\Log}(|u|,|u|)\le \mathcal E_{s+\Log}(u,u).
			\end{equation}
			Here, \eqref{eq:modulus-invariance} holds with equality if and only if $u$ does not change sign.
		\end{enumerate}
	\end{proposition}

	Next, we intend to study the existence, uniqueness, and regularity properties of weak solutions $u$ to the nonlocal Poisson problem
	\begin{equation}\label{eq:poisson}
		 \left\{
\begin{array}{lll}
			(-\Delta)^{s+\Log}u + V(x)\,u = f \quad &\text{in }\Omega,\\[2mm]
 \phantom{ ---------   }
			u = 0 \quad &\text{in }\Omega^c,
		\end{array}\right.
	\end{equation}
	driven by the fractional--logarithmic operator $(-\Delta)^{s+\Log}$, where $V\in L^{\infty}(\Omega)$. For this, we introduce the following notion of weak solutions to the boundary-value problem~\eqref{eq:poisson}. Here, we write $\langle\cdot,\cdot\rangle$ for the duality pairing between $\mathcal H^{s+\Log}_{0}(\Omega)$ and its dual space $\big(\mathcal H^{s+\Log}_{0}(\Omega)\big)^{\!*}$.

    	\begin{definition}
		\label{def:weak-solution}
		For given $f\in \big(\mathcal H^{s+\Log}_{0}(\Omega)\big)^{\!*}$ and
		$V\in L^\infty(\Omega)$, a function
		$u\in \mathcal H^{s+\Log}_{0}(\Omega)$ is called a \emph{weak solution} of
		\eqref{eq:poisson} if
		\begin{equation}\label{eq:weak-formulation}
			\mathcal E_{s+\Log}(u,\varphi)+\int_{\Omega} V(x)\,u(x)\,\varphi(x)\,dx
			=\langle f,\varphi\rangle
		\end{equation}
        for every $\varphi\in \mathcal H^{s+\Log}_{0}(\Omega)$.
	\end{definition}

We require the following notation.

\begin{notation}
   For $r\in(0,e^{-|b_{n,s}|/2})$, we set
		\begin{equation}\label{eq:alpha_r_def}
			\alpha_r:=1-\frac{|b_{n,s}|}{2\,\ln(1/r)}>0.
		\end{equation}
    Further, let \(\operatorname{essinf}_{\Omega} V\) be the essential infimum of \(V\) on \(\Omega\), i.e.
\[
\operatorname{essinf}_{\Omega} V=\sup\big\{a\in\mathbb R:\ V(x)\ge a \ \text{for a.e. }x\in\Omega\big\}.
\]
\end{notation}

We now apply the Lax--Milgram theorem to establish existence and uniqueness of weak solutions to the Poisson problem, and then use a Moser iteration scheme to derive \(L^\infty\)-regularity.

\begin{theorem}\label{thm:LM-general-b-refined}
Let $0<s<1$, let $\Omega\subset\R^n$ be a bounded domain, and let
$V\in L^\infty(\Omega)$ and $f\in(\mathcal H^{s+\Log}_0(\Omega))^*$.
Consider the Poisson problem \eqref{eq:poisson}.

\begin{enumerate}[label=\normalfont\textup{(\roman*)}]
\item Assume that, for some $r\in(0,e^{-|b_{n,s}|/2})$,
\begin{equation}\label{eq:w-cond-refined}
\operatorname*{ess\,inf}_{\Omega} V
\, \ge\,
\frac{c_{n,s}|\mathbb S^{n-1}|}{s^2}
+\frac{|b_{n,s}|\,c_{n,s}}{s}\,|\mathbb S^{n-1}|\,r^{-2s}.
\end{equation}
Then, the Poisson problem~\eqref{eq:poisson} admits a unique weak solution $u\in \mathcal H^{s+\Log}_0(\Omega).$
Moreover, the solution satisfies
\[
\|u\|_{\mathcal H^{s+\Log}_0(\Omega)}
\le \frac{1}{\alpha_r}\,\|f\|_{(\mathcal H^{s+\Log}_0(\Omega))^*},
\]
where \(c_{n,s}, b_{n,s}, \alpha_r\) are defined in \eqref{norm const1}, \eqref{eq:const-bns} and \eqref{eq:alpha_r_def} respectively.

\item Let \(u\) be a weak solution of \eqref{eq:poisson}. If $f\in L^{q}(\Omega)$ for some $q>\frac{n}{2s},$ then $u\in L^\infty(\Omega).$
\end{enumerate}
\end{theorem}

\begin{remark}
{\normalfont Evidently, the condition~\eqref{eq:w-cond-refined} is only a sufficient condition implying well-posedness of the Poisson problem~\eqref{eq:poisson}. Under the additional assumption that the boundary $\partial\Omega$ of $\Omega$ is continuous, the compact embedding~\eqref{eq:sobolev-embedd-frac-log} yields that the nonlocal Poisson problem \eqref{eq:poisson} may also be studied by using the Fredholm alternative. Indeed, let \(A_V\) denote the self-adjoint realization in \(L^2(\Omega)\) of
\[
(-\Delta)^{s+\Log}+V
\]
with homogeneous Dirichlet exterior condition. Since \(A_V\) has compact
resolvent, the problem admits a unique weak solution for every
\(f\in L^2(\Omega)\) whenever $0\notin\sigma(A_V).$
If \(0\in\sigma(A_V)\), then solvability requires the compatibility
condition
\begin{equation}
 \label{eq:fredholm}
\langle f,\phi\rangle_{L^2(\Omega)}=0
\qquad
\text{for every }\phi\in\ker A_V,
\end{equation}
and uniqueness holds only modulo \(\ker A_V\). Now, for given potential $V$ it remains to verify the condition~\eqref{eq:fredholm}.}
\end{remark}

	Our next aim is to study the eigenvalue problem associated with the fractional-logarithmic Laplacian $(-\Delta)^{s+\Log}$ equipped with homogeneous Dirichlet boundary conditions
	\begin{equation}\label{eq:eigen}
		\left\{
		\begin{aligned}
			(-\Delta)^{s+\Log} u &= \lambda\,u \quad &&\text{in }\Omega,\\
			u&=0 \quad &&\text{in }\R^n\setminus\Omega,
		\end{aligned}
		\right.
	\end{equation}
	where $\Omega$ is a bounded Lipschitz domain in $\R^n$.

    \begin{definition}
    For given $\lambda \in \R$, we call a nonzero function $u\in \mathcal H^{s+\Log}_{0}(\Omega)$ an \emph{eigenfunction}
	of~\eqref{eq:eigen} associated with the eigenvalue $\lambda$ provided 
	\[
	\mathcal E_{s+\Log}(u,\phi)
	=\lambda\int_{\Omega}u\,\phi\,dx
	\qquad\text{for all }\phi\in \mathcal H^{s+\Log}_{0}(\Omega).
	\]    
    \end{definition}
	
Our next result provides the existence of eigenvalues associated with the fractional-logarithmic Laplacian $(-\Delta)^{s+\Log}$ equipped with homogeneous Dirichlet boundary conditions.
    
\begin{theorem}\label{thm:eigen-spectrum}
Let $\Omega$ be a bounded Lipschitz domain in $\R^n$. Then
problem~\eqref{eq:eigen} admits a nondecreasing sequence of real eigenvalues
\[
\lambda_1^{s+\Log}(\Omega)\le \lambda_2^{s+\Log}(\Omega)\le \cdots \le \lambda_k^{s+\Log}(\Omega)\le \cdots,
\]
with corresponding eigenfunctions $\{\xi_k\}_{k\in\N}\subset \mathcal H^{s+\Log}_{0}(\Omega)$ such that:
\begin{enumerate}[label=\normalfont\textup{(\roman*)}]
\item For every $k\in\N$,
\[
\lambda_{k}^{s+\Log}(\Omega)
=\min\Bigl\{\mathcal E_{s+\Log}(u,u)\ :\ u\in \mathcal H_k(\Omega),\ \|u\|_{L^2(\Omega)}=1\Bigr\},
\]
where $\mathcal H_1(\Omega):=\mathcal H^{s+\Log}_0(\Omega)$ and, for $k\ge2$,
\[
\mathcal H_k(\Omega)
:=\Bigl\{u\in\mathcal H^{s+\Log}_0(\Omega)\ :\ \int_\Omega u\,\xi_i\,dx=0
\ \text{for }i=1,\dots,k-1\Bigr\}.
\]

\item  One has $\lambda_k^{s+\Log}(\Omega)\to+\infty$ as $k\to\infty$.

\item  The family $\{\xi_k\}_{k\in\N}$ forms an orthonormal basis of $L^2(\Omega)$.

\item Assume that $\operatorname{diam}(\Omega)<e^{-1/(2s)}$ and $b_{n,s}\ge0$. Then $\lambda_1^{s+\Log}(\Omega)>0,$
and the first eigenfunction $\xi_1$ can be chosen nonnegative.
\end{enumerate}
\end{theorem}

We next record a pointwise maximum principle for the fractional-logarithmic
Laplacian. Unlike the fractional Laplacian, whose kernel is nonnegative, the
kernel of \((-\Delta)^{s+\Log}\) may change sign. Therefore a maximum principle
is not automatic and requires additional restrictions. The following result
shows that, under a smallness assumption on the diameter of the domain, one can
recover a weak and strong maximum principle in the pointwise sense. This is motivated by the corresponding discussion for the
logarithmic Laplacian, see, for instance,~\cite[Proposition 4.1]{CW18}.

\begin{proposition}\label{prop:pointwise-MP}
	Let \(0<s<1\), and let \(\Omega\subset\mathbb R^n\) be a bounded domain.
	Assume that
	\[
	\operatorname{diam}(\Omega)
	<
	\exp\left(\frac{b_{n,s}-1/s}{2}\right).
	\]
	Let \(u\in C(\mathbb R^n)\) be such that \(u=0\) in
	\(\mathbb R^n\setminus\Omega\). Assume moreover that
	\((-\Delta)^{s+\Log}u\) is pointwise well-defined in \(\Omega\), and that
	\[
	(-\Delta)^{s+\Log}u\ge0
	\qquad\text{pointwise in}\Omega.
	\]
	Then $	u\ge0$ in $\Omega$.
	Moreover, if \(u\not\equiv0\) in \(\Omega\), then $u>0$ in $\Omega$.
\end{proposition}

	The above proposition gives only a basic pointwise version of the maximum
	principle under a smallness assumption on the domain. More general maximum
	principles may also be considered in the weak sense, see, for instance,~\cite{JarohsWeth2019}. For the fractional-logarithmic Laplacian, such a weak formulation is more
	delicate because the kernel of the operator may change sign. Nevertheless, once
	a strong maximum principle is established in the weak sense, it can be applied
	to the first eigenfunction in Theorem~\ref{thm:eigen-spectrum} \textup{(iv)}.
	Indeed, after proving the required regularity, the nonnegative first eigenfunction \(\xi_1\) can be shown to be
	strictly positive in \(\Omega\). This issue, together with related regularity
	results for the fractional-logarithmic Laplacian, will be treated in a
	forthcoming work.

\smallskip

	Finally, we derive the Weyl-type asymptotic law for the eigenvalues of the fractional--logarithmic operator.

	\begin{theorem}\label{thm:weyl-fraclog}
		Let $\Omega$ be a bounded Lipschitz domain in $\R^n$, and let
		$\{\lambda_k^{s+\Log}(\Omega)\}_{k\in\N}$ be the eigenvalues of~\eqref{eq:eigen}.
		For $\Lambda>0$ define the counting function
		\[
		\mathcal N(\Lambda):=\#\bigl\{k\in\N:\ \lambda_k^{s+\Log}(\Omega)\le \Lambda\bigr\}.
		\]
		Then
		\[
		\lim_{\Lambda\to+\infty}\,
		\mathcal N(\Lambda)\,\Lambda^{-\frac{n}{2s}}\bigl(\ln\Lambda\bigr)^{\frac{n}{2s}}
		=(2\pi)^{-n}\,s^{\frac{n}{2s}}\,\omega_n\,|\Omega|,
		\]
		and
		\[
		\lim_{k\to+\infty}
		\frac{\lambda_k^{s+\Log}(\Omega)\,k^{-\frac{2s}{n}}}{\ln k}
		=\frac{2}{n}\,(2\pi)^{2s}\bigl(\omega_n|\Omega|\bigr)^{-\frac{2s}{n}},
		\]
	where \(\omega_n=|B_1(0)|\) denotes the volume of the unit ball in
	\(\mathbb R^n\).
	\end{theorem}
	
	\begin{remark}
    {\normalfont For the fractional Laplacian $(-\Delta)^s$ on a bounded Lipschitz domain
		$\Omega\subset\R^n$, the eigenvalues $\{\lambda_{k}^s(\Omega)\}_{k\in\N}$
		satisfy the Weyl asymptotics
		\begin{equation}\label{eq:frac-Weyl}
			\lim_{k\to\infty}\lambda_{k}^s(\Omega)\,k^{-\frac{2s}{n}}
			=(2\pi)^{2s}\bigl(\omega_n|\Omega|\bigr)^{-\frac{2s}{n}}.
		\end{equation}
		For the logarithmic Laplacian $ (-\Delta)^{\Log}$, the eigenvalues
		$\{\lambda^{\ln}_k(\Omega)\}_{k\in\N}$ satisfy
		\begin{equation}\label{eq:log-Weyl}
			\lim_{k\to\infty}\frac{\lambda^{\ln}_k(\Omega)}{\ln k}=\frac{2}{n}.
		\end{equation}
		Consequently, the operator $(-\Delta)^{s+\Log}$ interpolates between the
		fractional and logarithmic regimes, and the asymptotic law in
		Theorem~\ref{thm:weyl-fraclog} is precisely the product of the fractional
		Weyl constant in~\eqref{eq:frac-Weyl} and the logarithmic growth factor
		in~\eqref{eq:log-Weyl}.}
	\end{remark}
	
	The remainder of this paper is organized as follows. In Section~2, we present several equivalent formulations of the fractional-logarithmic operator and establish their equivalence, including the singular-integral representation, the Fourier-multiplier representation, the spectral-calculus definition, and the extension characterization.
	 In Section~3, we first establish key properties of the underlying function spaces and energy forms, and then prove existence and regularity results for the Poisson problem. In Section~4, we study the Dirichlet eigenvalue problem, including qualitative properties of eigenvalues and the Weyl-type asymptotic law.

	\setcounter{equation}{0}
	\section{Representation Formulas and Basic Properties}

    In this section, we provide proofs of the various characterizations of the fractional-logarithmic Laplacian, including the singular-integral representation, the Fourier-multiplier representation, the spectral-calculus definition, and the extension characterization.

	\subsection{Equivalent Representations}

We begin with the proof of Proposition~\ref{prop:dt-fraclap}, which establishes the existence of the derivative of the fractional Laplacian with respect to the order at a general \(s\in(0,1)\).

\begin{proof}[\bf Proof of Proposition \ref{prop:dt-fraclap}.]
Fix $x\in\R^n$ and write
\[
I(t):=\PV\int_{\R^n}\frac{u(x)-u(y)}{|x-y|^{n+2t}}\,dy,
\qquad\text{so that}\qquad
(-\Delta)^t u(x)=c_{n,t}I(t).
\]

Since $u\in C_c^2(\R^n)$, we have the Taylor expansion
\[
u(y)=u(x)+\nabla u(x)\cdot (y-x)+O(|x-y|^2)\quad \text{as}\qquad y\to x
\]
In the principal value sense, the contribution of the linear term vanishes:
\[
\int_{\varepsilon<|y-x|<1}\frac{\nabla u(x)\cdot (x-y)}{|x-y|^{n+2t}}\,dy=0,\quad \text{for every}\quad 0<\varepsilon<1.
\]
Therefore, near $y=x$ the integrand behaves like $\frac{O(|x-y|^2)}{|x-y|^{n+2t}},$
which is locally integrable for all $t\in(0,1)$.
This shows that $(-\Delta)^t u(x)$ is well-defined.

Fix \(s\in(0,1)\) and choose \(\delta>0\) such that $0<s-\delta<s+\delta<1.$
For \(\varepsilon>0\), set
\[
I_\varepsilon(t)
:=
\int_{|x-y|>\varepsilon}
\frac{u(x)-u(y)}{|x-y|^{n+2t}}\,dy,
\qquad t\in(s-\delta,s+\delta),
\]
then the definition of the Cauchy
principal value gives $I(t)
=
\lim_{\varepsilon\downarrow0}I_\varepsilon(t).$

For each fixed \(\varepsilon>0\), the singularity at \(y=x\) is removed. Moreover,
for \(t\in(s-\delta,s+\delta)\),
\[
\partial_t\left(
\frac{u(x)-u(y)}{|x-y|^{n+2t}}
\right)
=
\frac{u(x)-u(y)}{|x-y|^{n+2t}}
\bigl(-2\ln|x-y|\bigr).
\]
Since \(u\) is bounded and compactly supported, and since \(s-\delta>0\), both
the integrand and its \(t\)-derivative are dominated on
\(\{|x-y|>\varepsilon\}\) by integrable functions independent of
\(t\in(s-\delta,s+\delta)\). Hence, by the dominated convergence theorem,
\[
I_\varepsilon'(t)
=
\int_{|x-y|>\varepsilon}
\frac{u(x)-u(y)}{|x-y|^{n+2t}}
\bigl(-2\ln|x-y|\bigr)\,dy.
\]

We claim that \(I_\varepsilon'(t)\) converges locally uniformly in
\(t\in(s-\delta,s+\delta)\) as \(\varepsilon\to0^+\). To see this, we split the
domain into \(B_1(x)\) and \(\mathbb R^n\setminus B_1(x)\). On \(B_1(x)\), by
Taylor's formula,
\[
u(x)-u(y)
=
-\nabla u(x)\cdot(y-x)+O(|x-y|^2).
\]
The principal value contribution of the linear term vanishes by symmetry.
Hence the singular part is controlled by $C|x-y|^{-n-2(s+\delta)+2}|\ln|x-y||,$
which is integrable near \(x\) because \(s+\delta<1\). On
\(\mathbb R^n\setminus B_1(x)\), since \(u\) is bounded and compactly
supported, the derivative integrand is controlled by $C|x-y|^{-n-2(s-\delta)}\ln(e+|x-y|),$
which is integrable at infinity because \(s-\delta>0\). Therefore the standard theorem on differentiation of a uniform limit now shows
\[
I'(s)
=
\PV\int_{\mathbb R^n}
\frac{u(x)-u(y)}{|x-y|^{n+2s}}
\bigl(-2\ln|x-y|\bigr)\,dy.
\]
Finally,
\begin{align*}
    \left.\frac{d}{dt}(-\Delta)^t u(x)\right|_{t=s}
&=c_{n,s}'I(s)+c_{n,s}I'(s)\\
&=\frac{c_{n,s}'}{c_{n,s}}(-\Delta)^s u(x)
-2c_{n,s}\PV\int_{\R^n}\frac{(u(x)-u(y))\ln|x-y|}{|x-y|^{n+2s}}\,dy,
\end{align*}
which is \eqref{eq:dt-fraclap}. Note that
\[
\ln c_{n,s}=2s\ln 2-\frac n2\ln\pi+\ln s+\ln\Gamma\!\left(\frac{n+2s}{2}\right)-\ln\Gamma(1-s).
\]
Differentiating with respect to $s$ and using $\psi=\Gamma'/\Gamma$, we obtain
\[
b_{n,s}=\frac{c_{n,s}'}{c_{n,s}}
=\frac{d}{ds}\ln c_{n,s}
=2\ln 2+\frac{1}{s}+\psi\!\left(\frac{n+2s}{2}\right)+\psi(1-s),
\]
which completes the proof.
\end{proof}

We now present proofs of further representations of the fractional-logarithmic Laplacian $(-\Delta)^{s+\Log}$.

\begin{proof}[\bf Proof of Theorem \ref{pr 1.1}.]
    Parts (i) and (ii) follow directly from Proposition~\ref{prop:dt-fraclap}.

(iii) It follows by \cite[Section~3]{EGE} that 
	$$ c_{n,s}=\Big(\int_{\R^n} \frac{1-\cos\zeta_1}{|\zeta|^{n+2s}} d\zeta\Big)^{-1}.  $$
	For $s\in(0,1)$, Taylor's formula and translation invariance give
\[
\int_{\mathbb R^n}|u(x+y)+u(x-y)-2u(x)|\,dx
\le \min\bigl\{|y|^2\|D^2u\|_{L^1(\mathbb R^n)},
4\|u\|_{L^1(\mathbb R^n)}\bigr\}.
\]
Hence
\[
\iint_{\mathbb R^n\times\mathbb R^n}
\frac{|u(x+y)+u(x-y)-2u(x)|}{|y|^{n+2s}}
|\ln|y||\,dx\,dy<\infty.
\]
	Consequently, by the Fubini theorem, we can exchange the integral in $y$ with the Fourier
	transform in $x$. Thus, we apply the Fourier transform in the variable $x$  and we obtain  
	\begin{align*}
		\cF\big(c_{n,s} \cL_1 u\big)(\xi)
		&=- \frac{c_{n,s}}2  \int_{\R^n} \frac{e^{i\xi\cdot y}+e^{-i\xi\cdot y}-2}{|y|^{n+2s}} \big(\!-2\log|y|\big) dy\,  \hat{u}(\xi)
		\\&=  c_{n,s} \int_{\R^n} \frac{1-\cos(\xi\cdot y)}{|y|^{n+2s}} \big(\!-2\log|y|\big) dy\,  \hat{u}(\xi),
	\end{align*}
	where 
	{\small \begin{align*}
		&\int_{\R^n} \frac{1-\cos(\xi\cdot y)}{|y|^{n+2s}} \big(\!\!-2\log|y|\big) dy\\=&\int_{\R^n} \frac{1-\cos(|\xi| y_1)}{|y|^{n+2s}} \big(\!-2\log|y|\big) dy
		\\=&|\xi|^{2s} \int_{\R^n} \frac{1-\cos z_1}{|z|^{n+2s}} \big(2\log|\xi|   -2\log|z|\big) dz
		\\=&|\xi|^{2s} \big(2\log|\xi| \big) \int_{\R^n} \frac{1-\cos z_1}{|z|^{n+2s}} dz 
		+
		|\xi|^{2s}\int_{\R^n} \frac{1-\cos z_1}{|z|^{n+2s}} \big( \!-2\log|z|\big)dz
		\\=&c_{n,s}^{-1}|\xi|^{2s} \big( 2\log|\xi| \big)+ |\xi|^{2s} \int_{\R^n} \frac{1-\cos z_1}{|z|^{n+2s}} \big( \!-2\log|z|\big)dz. 
	\end{align*}}
	Thus, we obtain that 
	\[\cF\big(c_{n,s} \cL_1 u  +\widehat b_{n,s}(-\Delta)^s u\big) (\xi) =|\xi|^{2s} \big(2\log|\xi| \big)  \hat{u}(\xi),\]
	where 
	$$\widehat b_{n,\s}=c_{n,s} \int_{\R^n} \frac{1-\cos\zeta_1}{|\zeta|^{n+2s}}\big(2\log(|\zeta|)\big) d\zeta. $$
    It is sufficient to prove that $\widehat b_{n,\s}=b_{n,s}.$

Split $\R^n=B_1(0)\cup(\R^n\setminus B_1(0))$. As $|\zeta|\to0$, $1-\cos\zeta_1\sim \frac{\zeta_1^2}{2}\lesssim |\zeta|^2,$
hence
\[
\frac{1-\cos\zeta_1}{|\zeta|^{n+2s}}\;\big|\ln|\zeta|\big|\;\lesssim\;|\zeta|^{-n-2s+2}\,\big|\ln|\zeta|\big|,
\]
which is integrable on $B_1(0)$ for $s\in(0,1)$ since in polar coordinates it becomes
$\int_0^1 r^{1-2s}|\ln r|\,dr<\infty$.
As $|\zeta|\to\infty$, we have $0\le 1-\cos\zeta_1\le 2$, so
\[
\frac{1-\cos\zeta_1}{|\zeta|^{n+2s}}\,|\ln|\zeta||\;\lesssim\;|\zeta|^{-n-2s}\ln|\zeta|,
\]
which is integrable on $\R^n\setminus B_1(0)$ for any $s>0$. Therefore, by dominated
convergence we may differentiate under the integral sign and obtain
\[
\frac{d}{ds} c_{n,s}^{-1}=\int_{\R^n}(1-\cos\zeta_1)\,\frac{d}{ds}\big(|\zeta|^{-n-2s}\big)\,d\zeta
=\int_{\R^n}\frac{1-\cos\zeta_1}{|\zeta|^{n+2s}}\bigl(-2\ln|\zeta|\bigr)\,d\zeta=-\frac{\widehat b_{n,s}}{c_{n,s}}.
\]
Hence $\widehat b_{n,s}=\frac{c_{n,s}'}{c_{n,s}}=b_{n,s}$.

\smallskip

(iv) Recall that, for any of the even symbols $m(\xi)$ considered here, the corresponding operator $m(D)$ on
$\mathcal Z'(\R^n)$ is defined by duality via
\begin{equation}\label{eq:def-mD-Zprime}
\langle m(D)u,\varphi\rangle_{\mathcal Z',\mathcal Z}
:=\Big\langle u,\ \mathcal F^{-1}\!\big(m(\xi)\widehat\varphi(\xi)\big)\Big\rangle_{\mathcal Z',\mathcal Z},
\quad u\in\mathcal Z'(\R^n),\ \varphi\in\mathcal Z(\R^n).
\end{equation}
In particular,
\[
(-\Delta)^s \equiv m_s(D)\ \text{with}\ m_s(\xi)=|\xi|^{2s},\qquad
 (-\Delta)^{\Log}\equiv \ell(D)\ \text{with}\ \ell(\xi)=\ln|\xi|^2,
\]
and
\[
(-\Delta)^{s+\Log}\equiv m_{s+\Log}(D)\ \text{with}\ m_{s+\Log}(\xi)=|\xi|^{2s}\ln|\xi|^2.
\]

Fix $u\in\mathcal Z'(\R^n)$ and $\varphi\in\mathcal Z(\R^n)$. Using \eqref{eq:def-mD-Zprime} twice and the fact
that multiplication of symbols is commutative, we compute
\begin{align*}
\big\langle (-\Delta)^s( (-\Delta)^{\Log}u),\varphi\big\rangle_{\mathcal Z',\mathcal Z}
&=\big\langle  (-\Delta)^{\Log}u,\ \mathcal F^{-1}\!\big(|\xi|^{2s}\widehat\varphi(\xi)\big)\big\rangle_{\mathcal Z',\mathcal Z}\\
&=\Big\langle u,\ \mathcal F^{-1}\!\Big(\ln|\xi|^2\cdot |\xi|^{2s}\widehat\varphi(\xi)\Big)\Big\rangle_{\mathcal Z',\mathcal Z}\\
&=\big\langle (-\Delta)^{s+\Log}u,\varphi\big\rangle_{\mathcal Z',\mathcal Z}.
\end{align*}
An analogous computation yields
{\footnotesize
\[
\big\langle  (-\Delta)^{\Log}((-\Delta)^s u),\varphi\big\rangle_{\mathcal Z',\mathcal Z}
=\Big\langle u,\ \mathcal F^{-1}\!\Big(\ln|\xi|^2\cdot |\xi|^{2s}\widehat\varphi(\xi)\Big)\Big\rangle_{\mathcal Z',\mathcal Z}
=\big\langle (-\Delta)^{s+\Log}u,\varphi\big\rangle_{\mathcal Z',\mathcal Z}.
\]
}
Since $\varphi\in\mathcal Z(\R^n)$ is arbitrary, the identities hold in $\mathcal Z'(\R^n)$.

\smallskip

(v) Let \(u\in C_c^2(\mathbb R^n)\). There exists a constant \(C_0>0\)
such that
\begin{equation}\label{eq:u-global-lipschitz}
|u(x)-u(y)|
\le C_0\min\{|x-y|,1\},
\qquad x,y\in\mathbb R^n.
\end{equation}
We restrict ourselves to \(s\in(0,\frac14)\). Define
\[
d_{n,s}
:=
2^{2s}\pi^{-\frac n2}
\frac{\Gamma\left(\frac{n+2s}{2}\right)}
{\Gamma(1-s)}.
\]
Then
\[
c_{n,s}=s\,d_{n,s},
\qquad
d_{n,s}\longrightarrow
c_n:=\pi^{-\frac n2}\Gamma\left(\frac n2\right)
=\frac{2}{|\mathbb S^{n-1}|}
\quad\text{as }s\to0^+.
\]
Moreover,
\[
b_{n,s}
=
\frac1s+\ln4+\psi(1-s)
+\psi\left(\frac{n+2s}{2}\right).
\]
If we set
\[
q_{n,s}
:=
\ln4+\psi(1-s)
+\psi\left(\frac{n+2s}{2}\right),
\]
then
\[
b_{n,s}=\frac1s+q_{n,s},
\qquad
q_{n,s}\longrightarrow
\rho_n:=\ln4+\psi(1)+\psi\left(\frac n2\right).
\]
In particular,
\begin{equation}\label{eq:csbs-limit}
c_{n,s}b_{n,s}
=
d_{n,s}\bigl(1+s q_{n,s}\bigr)
\longrightarrow c_n
\qquad\text{as }s\to0^+.
\end{equation}

Recall that
\[
(-\Delta)^{s+\Log}u
=
c_{n,s}\mathcal L_1u
+
b_{n,s}(-\Delta)^su.
\]
Splitting the integrals into \(B_1(x)\) and
\(\mathbb R^n\setminus B_1(x)\), we obtain
\begin{align}
(-\Delta)^{s+\Log}u(x)
={}&
c_{n,s}
\int_{B_1(x)}
\frac{u(x)-u(y)}
{|x-y|^{n+2s}}
\bigl(-2\ln|x-y|\bigr)\,dy
\label{eq:decomp-log-limit}\\
&+
c_{n,s}b_{n,s}
\int_{B_1(x)}
\frac{u(x)-u(y)}
{|x-y|^{n+2s}}\,dy+
\alpha_{n,s}u(x)
\nonumber\\
&-
c_{n,s}
\int_{\mathbb R^n\setminus B_1(x)}
\frac{u(y)}
{|x-y|^{n+2s}}
\bigl(-2\ln|x-y|\bigr)\,dy
\nonumber\\
&-
c_{n,s}b_{n,s}
\int_{\mathbb R^n\setminus B_1(x)}
\frac{u(y)}
{|x-y|^{n+2s}}\,dy,
\nonumber
\end{align}
where
\begin{align}
\alpha_{n,s}
:={}&
c_{n,s}
\int_{\mathbb R^n\setminus B_1(x)}
\frac{-2\ln|x-y|}
{|x-y|^{n+2s}}\,dy+
c_{n,s}b_{n,s}
\int_{\mathbb R^n\setminus B_1(x)}
\frac{1}{|x-y|^{n+2s}}\,dy.
\label{eq:alpha-ns}
\end{align}

We estimate the five terms in
\eqref{eq:decomp-log-limit} separately.

\medskip
\noindent
\textit{Step 1: the local logarithmic term.}
By \eqref{eq:u-global-lipschitz},
\begin{align*}
&\sup_{x\in\mathbb R^n}
\left|
c_{n,s}
\int_{B_1(x)}
\frac{u(x)-u(y)}
{|x-y|^{n+2s}}
\bigl(-2\ln|x-y|\bigr)\,dy
\right|\\
&\qquad\le
2C_0c_{n,s}|\mathbb S^{n-1}|
\int_0^1r^{-2s}|\ln r|\,dr\\
&\qquad=
\frac{2C_0c_{n,s}|\mathbb S^{n-1}|}
{(1-2s)^2}
\longrightarrow0
\qquad\text{as }s\to0^+.
\end{align*}
Therefore,
\begin{equation}\label{eq:local-log-vanish}
c_{n,s}
\int_{B_1(x)}
\frac{u(x)-u(y)}
{|x-y|^{n+2s}}
\bigl(-2\ln|x-y|\bigr)\,dy
\longrightarrow0 \quad  \text{in}\quad L^\infty(\mathbb R^n).
\end{equation}

\medskip
\noindent
\textit{Step 2: the local nonlogarithmic term.}
Set
\[
J_s(x)
:=
\int_{B_1(x)}
\frac{u(x)-u(y)}
{|x-y|^{n+2s}}\,dy,\quad J_0(x)
:=
\int_{B_1(x)}
\frac{u(x)-u(y)}
{|x-y|^n}\,dy.
\]
Since \(r^{-2s}\ge1\) for \(0<r<1\), we have
\begin{align*}
|J_s(x)-J_0(x)|
&\le
C_0
\int_{B_1(x)}
|x-y|
\left(
\frac1{|x-y|^{n+2s}}
-\frac1{|x-y|^n}
\right)\,dy\\
&=
C_0|\mathbb S^{n-1}|
\left(\frac1{1-2s}-1\right).
\end{align*}
Consequently,
\begin{equation}\label{eq:Js-J0}
J_s\longrightarrow J_0
\qquad\text{in }L^\infty(\mathbb R^n).
\end{equation}
Moreover, $\|J_s\|_{L^\infty(\mathbb R^n)}
\le
\frac{C_0|\mathbb S^{n-1}|}{1-2s},$
so \(\{J_s\}_{0<s<1/4}\) is uniformly bounded in
\(L^\infty(\mathbb R^n)\). Combining
\eqref{eq:csbs-limit} and \eqref{eq:Js-J0}, we obtain
\begin{equation}\label{eq:local-nonlog-limit}
c_{n,s}b_{n,s}J_s
\longrightarrow
c_nJ_0
\qquad\text{in }L^\infty(\mathbb R^n).
\end{equation}

\medskip
\noindent
\textit{Step 3: the coefficient of \(u(x)\).}
Using polar coordinates,
\[
\int_{\mathbb R^n\setminus B_1(x)}
\frac{1}{|x-y|^{n+2s}}\,dy
=
|\mathbb S^{n-1}|
\int_1^\infty r^{-1-2s}\,dr
=
\frac{|\mathbb S^{n-1}|}{2s},
\]
and
\begin{align*}
\int_{\mathbb R^n\setminus B_1(x)}
\frac{-2\ln|x-y|}
{|x-y|^{n+2s}}\,dy
&=
-2|\mathbb S^{n-1}|
\int_1^\infty
\frac{\ln r}{r^{1+2s}}\,dr=
-\frac{|\mathbb S^{n-1}|}{2s^2}.
\end{align*}
It follows from \eqref{eq:alpha-ns} that
\begin{align*}
\alpha_{n,s}
&=
-\frac{c_{n,s}|\mathbb S^{n-1}|}{2s^2}
+
\frac{c_{n,s}b_{n,s}|\mathbb S^{n-1}|}{2s}=
\frac{|\mathbb S^{n-1}|}{2}
d_{n,s}q_{n,s}.
\end{align*}
Since $\frac{|\mathbb S^{n-1}|}{2}c_n=1,$
we conclude that
\begin{equation}\label{eq:alpha-limit}
\alpha_{n,s}
\longrightarrow
\rho_n
\qquad\text{as }s\to0^+.
\end{equation}
Hence $\alpha_{n,s}u
\longrightarrow
\rho_nu$ in $L^\infty(\mathbb R^n).$

\medskip
\noindent
\textit{Step 4: the far-field logarithmic term.}
Since \(u\in L^1(\mathbb R^n)\),
\begin{align*}
&\sup_{x\in\mathbb R^n}
\left|
c_{n,s}
\int_{\mathbb R^n\setminus B_1(x)}
\frac{u(y)}
{|x-y|^{n+2s}}
\bigl(-2\ln|x-y|\bigr)\,dy
\right|\le
c_{n,s}\|u\|_{L^1(\mathbb R^n)}
\sup_{r\ge1}\frac{2\ln r}{r^n}.
\end{align*}
The last supremum is finite and \(c_{n,s}\to0\). Therefore,
\begin{equation}\label{eq:far-log-vanish}
c_{n,s}
\int_{\mathbb R^n\setminus B_1(x)}
\frac{u(y)}
{|x-y|^{n+2s}}
\bigl(-2\ln|x-y|\bigr)\,dy
\longrightarrow0 \quad  \text{in}\quad L^\infty(\mathbb R^n).
\end{equation}

\medskip
\noindent
\textit{Step 5: the far-field nonlogarithmic term.}
Set
\[
K_s(x)
:=
\int_{\mathbb R^n\setminus B_1(x)}
\frac{u(y)}
{|x-y|^{n+2s}}\,dy,\quad K_0(x)
:=
\int_{\mathbb R^n\setminus B_1(x)}
\frac{u(y)}
{|x-y|^n}\,dy.
\]
For \(r\ge1\), write \(r=e^\tau\), \(\tau\ge0\). Since $
0\le1-e^{-2s\tau}\le2s\tau,$
we obtain
\begin{align*}
\sup_{r\ge1}
\left|r^{-n-2s}-r^{-n}\right|
&=
\sup_{\tau\ge0}
e^{-n\tau}\bigl(1-e^{-2s\tau}\bigr)\le
2s\sup_{\tau\ge0}\tau e^{-n\tau}
=
\frac{2s}{en}.
\end{align*}
Thus,
\begin{align*}
\|K_s-K_0\|_{L^\infty(\mathbb R^n)}
&\le
\|u\|_{L^1(\mathbb R^n)}
\sup_{r\ge1}
\left|r^{-n-2s}-r^{-n}\right|
\longrightarrow0.
\end{align*}
Therefore, by \eqref{eq:csbs-limit},
\begin{equation}\label{eq:far-nonlog-limit}
c_{n,s}b_{n,s}K_s
\longrightarrow
c_nK_0
\qquad\text{in }L^\infty(\mathbb R^n).
\end{equation}

Combining
\eqref{eq:local-log-vanish},
\eqref{eq:local-nonlog-limit},
\eqref{eq:alpha-limit},
\eqref{eq:far-log-vanish}, and
\eqref{eq:far-nonlog-limit}, we conclude that
\begin{align*}
(-\Delta)^{s+\Log}u(x)
\longrightarrow{}&
c_n
\int_{B_1(x)}
\frac{u(x)-u(y)}{|x-y|^n}\,dy-
c_n
\int_{\mathbb R^n\setminus B_1(x)}
\frac{u(y)}{|x-y|^n}\,dy
+
\rho_nu(x)
\end{align*}
in \(L^\infty(\mathbb R^n)\) as \(s\to0^+\). By the singular-integral
representation of the logarithmic Laplacian, the expression on the
right-hand side is precisely $(-\Delta)^{\Log}u(x).$
Hence
\[
(-\Delta)^{s+\Log}u
\longrightarrow
(-\Delta)^{\Log}u
\qquad\text{in }L^\infty(\mathbb R^n)
\quad\text{as }s\to0^+.
\]
This completes the proof.
\end{proof}

	\subsection{Spectral Representation}\label{spectral}
	
	In fact, the functional-calculus framework provides a clean and equivalent definition of
the fractional--logarithmic Laplacian. We first recall the necessary notions.

	For each pair \(u,v\in L^2(\mathbb{R}^n)\), $E_{u,v}$ is a regular complex Borel measure of bounded variation on \([0,\infty)\), supported on \(\sigma(-\Delta)\),  where
	\[
	E_{u,v}(B)\;=\;\bigl\langle E(B)\,u,\;v\bigr\rangle_{L^2(\mathbb{R}^n)},\:\:E\left(\sigma\left(-\Delta\right)\right)=I \quad\text{where}\:\:
	B\subset[0,\infty)\text{ is Borel set}
	\]
	and satisfying $\bigl|E_{u,v}\bigr|\bigl([0,\infty)\bigr)
	\;\le\;\|u\|_{L^2(\mathbb{R}^n)}\,\|v\|_{L^2(\mathbb{R}^n)}.$
	Moreover, $E_{u,u}$ is a positive measure with $E_{u,u}\bigl([0,\infty)\bigr)=\|u\|_{L^2(\mathbb{R}^n)}^2.$ In particular, the domain of \(-\Delta\) is given by
	\[\operatorname{Dom}(-\Delta)=\Bigl\{u\in L^2(\R^n):
	\int_{[0,\infty)}\lambda^2 dE_{u,u}(\lambda)<\infty\Bigr\},\]
	and
	\[ \langle -\Delta u,v\rangle_{L^2}
	=\int_{[0,\infty)} \lambda\, dE_{u,v}(\lambda),\quad u\in \operatorname{Dom}(-\Delta),\ v\in L^2(\mathbb{R}^n). \]
	
	For every Borel function
	$\varphi:[0,\infty)\to\R$, one can define
	\[
	\varphi(-\Delta) := \int_{[0,\infty)} \varphi(\lambda)\, dE(\lambda),
	\]
	with
	\[\operatorname{Dom}(\varphi(-\Delta) )=\Bigl\{u\in L^2(\R^n):
	\int_{[0,\infty)}|\varphi(\lambda)|^2 dE_{u,u}(\lambda)<\infty\Bigr\}.\]
	The action rule
	\[
	\langle \varphi(-\Delta)u,v\rangle_{L^2}
	=\int_{[0,\infty)} \varphi(\lambda)\, dE_{u,v}(\lambda),\quad u\in \operatorname{Dom}(\varphi(-\Delta)),\ v\in L^2(\mathbb{R}^n). 
	\]
	
	We next present a proposition which rigorously establishes that $(-\Delta)^{s+\Log}$ coincides with the first-order derivative of the fractional Laplace $(-\Delta)^{t}$ with respect to $t$ at $t=s.$

     \begin{proof}[\bf Proof of Proposition \ref{prop:derivative-frac-power}:]
		By spectral calculus, for $t$ close to $s$ we can write
		\[
		\frac{(-\Delta)^t-(-\Delta)^s}{t-s}\,u - (-\Delta)^{s+\Log}u
		=\int_{[0,\infty)} h_{t,s}(\lambda)\,dE(\lambda)\,u,
		\]
		where
		\[
		h_{t,s}(\lambda):=\frac{\lambda^t-\lambda^s}{t-s}-\lambda^s\ln\lambda,
		\qquad \lambda>0,
		\]
		and we set $h_{t,s}(0):=0$.
		
		Hence, using the spectral measure $E_{u,u}$,
		\begin{equation}\label{eq:derivative-frac-power-L2}
			\Bigl\|
			\frac{(-\Delta)^t-(-\Delta)^s}{t-s}\,u - (-\Delta)^{s+\Log}u
			\Bigr\|_{L^2}^2
			=\int_{[0,\infty)} |h_{t,s}(\lambda)|^2\,dE_{u,u}(\lambda).
		\end{equation}
		For each fixed $\lambda>0$, we have $h_{t,s}(\lambda)\to0$ as $t\to s$, since
		$\partial_t(\lambda^t)=\lambda^t\ln\lambda$.
		
		It remains to control $h_{t,s}$ uniformly in $\lambda$ for $t$ near $s$.
		Fix $\delta\in(0,\min\left\{s,s_0-s\right\})$ and restrict to $|t-s|<\delta$.
		By Taylor's theorem with remainder, there exists $\theta=\theta(\lambda,t,s)$
		between $s$ and $t$ such that
		\[
		\lambda^t=\lambda^s+(t-s)\lambda^s\ln\lambda
		+\frac{(t-s)^2}{2}\lambda^\theta(\ln\lambda)^2.
		\]
		Therefore,
		\[
		h_{t,s}(\lambda)=\frac{(t-s)}{2}\lambda^\theta(\ln\lambda)^2,
		\]
		and hence, for $|t-s|<\delta$,
		\begin{equation}\label{eq:hts-bound}
			|h_{t,s}(\lambda)|^2
			\le C\,|t-s|^2\,(\ln\lambda)^4
			\Bigl(\lambda^{2(s-\delta)}\mathbf 1_{\{\lambda\le1\}}
			+\lambda^{2(s+\delta)}\mathbf 1_{\{\lambda\ge1\}}\Bigr),
		\end{equation}
		with a constant $C>0$ independent of $\lambda,t$.

		Since $s-\delta>0$, the function
		$\lambda\mapsto (\ln\lambda)^4\lambda^{2(s-\delta)}$ is bounded on $(0,1]$. Hence
		\[
		\int_{(0,1]} (\ln\lambda)^4\lambda^{2(s-\delta)}\,dE_{u,u}(\lambda)
		\le C\int_{(0,1]} dE_{u,u}(\lambda)
		\le C\|u\|_{L^2}^2<\infty.
		\]
		
		Since $\delta<s_0-s$ and $\ln\lambda$ grows slower than any power of
		$\lambda$, there exists $C_\delta>0$ such that for all $\lambda\ge1$, $(\ln\lambda)^4\,\lambda^{2(s+\delta)}\le C_\delta\,\lambda^{2s_0}.$
		Therefore, using $u\in H^{2s_0}(\R^n)$,
		\[
		\int_{[1,\infty)} (\ln\lambda)^4\lambda^{2(s+\delta)}\,dE_{u,u}(\lambda)
		\le C_\delta\int_{[1,\infty)} \lambda^{2s_0}\,dE_{u,u}(\lambda)
		<\infty.
		\]
		
		Combining these two bounds with \eqref{eq:hts-bound}, we obtain from
		\eqref{eq:derivative-frac-power-L2} that
		\[
		\int_{[0,\infty)} |h_{t,s}(\lambda)|^2\,dE_{u,u}(\lambda)
		\le C\,|t-s|^2,
		\]
		where $C$ is independent of $t$ near $s$. Hence the left-hand side tends to
		$0$ as $t\to s$.
        \end{proof}

	In what follows, we relate the spectral measure of $-\Delta$ on $\R^n$ to the Fourier transform, and thus provide Fourier-multiplier characterizations of the fractional--logarithmic Laplacian and the associated space. 
	
	In the Euclidean setting, $-\Delta$ is unitarily diagonalized by the Fourier
	transform. More precisely, for $f\in L^2(\mathbb{R}^n)$, the spectral measure of $-\Delta$ satisfies
	\[
	E_{f,f}(B)
	=\int_{\{|\xi|^2\in B\}}|\widehat f(\xi)|^2\,d\xi,
	\qquad B\subset[0,\infty)\ \text{Borel}.
	\]
	Writing
	\[
	F(\lambda):=E_{f,f}\bigl((0,\lambda]\bigr)
	=\int_{|\xi|\le\sqrt\lambda}|\widehat f(\xi)|^2\,d\xi
	\]
	and differentiating in $\lambda$ in polar coordinates $\xi=r\omega$ yields
	\[
	dE_{f,f}(\lambda)
	=F'(\lambda)\,d\lambda
	=\frac12\,\lambda^{\frac n2-1}
	\int_{\mathbb{S}^{n-1}}\bigl|\widehat f(\sqrt\lambda\,\omega)\bigr|^2\,d\omega\,d\lambda.
	\]

	\begin{lemma}
		Let $u,v\in L^2(\mathbb R^n)$. Then for every Borel set $B\subset[0,\infty)$,
		\[E_{u,v}(B):=\langle E(B)u,v\rangle_{L^2}
			=\int_{\{|\xi|^2\in B\}}\widehat u(\xi)\,\overline{\widehat v(\xi)}\,d\xi.\]
		In particular, writing
		\[
		F_{u,v}(\lambda):=E_{u,v}((0,\lambda])
		=\int_{|\xi|\le\sqrt\lambda}\widehat u(\xi)\,\overline{\widehat v(\xi)}\,d\xi,
		\]
		one has that $E_{u,v}$ is absolutely continuous on $(0,\infty)$ and
		\begin{equation}\label{eq:dEuv-density}
			dE_{u,v}(\lambda)=F'_{u,v}(\lambda)\,d\lambda
			=\frac12\,\lambda^{\frac n2-1}
			\int_{\mathbb{S}^{n-1}}\widehat u(\sqrt\lambda\,\omega)\,
			\overline{\widehat v(\sqrt\lambda\,\omega)}\,d\omega\,d\lambda.
		\end{equation}
	\end{lemma}
	
	\begin{proof}
		Since $\mathcal F$ is unitary and diagonalizes $-\Delta$ on $\R^n$,
		the spectral projection $E(B)$ acts as a Fourier multiplier:
		\[
		\widehat{E(B)f}(\xi)=\mathbf 1_{B}(|\xi|^2)\,\widehat f(\xi),
		\qquad f\in L^2(\mathbb R^n).
		\]
		Therefore, by Plancherel,
		\[
		E_{u,v}(B)=\langle E(B)u,v\rangle_{L^2}
		=\int_{\R^n}\mathbf 1_B(|\xi|^2)\,\widehat u(\xi)\,
		\overline{\widehat v(\xi)}\,d\xi
		=\int_{\{|\xi|^2\in B\}}\widehat u(\xi)\,\overline{\widehat v(\xi)}\,d\xi.
		\]
		For $\lambda>0$,
		\[
		F_{u,v}(\lambda)=E_{u,v}((0,\lambda])
		=\int_{|\xi|\le\sqrt\lambda}\widehat u(\xi)\,\overline{\widehat v(\xi)}\,d\xi.
		\]
		Passing to polar coordinates $\xi=r\omega$,
		\[
		F_{u,v}(\lambda)=\int_0^{\sqrt\lambda}\int_{\mathbb{S}^{n-1}}
		\widehat u(r\omega)\,\overline{\widehat v(r\omega)}\,r^{n-1}\,d\omega\,dr.
		\]
	Thus,
	{\small	\[
		F'_{u,v}(\lambda)
		=\frac{1}{2\sqrt\lambda}\int_{\mathbb{S}^{n-1}}
		\widehat u(\sqrt\lambda\,\omega)\,\overline{\widehat v(\sqrt\lambda\,\omega)}\,
		(\sqrt\lambda)^{\,n-1}\,d\omega
		=\frac12\,\lambda^{\frac n2-1}\int_{\mathbb{S}^{n-1}}
		\widehat u(\sqrt\lambda\,\omega)\,\overline{\widehat v(\sqrt\lambda\,\omega)}\,d\omega,
		\]}
		which gives \eqref{eq:dEuv-density}.
	\end{proof}

	\begin{lemma}
		\label{lem:fourier-dEuv}
		For every $s>0$ and $u\in H^{2s+\Log}(\R^n),\ v\in L^2(\mathbb{R}^n)$
		\begin{equation}\label{eq:lambda-s-log-Euv-Fourier}
			\int_{[0,\infty)}\lambda^s\ln\lambda\,dE_{u,v}(\lambda)
			=\int_{\R^n}|\xi|^{2s}\ln\bigl(|\xi|^2\bigr)\,
			\widehat u(\xi)\,\overline{\widehat v(\xi)}\,d\xi.
		\end{equation}
	\end{lemma}
	
	\begin{proof}
		Using \eqref{eq:dEuv-density},
		\begin{align*}
		    \int_{[0,\infty)}\lambda^s\ln\lambda\,dE_{u,v}(\lambda)
		&=\int_0^\infty \lambda^s\ln\lambda\,F'_{u,v}(\lambda)\,d\lambda
		\\&=\frac12\int_0^\infty \lambda^{\frac n2-1+s}\ln\lambda
		\int_{\mathbb{S}^{n-1}}\widehat u(\sqrt\lambda\,\omega)\,
		\overline{\widehat v(\sqrt\lambda\,\omega)}\,d\omega\,d\lambda.
		\end{align*}
		Now set $\lambda=r^2$, then
		\begin{align*}
			\int_{[0,\infty)}\lambda^s\ln\lambda\,dE_{u,v}(\lambda)
			&=\frac12\int_0^\infty\!\!\int_{\mathbb{S}^{n-1}}
			r^{n-2+2s}\ln(r^2)\,\widehat u(r\omega)\,\overline{\widehat v(r\omega)}\,d\omega\,(2r\,dr)\\
			&=\int_0^\infty\!\!\int_{\mathbb{S}^{n-1}}
			r^{n-1+2s}\ln(r^2)\,\widehat u(r\omega)\,\overline{\widehat v(r\omega)}\,d\omega\,dr\\
			&=\int_{\R^n}|\xi|^{2s}\ln\bigl(|\xi|^2\bigr)\,
			\widehat u(\xi)\,\overline{\widehat v(\xi)}\,d\xi,
		\end{align*}
		which is exactly \eqref{eq:lambda-s-log-Euv-Fourier}.
	\end{proof}

    \begin{proof}[\bf Proof of Theorem \ref{thm:fourspec}:]
(i) We now change variables $\lambda=r^2$. Then
\begin{align*}
	\int_0^\infty \lambda^{2s}\ln^2\lambda\,dE_{f,f}(\lambda)
	&=\frac12\int_0^\infty
	\lambda^{\frac n2-1+2s}\ln^2\lambda
	\int_{\mathbb S^{n-1}}
	\bigl|\widehat f(\sqrt\lambda\,\omega)\bigr|^2\,d\omega\,d\lambda \\
	&=\frac12\int_0^\infty\!\!\int_{\mathbb S^{n-1}}
	r^{n-2+4s}\ln^2(r^2)
	\bigl|\widehat f(r\omega)\bigr|^2\,d\omega\,2r\,dr \\
	&=\int_{\mathbb R^n}
	|\xi|^{4s}\ln^2\bigl(|\xi|^2\bigr)\,
	\bigl|\widehat f(\xi)\bigr|^2\,d\xi.
\end{align*}
	
	(ii)	Let $\varphi\in L^2(\mathbb R^n)$. By the definition of the functional calculus,
		\[
		\langle (-\Delta)^{s+\Log}u,\varphi\rangle_{L^2}
		=\int_{[0,\infty)}\lambda^s\ln\lambda\,dE_{u,\varphi}(\lambda).
		\]
		Applying Lemma~\ref{lem:fourier-dEuv} with $v=\varphi$ gives
		\[
		\langle (-\Delta)^{s+\Log}u,\varphi\rangle_{L^2}
		=\int_{\R^n}|\xi|^{2s}\ln(|\xi|^2)\,\widehat u(\xi)\,
		\overline{\widehat\varphi(\xi)}\,d\xi.
		\]
		By Plancherel, the right-hand side equals
		\[
		\Big\langle \cF^{-1}\!\big(|\xi|^{2s}\ln(|\xi|^2)\,\widehat u\big),\,
		\varphi\Big\rangle_{L^2}.
		\]
	Therefore,
		\[
		(-\Delta)^{s+\Log}u=\cF^{-1}\!\big(|\xi|^{2s}\ln(|\xi|^2)\,\widehat u\big)
		\quad\text{in }L^2(\R^n),
		\]
		which is  \eqref{eq:fourier-symbol-frac-log}, we complete the proof.
\end{proof}

	\subsection{Extension Problem}

    In this subsection, we derive the extension characterization of the fractional-logarithmic Laplacian and identify the associated Dirichlet-to-Neumann-type boundary formula. We first establish the boundary trace of \(v_s\).

	\begin{lemma}\label{b1v1s}
Let \(u\in C_c^\infty(\mathbb R^n)\), and let \(v_s\) be defined as
in~\eqref{pot 2.11}. Then, for every \(x\in\mathbb R^n\),
\[
\lim_{t\to0^+}v_s(x,t)=b_1u(x),
\]
where
\[
b_1
:=
p_{n,s}\int_{\mathbb R^n}
\frac{-\ln(1+|z|^2)}
{(1+|z|^2)^{\frac{n+2s}{2}}}\,dz
=
-\frac{\partial_s p_{n,s}}{p_{n,s}}.
\]
\end{lemma}

\begin{proof}
Making the change of variables \(y=x+tz\), we obtain
\begin{align*}
v_s(x,t)
&=
p_{n,s}t^{2s}
\int_{\mathbb R^n}
\frac{2\ln t-\ln(|x-y|^2+t^2)}
{(|x-y|^2+t^2)^{\frac{n+2s}{2}}}
u(y)\,dy\\
&=
p_{n,s}
\int_{\mathbb R^n}
\frac{-\ln(1+|z|^2)}
{(1+|z|^2)^{\frac{n+2s}{2}}}
u(x+tz)\,dz.
\end{align*}
Since \(u\) is bounded and
\[
\frac{\ln(1+|z|^2)}
{(1+|z|^2)^{\frac{n+2s}{2}}}
\in L^1(\mathbb R^n),
\]
the dominated convergence theorem yields
\[
\lim_{t\to0^+}v_s(x,t)
=
p_{n,s}u(x)
\int_{\mathbb R^n}
\frac{-\ln(1+|z|^2)}
{(1+|z|^2)^{\frac{n+2s}{2}}}\,dz
=
b_1u(x).
\]

It remains to verify the identity involving \(p_{n,s}\). Set
\[
g(s)
:=
\int_{\mathbb R^n}
\frac{1}
{(1+|z|^2)^{\frac{n+2s}{2}}}\,dz,
\qquad
p_{n,s}=\frac1{g(s)}.
\]
For each fixed \(s_0\in(0,1)\), choose \(\delta>0\) such that
\(s_0-\delta>0\). Uniformly for
\(s\in[s_0-\delta,s_0+\delta]\), we have
\[
\left|
\partial_s
(1+|z|^2)^{-\frac{n+2s}{2}}
\right|
\le
\ln(1+|z|^2)
(1+|z|^2)^{-\frac{n+2(s_0-\delta)}{2}},
\]
and the function on the right-hand side belongs to
\(L^1(\mathbb R^n)\). Hence differentiation under the integral sign is
justified, and
\[
g'(s)
=
\int_{\mathbb R^n}
\frac{-\ln(1+|z|^2)}
{(1+|z|^2)^{\frac{n+2s}{2}}}\,dz.
\]
Therefore, $b_1
=
p_{n,s}g'(s)
=
\frac{g'(s)}{g(s)}.$
Since \(p_{n,s}=g(s)^{-1}\), it follows that
\[
\frac{\partial_s p_{n,s}}{p_{n,s}}
=
\partial_s\ln p_{n,s}
=
-\frac{g'(s)}{g(s)}
=
-b_1.
\]
This completes the proof.
\end{proof}

	We now turn to the corresponding extension problem and complete the proof.
	We first recall the notation needed in the argument. For a given
	\(u\in C_c^\infty(\R^n)\), define
	\[
	w_s(x,t)
	=
	p_{n,s}\int_{\R^n}
	\frac{t^{2s}u(y)}
	{\bigl(|x-y|^2+t^2\bigr)^{\frac{n+2s}{2}}}\,dy,
	\qquad (x,t)\in\R^{n+1}_+,
	\]
	where
	\[
	p_{n,s}
	=
	\Big(
	\int_{\R^n}
	\frac{1}{(|z|^2+1)^{\frac{n+2s}{2}}}\,dz
	\Big)^{-1}
	=
	\pi^{-\frac n2}
	\frac{\Gamma\left(\frac{n+2s}{2}\right)}{\Gamma(s)}.
	\]
	Set
	\[
	v_s(x,t)
	=
	p_{n,s}\,\partial_s\left(\frac{w_s}{p_{n,s}}\right)
	=
	p_{n,s}\,t^{2s}
	\int_{\R^n}
	\frac{2\ln t-\ln\bigl(|x-y|^2+t^2\bigr)}
	{\bigl(|x-y|^2+t^2\bigr)^{\frac{n+2s}{2}}}
	u(y)\,dy,
	\]
	and
	\[
	{\bf d}_{s}
	=
	\frac{c_{n,s}}{p_{n,s}}
	=
	2^{2s}\frac{s\,\Gamma(s)}{\Gamma(1-s)},
	\qquad
	b_1
	=
	p_{n,s}
	\int_{\R^n}
	\frac{-\ln\bigl(|z|^2+1\bigr)}
	{\bigl(|z|^2+1\bigr)^{\frac{n+2s}{2}}}\,dz.
	\]

\begin{proof}[\bf Proof of Theorem~\ref{pr 2.1}]
By the definition of \(v_s\), we have
\begin{equation}\label{eq:relation-vs-ws}
\partial_s w_s
=
\frac{\partial_s p_{n,s}}{p_{n,s}}\,w_s+v_s.
\end{equation}
Moreover, Lemma~\ref{b1v1s} gives
\[
v_s(x,0)=b_1u(x),
\qquad
b_1=-\frac{\partial_s p_{n,s}}{p_{n,s}},
\]
while \(w_s(x,0)=u(x)\). Consequently,
\[
\lim_{t\to0^+}\partial_s w_s(x,t)
=
\left(
\frac{\partial_s p_{n,s}}{p_{n,s}}+b_1
\right)u(x)
=0.
\]

We first prove the extension equation. Recall that \(w_s\) solves
\[
\operatorname{div}_{(x,t)}
\bigl(t^{1-2s}\nabla_{(x,t)}w_s\bigr)=0
\qquad\text{in }\mathbb R^{n+1}_+.
\]
Since \(u\in C_c^\infty(\mathbb R^n)\), differentiation with respect to
\(s\) is legitimate for \(t>0\). Hence
\begin{align*}
0
&=
\partial_s
\operatorname{div}_{(x,t)}
\bigl(t^{1-2s}\nabla_{(x,t)}w_s\bigr)\\
&=
\operatorname{div}_{(x,t)}
\bigl(-2(\ln t)t^{1-2s}\nabla_{(x,t)}w_s\bigr)
+
\operatorname{div}_{(x,t)}
\bigl(t^{1-2s}\nabla_{(x,t)}(\partial_s w_s)\bigr).
\end{align*}
Using $\operatorname{div}_{(x,t)}
\bigl(t^{1-2s}\nabla_{(x,t)}w_s\bigr)=0,$
we obtain
\begin{align*}
&\operatorname{div}_{(x,t)}
\bigl(-2(\ln t)t^{1-2s}\nabla_{(x,t)}w_s\bigr)\\
&\qquad
=
\nabla_{(x,t)}(-2\ln t)\cdot
t^{1-2s}\nabla_{(x,t)}w_s
=
-2t^{-2s}\partial_t w_s.
\end{align*}
It follows that
\[
\operatorname{div}_{(x,t)}
\bigl(t^{1-2s}\nabla_{(x,t)}(\partial_s w_s)\bigr)
=
2t^{-2s}\partial_t w_s.
\]
On the other hand, by \eqref{eq:relation-vs-ws}, $v_s
=
\partial_s w_s
-
\frac{\partial_s p_{n,s}}{p_{n,s}}\,w_s.$
Since \(\frac{\partial_s p_{n,s}}{p_{n,s}}\) is independent of \((x,t)\),
the equation satisfied by \(w_s\) yields
\[
\operatorname{div}_{(x,t)}
\bigl(t^{1-2s}\nabla_{(x,t)}v_s\bigr)
=
2t^{-2s}\partial_t w_s
\qquad\text{in }\mathbb R^{n+1}_+.
\]
Together with Lemma~\ref{b1v1s}, this proves that \(v_s\) solves
\eqref{eq 1.1-ext}.

We next prove the boundary representation of
\((-\Delta)^{s+\Log}\). By a change of variables, for every \(t>0\),
\[
p_{n,s}t^{2s}
\int_{\mathbb R^n}
\frac{2\ln t-\ln\bigl(|x-y|^2+t^2\bigr)}
{\bigl(|x-y|^2+t^2\bigr)^{\frac{n+2s}{2}}}\,dy
=
b_1.
\]
Therefore,
\begin{align*}
&t^{-2s}\bigl(v_s(x,t)-v_s(x,0)\bigr)\\
&\quad=
p_{n,s}
\int_{\mathbb R^n}
\frac{2\ln t-\ln\bigl(|x-y|^2+t^2\bigr)}
{\bigl(|x-y|^2+t^2\bigr)^{\frac{n+2s}{2}}}
\bigl(u(y)-u(x)\bigr)\,dy\\
&\quad=
(2\ln t)t^{-2s}\bigl(w_s(x,t)-w_s(x,0)\bigr)\\
&\qquad\quad
+
p_{n,s}\int_{\mathbb R^n}
\frac{-\ln\bigl(|x-y|^2+t^2\bigr)}
{\bigl(|x-y|^2+t^2\bigr)^{\frac{n+2s}{2}}}
\bigl(u(y)-u(x)\bigr)\,dy.
\end{align*}
It follows that
\begin{align}\label{eq:extension-combination}
&\bigl(b_{n,s}-2\ln t\bigr)t^{-2s}
\bigl(w_s(x,t)-w_s(x,0)\bigr)
+t^{-2s}\bigl(v_s(x,t)-v_s(x,0)\bigr)
\nonumber\\
&\quad=
b_{n,s}t^{-2s}\bigl(w_s(x,t)-u(x)\bigr)
\nonumber\\
&\qquad\quad
+
p_{n,s}\int_{\mathbb R^n}
\frac{-\ln\bigl(|x-y|^2+t^2\bigr)}
{\bigl(|x-y|^2+t^2\bigr)^{\frac{n+2s}{2}}}
\bigl(u(y)-u(x)\bigr)\,dy.
\end{align}

To pass to the limit, set $\delta_z^2u(x)
:=
u(x+z)+u(x-z)-2u(x).$
By symmetry,
\[
{\bf d}_st^{-2s}\bigl(w_s(x,t)-u(x)\bigr)
=
\frac{c_{n,s}}{2}
\int_{\mathbb R^n}
\frac{\delta_z^2u(x)}
{\bigl(|z|^2+t^2\bigr)^{\frac{n+2s}{2}}}\,dz.
\]
Since $|\delta_z^2u(x)|
\le C\min\{|z|^2,1\},$
the dominated convergence theorem gives
\begin{equation}\label{eq:ws-boundary-limit}
    \lim_{t\to0^+}
{\bf d}_st^{-2s}\bigl(w_s(x,t)-u(x)\bigr)=
\frac{c_{n,s}}{2}
\int_{\mathbb R^n}
\frac{\delta_z^2u(x)}{|z|^{n+2s}}\,dz=
-(-\Delta)^su(x).
\end{equation}

Similarly,
\begin{align*}
&{\bf d}_sp_{n,s}
\int_{\mathbb R^n}
\frac{-\ln\bigl(|x-y|^2+t^2\bigr)}
{\bigl(|x-y|^2+t^2\bigr)^{\frac{n+2s}{2}}}
\bigl(u(y)-u(x)\bigr)\,dy\\
&\qquad=
\frac{c_{n,s}}{2}
\int_{\mathbb R^n}
\frac{-\ln\bigl(|z|^2+t^2\bigr)}
{\bigl(|z|^2+t^2\bigr)^{\frac{n+2s}{2}}}
\delta_z^2u(x)\,dz.
\end{align*}
For \(|z|<1\), the integrand is bounded by $C|z|^{2-n-2s}\bigl(1+|\ln|z||\bigr),$
whereas for \(|z|\ge1\), it is bounded by $C|z|^{-n-2s}\ln(e+|z|).$
Both functions are integrable because \(s\in(0,1)\). Hence, again by
the dominated convergence theorem,
\begin{align}\label{eq:vs-log-limit}
&-\lim_{t\to0^+}{\bf d}_sp_{n,s}
\int_{\mathbb R^n}
\frac{-\ln\bigl(|x-y|^2+t^2\bigr)}
{\bigl(|x-y|^2+t^2\bigr)^{\frac{n+2s}{2}}}
\bigl(u(y)-u(x)\bigr)\,dy
\nonumber\\
&\qquad=
2c_{n,s}\PV\int_{\mathbb R^n}
\frac{(u(x)-u(y))(-\ln|x-y|)}
{|x-y|^{n+2s}}\,dy.
\end{align}

Combining \eqref{eq:extension-combination},
\eqref{eq:ws-boundary-limit}, and \eqref{eq:vs-log-limit}, we conclude that
\begin{align*}
&-\lim_{t\to0^+}{\bf d}_s
\Big[
\bigl(b_{n,s}-2\ln t\bigr)t^{-2s}
\bigl(w_s(x,t)-w_s(x,0)\bigr)\\
&\hspace{3.7cm}
+t^{-2s}\bigl(v_s(x,t)-v_s(x,0)\bigr)
\Big]\\
&\qquad=
b_{n,s}(-\Delta)^su(x)
+
2c_{n,s}\PV\int_{\mathbb R^n}
\frac{(u(x)-u(y))(-\ln|x-y|)}
{|x-y|^{n+2s}}\,dy\\
&\qquad=
(-\Delta)^{s+\Log}u(x),
\end{align*}
where the last equality follows from the singular-integral
representation of \((-\Delta)^{s+\Log}\). This completes the proof.
\end{proof}

	\setcounter{equation}{0}
	\section{Poisson Problem}
    In this section, we study the Poisson problem associated with the fractional-logarithmic Laplacian, focusing on existence and regularity of weak solutions. The existence theory is based mainly on the Lax--Milgram framework, while the regularity part relies on a Moser iteration argument. To carry this out, we first establish the required structural properties of the corresponding function spaces and energy forms.

	\subsection{Properties of Functional Spaces}
We first clarify the relationship between the spectral fractional-logarithmic space and the energy fractional-logarithmic space, which helps establish a transparent link between functional calculus and the Fourier-transform approach.

\begin{proof}[\bf Proof of Proposition \ref{prop:Hslog-vs-form}:]

Setting \(z=x-y\) and arguing as in the standard fractional Sobolev case; see, for instance, \cite[Proposition~3.4]{EGE}, we obtain
\[\iint_{\mathbb R^n\times\mathbb R^n}
\bigl(u(x)-u(y)\bigr)^2
{\bf k}_{s+\Log,+}(x-y)\,dx\,dy
=
\int_{\mathbb R^n}
m(\xi)\,|\widehat u(\xi)|^2\,d\xi,\]
where
\[
m(\xi)
:=
2\int_{\mathbb R^n}
\bigl(1-\cos(\xi\cdot z)\bigr)
{\bf k}_{s+\Log,+}(z)\,dz.
\]
Since \({\bf k}_{s+\Log,+}\) is radial and non-negative, \(m\) is radial and non-negative.\medskip

Now, write \(r:=|\xi|\). Since \({\bf k}_{s+\Log,+}\) is radial, after possible  rotation we may assume that \(\xi=re_1\). Then, by using polar coordinates \(z=\rho\theta\), one sees that
\begin{equation}\label{eq:m-polar-form}
\begin{aligned}
m(\xi)
&=
2c_{n,s}
\int_0^1\int_{\mathbb S^{n-1}}
\bigl(1-\cos(r\rho\theta_1)\bigr)
\rho^{-1-2s}\ln\frac1\rho
\,d\theta\,d\rho.
\end{aligned}
\end{equation}
We claim that there exist constants \(C>0\), depending only on
\(n\) and \(s\), such that
\begin{equation}\label{eq:m-two-sided}
m(\xi)
\le C\,\Psi_s(|\xi|)
\qquad\text{for every }\xi\in\mathbb R^n,
\end{equation}
where
\[
\Psi_s(r)
:=
\begin{cases}
r^2, & 0\le r<2,\\[1mm]
r^{2s}\ln r, & r\ge2.
\end{cases}
\]

We first consider the low-frequency regime \(0<r\le2\). Since $1-\cos t\le \frac{t^2}{2}$ for every $t\in\mathbb R,$
formula \eqref{eq:m-polar-form} gives
\[
m(\xi)
\le
Cr^2
\int_0^1
\rho^{1-2s}\ln\frac1\rho\,d\rho
\le Cr^2,
\]
where the last integral is finite because \(s\in(0,1)\).

We next consider the high-frequency regime \(r\ge2\). Using $1-\cos(\xi\cdot z)
\le
\min\left\{2,r^2|z|^2\right\},$
we split the integral at the scale \(1/r\) and obtain
\begin{align*}
m(\xi)
&\le
Cr^2\int_0^{1/r}
\rho^{1-2s}\ln\frac1\rho\,d\rho
+
C\int_{1/r}^1
\rho^{-1-2s}\ln\frac1\rho\,d\rho.
\end{align*}
A direct computation gives
\[
\int_0^{1/r}
\rho^{1-2s}\ln\frac1\rho\,d\rho
=
r^{-2+2s}
\left(
\frac{\ln r}{2-2s}
+
\frac{1}{(2-2s)^2}
\right),
\]
and
\[
\int_{1/r}^1
\rho^{-1-2s}\ln\frac1\rho\,d\rho
=
r^{2s}
\left(
\frac{\ln r}{2s}
-
\frac{1}{4s^2}
\right)
+
\frac{1}{4s^2}.
\]
Since \(r\ge2\), these identities imply
\begin{equation}\label{eq:m-large-upper}
m(\xi)\le Cr^{2s}\ln r.
\end{equation}

Consequently, by Theorem \ref{thm:fourspec} (i),
\[
\|u\|_{\mathcal H^{s+\Log}(\mathbb R^n)}^2
=
\|u\|_{L^2(\mathbb R^n)}^2
+
\int_{\mathbb R^n}
m(\xi)|\widehat u(\xi)|^2\,d\xi
\le
C\|u\|_{H^{s+\Log}(\mathbb R^n)}^2,
\]
where
\[
\|u\|_{H^{s+\Log}(\mathbb R^n)}^2
:=
\|u\|_{L^2(\mathbb R^n)}^2
+
\int_{\mathbb R^n}
|\xi|^{2s}\ln^2(|\xi|^2)
|\widehat u(\xi)|^2\,d\xi.
\]
Thus, $H^{s+\Log}(\mathbb R^n)
\hookrightarrow
\mathcal H^{s+\Log}(\mathbb R^n)$
continuously. This inclusion is strict. To see this, choose \(\varphi\in L^2(\mathbb R^n)\)
such that $|\widehat \varphi(\xi)|^2
=
\mathbf 1_{\{|\xi|\ge e\}}
|\xi|^{-n-2s}
(\ln|\xi|)^{-3}.$
Such a function belongs to \(L^2(\mathbb R^n)\), since
\[
\int_{|\xi|\ge e}|\widehat \varphi(\xi)|^2\,d\xi
=
|\mathbb S^{n-1}|
\int_e^\infty
r^{-1-2s}(\ln r)^{-3}\,dr
<\infty.
\]
Using \eqref{eq:m-two-sided}, we obtain
\begin{align*}
\int_{\mathbb R^n}
m(\xi)|\widehat \varphi(\xi)|^2\,d\xi
&\le C
\int_e^\infty
r^{2s}\ln r\,
r^{-n-2s}(\ln r)^{-3}r^{n-1}\,dr
<\infty.
\end{align*}
Hence \(\varphi\in\mathcal H^{s+\Log}(\mathbb R^n)\). On the other hand,
\begin{align*}
\int_{\mathbb R^n}
|\xi|^{2s}\ln^2(|\xi|^2)
|\widehat \varphi(\xi)|^2\,d\xi
&=
4|\mathbb S^{n-1}|
\int_e^\infty
r^{2s}(\ln r)^2
r^{-n-2s}(\ln r)^{-3}r^{n-1}\,dr
=\infty.
\end{align*}
Therefore \(\varphi\notin H^{s+\Log}(\mathbb R^n)\), and hence $H^{s+\Log}(\mathbb R^n)
\subsetneq
\mathcal H^{s+\Log}(\mathbb R^n).$
\end{proof}

	\begin{proposition}\label{prop:Hslog-Hilbert}
		For $0<s<1$, the space $\mathcal H^{s+\Log}(\R^n)$ endowed with the inner product $\langle u,v\rangle_{\mathcal H^{s+\Log}}$ defined in~\eqref{eq:inner-prod-fraglog} 
		is a Hilbert space. In particular, one has that
        \[
        \|u\|_{\mathcal H^{s+\Log}}^{\,2}
		=\langle u,u\rangle_{\mathcal H^{s+\Log}}.
        \]
	\end{proposition}

\begin{proof}
		For $u,v\in\mathcal H^{s+\Log}(\mathbb R^n)$, the bilinear form is well-defined. Indeed, by Cauchy--Schwarz's inequality, one sees that
		\begin{align*}
		&\left|
		\iint (u(x)-u(y))(v(x)-v(y))\,{\bf k}_{s+\Log,+}(x-y)\,dx\,dy
		\right|\\
		&\qquad \le \,[u]_{s+\Log,+}\,[v]_{s+\Log,+}<\infty.
		\end{align*}
    	If $\langle u,u\rangle_{\mathcal H^{s+\Log}}=0$, then
		$\|u\|_{L^2(\mathbb R^n)}^{\,2}=0$, hence $u=0$ a.e.\ in $\mathbb R^n$.
		Therefore, it is an inner product.
		
		Let $(u_k)_{k\in\mathbb N}$ be a Cauchy sequence in
		$\mathcal H^{s+\Log}(\mathbb R^n)$. Then it is Cauchy in $L^2(\mathbb R^n)$,
		so there exists $u\in L^2(\mathbb R^n)$ such that $u_k\to u$ in $L^2$. Moreover, since $(u_k)$ is Cauchy, the sequence is bounded in the
		$\mathcal H^{s+\Log}$-norm; thus, $\sup_k [u_k]_{s+\Log,+}<\infty$. After possibly passing to a subsequence, we may assume that $u_k(x)\to u(x)$
		for a.e.\ $x\in\mathbb R^n$. By Fatou's lemma and the non-negativity of the
		kernel ${\bf k}_{s+\Log,+}$, we obtain
		\begin{align*}
		 &   \iint_{\mathbb R^n\times\mathbb R^n}
		\bigl(u(x)-u(y)\bigr)^2\,{\bf k}_{s+\Log,+}(x-y)\,dx\,dy
		\\ &\qquad\le \liminf_{k\to\infty}
		\iint_{\mathbb R^n\times\mathbb R^n}
		\bigl(u_k(x)-u_k(y)\bigr)^2\,{\bf k}_{s+\Log,+}(x-y)\,dx\,dy,
		\end{align*}
		which shows that $[u]_{s+\Log,+}<\infty$, i.e.\ $u\in\mathcal H^{s+\Log}(\mathbb R^n)$.
		
		Finally, we show that $u_k\to u$ in $\mathcal H^{s+\Log}$. Fix $\varepsilon>0$.
		Since $(u_k)$ is Cauchy in $\mathcal H^{s+\Log}$, there exists $N$ such that
		for all $k,m\ge N$,
		\[
		\|u_k-u_m\|_{L^2}^2+[u_k-u_m]_{s+\Log,+}^2<\varepsilon.
		\]
		Letting $m\to\infty$ and using $u_m\to u$ in $L^2$ we get
		$\|u_k-u\|_{L^2}^2\le \varepsilon$ for all $k\ge N$. Moreover, by Fatou's lemma, we obtain
		\[
		[u_k-u]_{s+\Log,+}^2
		\le \liminf_{m\to\infty}[u_k-u_m]_{s+\Log,+}^2
		\le \varepsilon,
		\qquad k\ge N.
		\]
		Hence $\|u_k-u\|_{\mathcal H^{s+\Log}}^2
		=\|u_k-u\|_{L^2}^2+[u_k-u]_{s+\Log,+}^2\le 2\varepsilon$ for all $k\ge N$.
		This proves that $\mathcal H^{s+\Log}(\mathbb R^n)$ is complete, and therefore
		a Hilbert space.
        \end{proof}
        
We now prove the classical Poincar\'e inequality associated with the fractional-logarithmic Laplacian. The following translation estimate will be used to derive the Poincar\'e inequality from a finite-range fractional seminorm.

\begin{lemma}\label{lem:tran}
	Let \(0<s<1\) and \(\delta>0\). Then there exists
	\(C=C(n,s,\delta)>0\) such that, for every \(u\in L^2(\mathbb R^n)\) with
	\[
	\iint_{\mathbb R^n\times\mathbb R^n}
	\frac{|u(x)-u(y)|^2}{|x-y|^{n+2s}}
	\mathbf 1_{\{|x-y|\le \delta\}}\,dx\,dy<\infty,
	\]
	and every \(h\in\mathbb R^n\) with \(|h|<\delta/2\), one has
	\[
	\|u(\cdot+h)-u\|_{L^2(\mathbb R^n)}^2
	\le
	C |h|^{2s}
	\iint_{\mathbb R^n\times\mathbb R^n}
	\frac{|u(x)-u(y)|^2}{|x-y|^{n+2s}}
	\mathbf 1_{\{|x-y|\le \delta\}}\,dx\,dy .
	\]
\end{lemma}

\begin{proof}
	Set
	\[
	E_\delta(u):=
	\iint_{\mathbb R^n\times\mathbb R^n}
	\frac{|u(x)-u(y)|^2}{|x-y|^{n+2s}}
	\mathbf 1_{\{|x-y|\le \delta\}}\,dx\,dy .
	\]
	Writing \(z=y-x\), we have
	\[
	E_\delta(u)
	=
	\int_{|z|\le\delta}
	\frac{\|u(\cdot+z)-u\|_{L^2(\mathbb R^n)}^2}{|z|^{n+2s}}\,dz=
	\int_{\mathbb R^n}
	m_\delta(\xi)|\widehat u(\xi)|^2\,d\xi , 
	\]
	where
	\[
	m_\delta(\xi)
	:=
	\int_{|z|\le\delta}
	\frac{|e^{i z\cdot \xi}-1|^2}{|z|^{n+2s}}\,dz .
	\]
	
	We claim that there exists \(c=c(n,s,\delta)>0\) such that
	\[
	m_\delta(\xi)
	\ge
	c
	\begin{cases}
		\delta^{2-2s}|\xi|^2, & |\xi|\le \delta^{-1},\\[1mm]
		|\xi|^{2s}, & |\xi|\ge \delta^{-1}.
	\end{cases}
	\]
	Indeed, if \(|\xi|\le\delta^{-1}\), then \(|z\cdot\xi|\le1\) for
	\(|z|\le\delta\). Since
	\[
	|e^{it}-1|^2=2(1-\cos t)\ge c t^2
	\qquad\text{for } |t|\le1,
	\]
	we get
	\[
	m_\delta(\xi)
	\ge
	c\int_{|z|\le\delta}
	\frac{|z\cdot\xi|^2}{|z|^{n+2s}}\,dz
	=
	c\,\delta^{2-2s}|\xi|^2 .
	\]
	If \(|\xi|\ge\delta^{-1}\), write \(e=\xi/|\xi|\) and use the change of
	variables \(\eta=|\xi|z\). Then
	\[
	m_\delta(\xi)
	=
	|\xi|^{2s}
	\int_{|\eta|\le \delta|\xi|}
	\frac{|e^{i e\cdot \eta}-1|^2}{|\eta|^{n+2s}}\,d\eta .
	\]
	Since \(\delta|\xi|\ge1\), we obtain
	\[
	m_\delta(\xi)
	\ge
	|\xi|^{2s}
	\int_{|\eta|\le1}
	\frac{|e^{i e\cdot \eta}-1|^2}{|\eta|^{n+2s}}\,d\eta
	\ge
	c|\xi|^{2s}.
	\]
	This proves the claim.
	
	 Since
	\[
	|e^{i h\cdot \xi}-1|^2
	\le
	C\min\{1,|h|^2|\xi|^2\},
	\]
	we consider two cases.
	
	If \(|\xi|\le\delta^{-1}\), then
	\[
	|e^{i h\cdot \xi}-1|^2
	\le
	C |h|^2|\xi|^2
	=
	C |h|^{2s}|h|^{2-2s}|\xi|^2 .
	\]
	Since \(|h|<\delta/2\), we have $	|h|^{2-2s}\le C\delta^{2-2s}.$ 
	Hence
	\[
	|e^{i h\cdot \xi}-1|^2
	\le
	C |h|^{2s}\delta^{2-2s}|\xi|^2
	\le
	C |h|^{2s} m_\delta(\xi).
	\]
	
	If \(|\xi|\ge\delta^{-1}\), then using
	\[
	\min\{1,t^2\}\le C_s t^{2s}, \qquad t\ge0,
	\]
	we obtain
	\[
	|e^{i h\cdot \xi}-1|^2
	\le
	C_s |h|^{2s}|\xi|^{2s}
	\le
	C |h|^{2s}m_\delta(\xi).
	\]
	Therefore, for all \(\xi\in\mathbb R^n\), $|e^{i h\cdot \xi}-1|^2
	\le
	C |h|^{2s}m_\delta(\xi).$
	Consequently,
	\[
	\begin{aligned}
		\|u(\cdot+h)-u\|_{L^2(\mathbb R^n)}^2
		&=
		\int_{\mathbb R^n}
		|e^{i h\cdot \xi}-1|^2|\widehat u(\xi)|^2\,d\xi \\
		&\le
		C |h|^{2s}
		\int_{\mathbb R^n}
		m_\delta(\xi)|\widehat u(\xi)|^2\,d\xi =
		C |h|^{2s} E_\delta(u).
	\end{aligned}
	\]
	This proves the estimate.
\end{proof}

The preceding translation estimate yields the following finite-range fractional
Poincar\'e inequality for functions vanishing outside a bounded set.

\begin{lemma}\label{localestimate}
	Let \(0<s<1\), \(\delta>0\), and let \(\Omega\subset\mathbb R^n\) be
	bounded. Then there exists a constant \(C=C(n,s,\delta,\Omega)>0\) such
	that, for every \(u\in L^2(\mathbb R^n)\) with \(u=0\) a.e. in
	\(\Omega^c\), one has
	\[
	\|u\|_{L^2(\mathbb R^n)}^2
	\le
	C
	\iint_{\mathbb R^n\times\mathbb R^n}
	\frac{|u(x)-u(y)|^2}{|x-y|^{n+2s}}
	\mathbf 1_{\{|x-y|\le \delta\}}\,dx\,dy .
	\]
\end{lemma}

\begin{proof}
	We argue by contradiction. Suppose that the assertion is false. Then there
	exists a sequence \(\{u_k\}\subset L^2(\mathbb R^n)\), with
	\(u_k=0\) a.e. in \(\Omega^c\), such that $	\|u_k\|_{L^2(\mathbb R^n)}=1$
	and
	\[
	\iint_{\mathbb R^n\times\mathbb R^n}
	\frac{|u_k(x)-u_k(y)|^2}{|x-y|^{n+2s}}
	\mathbf 1_{\{|x-y|\le \delta\}}\,dx\,dy
	\longrightarrow 0 .
	\]
	Since \(\Omega\) is bounded, there exists \(R>0\) such that $\Omega\subset B_R.$
	Hence each \(u_k\) is supported in \(B_R\).
	
	We first note that the sequence \(\{u_k\}\) is relatively compact in
	\(L^2(B_R)\). Indeed, by the Kolmogorov--Riesz compactness criterion, see \cite[Theorem 2.32]{AdamsFournier}, it
	suffices to prove uniform smallness of translations. For every
	\(|h|<\delta/2\), by Lemma \ref{lem:tran},
	\[
	\|u_k(\cdot+h)-u_k\|_{L^2(\mathbb R^n)}^2
	\le
	C |h|^{2s}
	\iint_{\mathbb R^n\times\mathbb R^n}
	\frac{|u_k(x)-u_k(y)|^2}{|x-y|^{n+2s}}
	\mathbf 1_{\{|x-y|\le \delta\}}\,dx\,dy ,
	\]
	where \(C=C(n,s,\delta)>0\). Since the right-hand side is uniformly bounded
	in \(k\) and tends to \(0\) as \(h\to0\), the Kolmogorov--Riesz criterion
	implies that, up to a subsequence, $u_k\to u$ strongly in $L^2(\mathbb R^n).$
	Here we used the fact that all \(u_k\) are supported in the fixed bounded set
	\(B_R\).
	
	Since \(u_k=0\) a.e. in \(\Omega^c\), the strong \(L^2\)-convergence implies
	that
	\[
	u=0
	\qquad\text{a.e. in }\Omega^c .
	\]
	Moreover, after passing to a further subsequence, we may assume that
	\(u_k(x)\to u(x)\) for a.e. \(x\in\mathbb R^n\). By Fatou's lemma,
	\[
	\begin{aligned}
		&\quad \iint_{\mathbb R^n\times\mathbb R^n}
		\frac{|u(x)-u(y)|^2}{|x-y|^{n+2s}}
		\mathbf 1_{\{|x-y|\le \delta\}}\,dx\,dy  \\
		& \qquad \qquad\le
		\liminf_{k\to\infty}
		\iint_{\mathbb R^n\times\mathbb R^n}
		\frac{|u_k(x)-u_k(y)|^2}{|x-y|^{n+2s}}
		\mathbf 1_{\{|x-y|\le \delta\}}\,dx\,dy
		=0 .
	\end{aligned}
	\]
	Therefore $	u(x)=u(y)$
	for a.e. pair \((x,y)\in\mathbb R^n\times\mathbb R^n\) satisfying
	\(|x-y|\le\delta\). This implies that \(u\) is a.e. constant on
	\(\mathbb R^n\). Since \(u=0\) a.e. in \(\Omega^c\), and \(\Omega^c\) has positive measure,
	the constant must be zero. Hence
	\[
	u\equiv0
	\qquad\text{a.e. in }\mathbb R^n .
	\]
	On the other hand, from the strong convergence in \(L^2(\mathbb R^n)\), we
	get
	\[
	\|u\|_{L^2(\mathbb R^n)}
	=
	\lim_{k\to\infty}\|u_k\|_{L^2(\mathbb R^n)}
	=1,
	\]
	which is a contradiction. The proof is complete.
\end{proof}

\begin{proof}[\bf Proof of Proposition~\ref{prop:Poincare-Hslog-plus11}]
Set
\[
K(z):={\bf k}_{s+\Log,+}(z)
=
c_{n,s}|z|^{-n-2s}(-\ln|z|)_+,
\qquad z\in\mathbb R^n\setminus\{0\}.
\]
Since
\[
(-\ln|z|)_+\ge1
\qquad\text{for }|z|\le e^{-1},
\]
we have
\[
K(z)
\ge
c_{n,s}|z|^{-n-2s}
\mathbf 1_{\{|z|\le e^{-1}\}}.
\]
Consequently,
\begin{align*}
&\iint_{\mathbb R^n\times\mathbb R^n}
\frac{|u(x)-u(y)|^2}{|x-y|^{n+2s}}
\mathbf 1_{\{|x-y|\le e^{-1}\}}\,dx\,dy\\
&\qquad\le
\frac1{c_{n,s}}
\iint_{\mathbb R^n\times\mathbb R^n}
|u(x)-u(y)|^2K(x-y)\,dx\,dy\\
&\qquad=
\frac1{c_{n,s}}[u]_{s+\Log,+}^2.
\end{align*}
Since \(u=0\) a.e. in \(\Omega^c\), Lemma~\ref{localestimate},
applied with \(\delta=e^{-1}\), gives
\[
\|u\|_{L^2(\mathbb R^n)}^2
\le
C
\iint_{\mathbb R^n\times\mathbb R^n}
\frac{|u(x)-u(y)|^2}{|x-y|^{n+2s}}
\mathbf 1_{\{|x-y|\le e^{-1}\}}\,dx\,dy.
\]
Combining the preceding two estimates, we obtain
\[
\|u\|_{L^2(\mathbb R^n)}^2
\le
C [u]_{s+\Log,+}^2,
\]
which proves~\eqref{eq:Poincare-Hslog-plus11}.
\end{proof}

Next, we quantify the relation between the fractional-logarithmic energy and the pure fractional energy, making explicit the effect of the logarithmic factor; this estimate will be a key ingredient in proving compactness even at the critical exponent.

	\begin{lemma}
		\label{lem:Es-controlled}
		Let  $0<r<1$. Then 
		\begin{equation}\label{eq:Es-split-control}
			\mathcal E_s(u,u)
			\le \frac{1}{c_{n,s}\,\ln(1/r)}\,\norm{u}_{\cH^{s+\Log}_{0}(\Omega)}^2
			+ \frac{2}{s}|\mathbb{S}^{n-1}|\,r^{-2s}\,\|u\|_{L^2(\Omega)}^2
		\end{equation}
        for every $u\in \mathcal H^{s+\Log}_0(\Omega)$.
	\end{lemma}
	
	\begin{proof} Let $u\in \mathcal H^{s+\Log}_0(\Omega)$. Then, for $0<r<1$, we split the fractional energy $\mathcal E_s(u,u)$ into two parts:
		 \begin{align*}
		   \mathcal E_s(u,u)
		&=\iint_{|x-y|<r}\frac{(u(x)-u(y))^2}{|x-y|^{n+2s}}\,dx\,dy
		+\iint_{|x-y|\ge r}\frac{(u(x)-u(y))^2}{|x-y|^{n+2s}}\,dx\,dy
		\\[1mm]&=:I_{\rm near}+I_{\rm far}.  
		 \end{align*}

		For $|z|<r$, we have $(-\ln|z|)_+=\ln(1/|z|)\ge \ln(1/r)$, hence
		\[
		|z|^{-n-2s}\le \frac{1}{\ln(1/r)}\,|z|^{-n-2s}(-\ln|z|)_+
		=\frac{1}{c_{n,s}\,\ln(1/r)}\,{\bf k}_{s+\Log,+}(z).
		\]
		Therefore,
		\begin{align*}
		I_{\rm near}&\le \frac{1}{c_{n,s}\,\ln(1/r)}
		\iint_{\mathbb R^n\times\mathbb R^n}(u(x)-u(y))^2\,{\bf k}_{s+\Log,+}(x-y)\,dx\,dy\\
		&=\frac{1}{c_{n,s}\,\ln(1/r)}\,[u]_{s+\Log,+}^2.
		\end{align*}
		For $I_{\rm far}$, we use the integrability of
		$|z|^{-n-2s}$ on $\{|z|\ge r\}$:
		\begin{align*}
		    I_{\rm far}\le& 2\iint_{|x-y|\ge r}\frac{u(x)^2+u(y)^2}{|x-y|^{n+2s}}\,dx\,dy
		\\=&4\int_{\mathbb R^n}u(x)^2\Bigl(\int_{|z|\ge r}|z|^{-n-2s}\,dz\Bigr)\,dx
		=\frac{2}{s}|\mathbb{S}^{n-1}|\,r^{-2s}\,\|u\|_{L^2(\Omega)}^2.
		\end{align*}
		Combining the two estimates yields \eqref{eq:Es-split-control}.
	\end{proof}

We now prove the embedding properties, precisely describing the relation between the fractional and fractional-logarithmic spaces, and showing that the embedding is compact for all \(1\le p\le 2_s^*\).

\begin{proof}[\bf Proof of Proposition~\ref{embedding}]

\medskip
\noindent
\textit{(i)}
Let \(\varepsilon\in(0,1-s)\).
For \(0<r<1\), writing \(r=e^{-t}\) with \(t>0\), we have $r^{2\varepsilon}\ln\frac1r
=
te^{-2\varepsilon t}
\le
\frac{1}{2e\varepsilon}.$
Consequently,
\[
(-\ln r)_+
\le
\frac{1}{2e\varepsilon}\,r^{-2\varepsilon},
\qquad r>0.
\]
It follows that, for every \(z\in\mathbb R^n\setminus\{0\}\),
\[
{\bf k}_{s+\Log,+}(z)
=
c_{n,s}|z|^{-n-2s}(-\ln|z|)_+
\le
\frac{c_{n,s}}{2e\varepsilon}
|z|^{-n-2(s+\varepsilon)}.
\]
Hence, for every \(u\in\mathcal H^{s+\varepsilon}_0(\Omega)\),
\begin{align*}
\|u\|_{\mathcal H^{s+\Log}_0(\Omega)}^2
&=
\iint_{\mathbb R^n\times\mathbb R^n}
\bigl(u(x)-u(y)\bigr)^2
{\bf k}_{s+\Log,+}(x-y)\,dx\,dy\\
&\le
\frac{c_{n,s}}{2e\varepsilon}
\iint_{\mathbb R^n\times\mathbb R^n}
\frac{\bigl(u(x)-u(y)\bigr)^2}
{|x-y|^{n+2(s+\varepsilon)}}\,dx\,dy=
\frac{c_{n,s}}{2e\varepsilon}
\|u\|_{\mathcal H^{s+\varepsilon}_0(\Omega)}^2.
\end{align*}
Therefore, we have the continuous embedding $\mathcal H^{s+\varepsilon}_0(\Omega)
\hookrightarrow
\mathcal H^{s+\Log}_0(\Omega)$.\medskip

Next, let \(u\in\mathcal H^{s+\Log}_0(\Omega)\). We split the fractional
energy as
\begin{align*}
\|u\|_{\mathcal H^s_0(\Omega)}^2
&=
\iint_{\{|x-y|<e^{-1}\}}
\frac{\bigl(u(x)-u(y)\bigr)^2}
{|x-y|^{n+2s}}\,dx\,dy\\
&\hspace{1.5cm}+
\iint_{\{|x-y|\ge e^{-1}\}}
\frac{\bigl(u(x)-u(y)\bigr)^2}
{|x-y|^{n+2s}}\,dx\,dy.
\end{align*}
Since
\[
(-\ln|z|)_+\ge1
\qquad\text{for }0<|z|\le e^{-1},
\]
it follows that
\begin{align*}
&c_{n,s}
\iint_{\{|x-y|<e^{-1}\}}
\frac{\bigl(u(x)-u(y)\bigr)^2}
{|x-y|^{n+2s}}\,dx\,dy\\
&\hspace{1.5cm}\le
\iint_{\mathbb R^n\times\mathbb R^n}
\bigl(u(x)-u(y)\bigr)^2
{\bf k}_{s+\Log,+}(x-y)\,dx\,dy\\
&\hspace{1.5cm}=
\|u\|_{\mathcal H^{s+\Log}_0(\Omega)}^2.
\end{align*}
For the far-field part, using $(a-b)^2\le2a^2+2b^2$
and the fact that \(u=0\) a.e. in \(\Omega^c\), we obtain
\begin{align*}
&\iint_{\{|x-y|\ge e^{-1}\}}
\frac{\bigl(u(x)-u(y)\bigr)^2}
{|x-y|^{n+2s}}\,dx\,dy\\
&\hspace{1.5cm}\le
4\int_{\mathbb R^n}u(x)^2
\left(
\int_{\{|z|\ge e^{-1}\}}
|z|^{-n-2s}\,dz
\right)dx\\
&\hspace{1.5cm}=
\frac{2e^{2s}}{s}|\mathbb S^{n-1}|
\|u\|_{L^2(\Omega)}^2.
\end{align*}
Consequently,
\begin{align*}
c_{n,s}\|u\|_{\mathcal H^s_0(\Omega)}^2
&\le
\|u\|_{\mathcal H^{s+\Log}_0(\Omega)}^2
+
\frac{2c_{n,s}e^{2s}}{s}
|\mathbb S^{n-1}|
\|u\|_{L^2(\Omega)}^2.
\end{align*}
By Proposition~\ref{prop:Poincare-Hslog-plus11}, we obtain that
\[
\|u\|_{\mathcal H^s_0(\Omega)}^2
\le
C\|u\|_{\mathcal H^{s+\Log}_0(\Omega)}^2,
\]
which proves the continuous embedding $\mathcal H^{s+\Log}_0(\Omega)
\hookrightarrow
\mathcal H^s_0(\Omega).$
\medskip

\noindent
\textit{(ii)}
By part~{\rm (i)}, the embedding $\mathcal H^{s+\Log}_0(\Omega)
\hookrightarrow
\mathcal H^s_0(\Omega)$
is continuous. Since the embedding 
\(\mathcal H^s_0(\Omega)\hookrightarrow L^q(\Omega)\) is compact for every
\(1\le q<2_s^*\), it remains only to prove the compactness of the
critical embedding
\[
\mathcal H^{s+\Log}_0(\Omega)
\hookrightarrow
L^{2_s^*}(\Omega),
\qquad
2_s^*=\frac{2n}{n-2s}.
\]

We first establish a localized estimate. Fix \(x_0\in\mathbb R^n\) and
\(r\in(0,1)\). Choose
\(\eta\in C_c^\infty(B_{2r}(x_0))\) such that
\[
0\le\eta\le1,
\qquad
\eta\equiv1\quad\text{in }B_r(x_0),
\]
and
\[
|\eta(x)-\eta(y)|
\le
C\min\left\{1,\frac{|x-y|}{r}\right\},
\qquad x,y\in\mathbb R^n,
\]
where \(C>0\) is independent of \(r\) and \(x_0\). For
\(u\in\mathcal H^{s+\Log}_0(\Omega)\), set $v:=\eta u.$
We split $\mathcal E_s(v,v)
=
I_{\mathrm{near}}+I_{\mathrm{far}},$
where
\[
I_{\mathrm{near}}
:=
\iint_{\{|x-y|<r\}}
\frac{|v(x)-v(y)|^2}{|x-y|^{n+2s}}\,dx\,dy
\]
and
\[
I_{\mathrm{far}}
:=
\iint_{\{|x-y|\ge r\}}
\frac{|v(x)-v(y)|^2}{|x-y|^{n+2s}}\,dx\,dy.
\]

For the near-field part, we use
\[
v(x)-v(y)
=
\eta(x)\bigl(u(x)-u(y)\bigr)
+
u(y)\bigl(\eta(x)-\eta(y)\bigr),
\]
which gives
\[
|v(x)-v(y)|^2
\le
2\eta(x)^2|u(x)-u(y)|^2
+
2|u(y)|^2|\eta(x)-\eta(y)|^2.
\]
If \(|x-y|<r\) and \(\eta(x)\neq0\), then
\(x\in B_{2r}(x_0)\) and \(y\in B_{3r}(x_0)\). Hence
\begin{align*}
I_{\mathrm{near}}
&\le
2\iint_{\substack{
B_{3r}(x_0)\times B_{3r}(x_0)\\
|x-y|<r}}
\frac{|u(x)-u(y)|^2}{|x-y|^{n+2s}}\,dx\,dy\\
&\quad+
C r^{-2}
\iint_{\substack{|x-y|<r\\
\eta(x)\neq\eta(y)}}
\frac{|u(y)|^2|x-y|^2}{|x-y|^{n+2s}}\,dx\,dy.
\end{align*}

For \(|z|<r<1\), we have $(-\ln|z|)_+
\ge
\ln\frac1r.$
Therefore,
\[
|z|^{-n-2s}
\le
\frac{1}{c_{n,s}\ln(1/r)}
{\bf k}_{s+\Log,+}(z),
\qquad |z|<r.
\]
It follows that
\begin{align*}
&\iint_{\substack{
B_{3r}(x_0)\times B_{3r}(x_0)\\
|x-y|<r}}
\frac{|u(x)-u(y)|^2}{|x-y|^{n+2s}}\,dx\,dy\le
\frac{1}{c_{n,s}\ln(1/r)}
\mathcal E_+\bigl(u;B_{3r}(x_0)\bigr),
\end{align*}
where
\[
\mathcal E_+\bigl(u;B_{3r}(x_0)\bigr)
:=
\iint_{B_{3r}(x_0)\times B_{3r}(x_0)}
|u(x)-u(y)|^2
{\bf k}_{s+\Log,+}(x-y)\,dx\,dy.
\]

Moreover, if \(|x-y|<r\) and
\(\eta(x)\neq\eta(y)\), then \(y\in B_{3r}(x_0)\). Thus,
using polar coordinates,
\begin{align*}
&Cr^{-2}
\iint_{\substack{|x-y|<r\\
\eta(x)\neq\eta(y)}}
\frac{|u(y)|^2|x-y|^2}{|x-y|^{n+2s}}\,dx\,dy\\
&\qquad\le
Cr^{-2}
\int_{B_{3r}(x_0)}|u(y)|^2
\left(
\int_{|z|<r}|z|^{2-n-2s}\,dz
\right)dy\\
&\qquad\le
Cr^{-2s}
\|u\|_{L^2(B_{3r}(x_0))}^2.
\end{align*}
Consequently,
\begin{equation}\label{eq:near-local-estimate}
I_{\mathrm{near}}
\le
\frac{C}{\ln(1/r)}
\mathcal E_+\bigl(u;B_{3r}(x_0)\bigr)
+
Cr^{-2s}
\|u\|_{L^2(B_{3r}(x_0))}^2.
\end{equation}

For the far-field part, using
\((a-b)^2\le2a^2+2b^2\), we obtain
\begin{align*}
I_{\mathrm{far}}
&\le
4\int_{\mathbb R^n}|v(x)|^2
\left(
\int_{|z|\ge r}|z|^{-n-2s}\,dz
\right)dx\le
Cr^{-2s}\|u\|_{L^2(B_{2r}(x_0))}^2.
\end{align*}
Combining this with~\eqref{eq:near-local-estimate}, we obtain
\begin{equation}\label{eq:localized-Hs-estimate}
\mathcal E_s(v,v)
\le
\frac{C}{\ln(1/r)}
\mathcal E_+\bigl(u;B_{3r}(x_0)\bigr)
+
Cr^{-2s}
\|u\|_{L^2(B_{3r}(x_0))}^2.
\end{equation}

By the fractional Sobolev inequality on \(\mathbb R^n\), $\|v\|_{L^{2_s^*}(\mathbb R^n)}^2
\le
C\mathcal E_s(v,v).$
Since \(\eta\equiv1\) on \(B_r(x_0)\), it follows from
\eqref{eq:localized-Hs-estimate} that
\begin{equation}\label{eq:local-Lp}
\|u\|_{L^{2_s^*}(B_r(x_0))}^2
\le
\frac{C}{\ln(1/r)}
\mathcal E_+\bigl(u;B_{3r}(x_0)\bigr)
+
Cr^{-2s}
\|u\|_{L^2(B_{3r}(x_0))}^2.
\end{equation}

We now prove compactness. Let \((u_k)\) be a bounded sequence in
\(\mathcal H^{s+\Log}_0(\Omega)\). Since this space is Hilbert, after
passing to a subsequence, there exists
\(u\in\mathcal H^{s+\Log}_0(\Omega)\) such that
\[
u_k\rightharpoonup u
\qquad\text{weakly in }
\mathcal H^{s+\Log}_0(\Omega).
\]
By part~{\rm (i)}, the sequence is bounded in
\(\mathcal H^s_0(\Omega)\). The fractional Rellich compactness theorem
therefore yields, after passing to a further subsequence,
\[
u_k\longrightarrow u
\qquad\text{strongly in }L^2(\Omega).
\]

Set $w_k:=u_k-u.$
Then \((w_k)\) is bounded in
\(\mathcal H^{s+\Log}_0(\Omega)\), and
\[
w_k\longrightarrow0
\qquad\text{strongly in }L^2(\Omega).
\]
For each fixed \(r\in(0,1)\), choose finitely many points
\(x_1,\ldots,x_N\) such that $\Omega\subset\bigcup_{j=1}^N B_r(x_j),$
and such that the family
\(\{B_{3r}(x_j)\}_{j=1}^N\) has overlap bounded by a constant depending
only on \(n\). Since \(2/2_s^*<1\), we have
\begin{align*}
\|w_k\|_{L^{2_s^*}(\Omega)}^2
&\le
\sum_{j=1}^N
\|w_k\|_{L^{2_s^*}(B_r(x_j))}^2.
\end{align*}
Applying~\eqref{eq:local-Lp} to \(w_k\) and summing over \(j\), we get
\begin{align*}
\|w_k\|_{L^{2_s^*}(\Omega)}^2
&\le
\frac{C}{\ln(1/r)}
\sum_{j=1}^N
\mathcal E_+\bigl(w_k;B_{3r}(x_j)\bigr)+
Cr^{-2s}
\sum_{j=1}^N
\|w_k\|_{L^2(B_{3r}(x_j))}^2.
\end{align*}
By the bounded-overlap property, we obtain that
\begin{equation}\label{eq:global-critical-estimate}
\|w_k\|_{L^{2_s^*}(\Omega)}^2
\le
\frac{C}{\ln(1/r)}
\mathcal E_+(w_k,w_k)
+
Cr^{-2s}\|w_k\|_{L^2(\Omega)}^2.
\end{equation}

Since \((w_k)\) is bounded in
\(\mathcal H^{s+\Log}_0(\Omega)\), there exists \(M>0\) such that
\[
\mathcal E_+(w_k,w_k)\le M
\qquad\text{for all }k.
\]
Given \(\varepsilon>0\), choose \(r\in(0,1)\) sufficiently small so that $\frac{CM}{\ln(1/r)}<\varepsilon.$
Keeping this \(r\) fixed and letting \(k\to\infty\) in
\eqref{eq:global-critical-estimate}, we obtain
\[
\limsup_{k\to\infty}
\|w_k\|_{L^{2_s^*}(\Omega)}^2
\le
\varepsilon,
\]
because \(w_k\to0\) in \(L^2(\Omega)\). Since \(\varepsilon>0\) is
arbitrary,
\[
u_k\longrightarrow u
\qquad\text{strongly in }L^{2_s^*}(\Omega).
\]
Thus, $\mathcal H^{s+\Log}_0(\Omega)
\hookrightarrow
L^{2_s^*}(\Omega)$
is compact. Finally, since \(\Omega\) is bounded, for every
\(1\le p\le2_s^*\),
\[
\|u_k-u\|_{L^p(\Omega)}
\le
|\Omega|^{\frac1p-\frac1{2_s^*}}
\|u_k-u\|_{L^{2_s^*}(\Omega)}.
\]
Hence $u_k\longrightarrow u$ strongly in $L^p(\Omega)$
for every \(1\le p\le2_s^*\). This proves the claim.
\end{proof}

We now establish several basic estimates and positivity properties of the
energy form associated with the fractional-logarithmic Laplacian.

\begin{proof}[\bf Proof of Proposition \ref{embedding-1}:]

(i) Let $u\in L^{2}(\Omega)$. Since 
	\[\int_1^{\infty}\frac{\ln r}{r^{2s+1}}dr=\frac{1}{4s^2},\]
    one sees that
	\begin{align*}
		\mathcal{E}_{-}(u,u)&=  c_{n,s}\int_{\R^n}\int_{\R^n} \frac{\big(u(x)-u(y)\big)^2}{  |x-y|^{n+2s} }  (-\ln |x-y|\big)_{-} dx dy
		\\[1mm] &\le 4c_{n,s} \int_{\R^n} u(x)^2 \int_{\R^n}   \frac{ (\ln |x-y|\big)_+}{|x-y|^{n+2s}}   dy dx 
		\\[1mm] &=4 c_{n,s} \left|\mathbb{S}^{n-1}\right| \|u\|_{L^2(\Omega)}^2 \int_1^{\infty}\frac{\ln r}{r^{2s+1}}dr
		\\[1mm] &=\frac{1}{s^2} c_{n,s} \left|\mathbb{S}^{n-1}\right| \|u\|_{L^2(\Omega)}^2, 
	\end{align*}
    showing that \eqref{eq:estimate-of-E-} holds.

	(ii)  Now, let $u$, $v\in \mathcal H^{s+\Log}_0(\Omega)$. Then, by Cauchy-Schwarz's inequality, one sees that
    \begin{align*}
        \abs{\cE_{s+\Log}(u,v)} &\le \abs{\mathcal{E}_{+}(u,v)}+\abs{\mathcal{E}_{-}(u,v)}+\frac{\abs{b_{n,s}}c_{n,s}}{2}
        \abs{\mathcal{E}_s(u,v)}\\
        &\le \norm{u}_{\mathcal H^{s+\Log}_0(\Omega)}\,\norm{v}_{\mathcal H^{s+\Log}_0(\Omega)}+\sqrt{\mathcal{E}_{-}(u,u)}\sqrt{\mathcal{E}_{-}(v,v)}\\
        &\hspace{3cm}+\frac{\abs{b_{n,s}}c_{n,s}}{2}
        \norm{u}_{\mathcal H^{s}_0(\Omega)}\,\norm{v}_{\mathcal H^{s}_0(\Omega)}
    \end{align*} 
    Now, by the upper estimate~\eqref{eq:estimate-of-E-} on $\mathcal{E}_{-}(u,u)$, one gets that
    \begin{align*}
        \abs{\cE_{s+\Log}(u,v)} &\le \norm{u}_{\mathcal H^{s+\Log}_0(\Omega)}\,\norm{v}_{\mathcal H^{s+\Log}_0(\Omega)}+\frac{1}{s^2} c_{n,s} \left|\mathbb{S}^{n-1}\right| \|u\|_{L^2(\Omega)}\,\|v\|_{L^2(\Omega)}\\
        &\hspace{3cm}+\frac{\abs{b_{n,s}}c_{n,s}}{2}
        \norm{u}_{\mathcal H^{s}_0(\Omega)}\,\norm{v}_{\mathcal H^{s}_0(\Omega)}
    \end{align*}
    Thus, by the embedding~\eqref{eq:embeddings-of-frac-log} of $\mathcal H^{s+\Log}_0(\Omega)\hookrightarrow \mathcal H^{s}_0(\Omega)$, we prove the desired result.
    
    (iii) Recall that for $u\in \mathcal H^{s+\Log}_0(\Omega)$ one has $u\equiv 0$ a.e.\ in $\Omega^c$ and
	\begin{align*}
		\mathcal E_{+}(u,u)-\mathcal E_{-}(u,u)
		&=c_{n,s}\iint_{\mathbb R^n\times\mathbb R^n}
		\frac{(u(x)-u(y))^2}{|x-y|^{n+2s}}\big(-\ln|x-y|\big)\,dx\,dy \\
		&=c_{n,s}\iint_{\Omega\times\Omega}
		\frac{(u(x)-u(y))^2}{|x-y|^{n+2s}}\big(-\ln|x-y|\big)\,dx\,dy \\
		&\quad +2c_{n,s}\int_{\Omega}u(x)^2\left(\int_{\Omega^c}
		\frac{-\ln|x-y|}{|x-y|^{n+2s}}\,dy\right)\,dx.
	\end{align*}
	
	Fix $x\in\Omega$ and write $D:=\operatorname{diam}(\Omega)< 1$. Since $|x-z|\le D$ for every $z\in\Omega$,
	we have the inclusion $\Omega\subset B_D(x)$, hence $	\Omega^c \supset \mathbb R^n\setminus B_D(x).$
	We split the exterior integral according to the sign of $-\ln|x-y|$:
	\begin{align*}
		\int_{\Omega^c}\frac{-\ln|x-y|}{|x-y|^{n+2s}}\,dy
		&=
		\int_{\Omega^c\cap\{|x-y|<1\}}\frac{-\ln|x-y|}{|x-y|^{n+2s}}\,dy
		-\int_{\Omega^c\cap\{|x-y|>1\}}\frac{\ln|x-y|}{|x-y|^{n+2s}}\,dy.
	\end{align*}
	Note that $\Omega^c\cap\{|x-y|<1\}\supset\{D<|x-y|<1\}$,  we obtain
	\[
	\int_{\Omega^c\cap\{|x-y|<1\}}\frac{-\ln|x-y|}{|x-y|^{n+2s}}\,dy
	\ \ge\
	\int_{D<|x-y|<1}\frac{-\ln|x-y|}{|x-y|^{n+2s}}\,dy
	=|\mathbb S^{n-1}|\int_{D}^{1}\frac{-\ln r}{r^{2s+1}}\,dr.
	\]
	Moreover, since $\Omega^c\cap\{|x-y|>1\}\subset\{|x-y|>1\}$ and the integrand is nonnegative,
	\[
	\int_{\Omega^c\cap\{|x-y|>1\}}\frac{\ln|x-y|}{|x-y|^{n+2s}}\,dy
	\ \le\
	\int_{|x-y|>1}\frac{\ln|x-y|}{|x-y|^{n+2s}}\,dy
	=|\mathbb S^{n-1}|\int_{1}^{\infty}\frac{\ln r}{r^{2s+1}}\,dr.
	\]
	Moreover, since $D<1,$
	\[\iint_{\Omega\times\Omega}
	\frac{(u(x)-u(y))^2}{|x-y|^{n+2s}}\big(-\ln|x-y|\big)\,dx\,dy\ge 0.\]
	Thus, we infer
	\begin{align*}
		\mathcal E_{+}(u,u)-\mathcal E_{-}(u,u) \ge&\
		2c_{n,s}|\mathbb S^{n-1}|\Bigg(\int_{D}^{1}\frac{-\ln r}{r^{2s+1}}\,dr
		-\int_{1}^{\infty}\frac{\ln r}{r^{2s+1}}\,dr\Bigg)\|u\|_{L^2(\Omega)}^2\\=&\frac{1}{s}c_{n,s}|\mathbb S^{n-1}|D^{-2s}(\ln \frac{1}{D}-\frac{1}{2s})\|u\|_{L^2(\Omega)}^2.
	\end{align*}
	
	(iv) For $\operatorname{diam}\Omega<e^{-\frac{1}{2s}}$ and $b_{n,s}\geq 0$, it is obvious that $\cE_{s+\Log}(u,u)\ge 0$ and $\cE_{s+\Log}(u,u)= 0$ implies $u=0.$ By the proof of (ii), we obtain
	\[\cE_{+}(|u|,|u|)-\cE_{-}(|u|,|u|)\le \cE_{+}(u,u)-\cE_{-}(u,u).\]
	Therefore, 
	\begin{align*}
		0\le \cE_{s+\Log}(|u|,|u|)  &=   \mathcal{E}_{+}(|u|,|u|)-\mathcal{E}_{-}(|u|,|u|) +\frac{b_{n,s}c_{n,s}}{2} \cE_{s}(|u|,|u|)
		\\[1mm] &\leq  \cE_{+}(u,u)-\cE_{-}(u,u)+\frac{b_{n,s}c_{n,s}}{2} \cE_{s}(u,u). 
	\end{align*}
	Moreover, equality holds in (\ref{eq:modulus-invariance}) if and only if $u$ does not change sign. 
\end{proof}

	\subsection{Existence and Regularity of Solutions}
	Next, we study the existence, uniqueness, and regularity properties of weak
	solutions to the nonlocal Poisson problem \eqref{eq:poisson}, governed by the \emph{Schrödinger 
	fractional--logarithmic operator} 
    \begin{displaymath}
     (-\Delta)^{s+\Log}+V   
    \end{displaymath}
    for given $V\in L^{\infty}(\Omega)$. We begin by applying the Lax--Milgram theorem to prove existence (and uniqueness) of weak solutions, and then apply a Moser iteration to establish \(L^\infty\)-regularity.

\begin{proof}[\bf Proof of Theorem \ref{thm:LM-general-b-refined}]

     (i)	Let $V\in L^{\infty}(\Omega)$. Then, we consider the bilinear form
		\[
		a(u,\varphi):=\cE_{s+\Log}(u,\varphi)+\int_\Omega V(x)\,u(x)\,\varphi(x)\,dx
		\]
        for every $u$, $\varphi\in \mathcal H^{s+\Log}_0(\Omega)$.
		It follows from Proposition~\ref{embedding-1} (ii) and since $V\in L^{\infty}(\Omega)$ that the bilinear form $a$ is continuous on
		$\mathcal H^{s+\Log}_0(\Omega)$.
		
		Next we prove that $a$ is coercive. Let $u\in \mathcal H^{s+\Log}_0(\Omega)$. Then,
		\begin{align*}
			a(u,u)
			&=[u]_{s+\Log,+}^2-\cE_-(u,u)+\frac{b_{n,s}c_{n,s}}{2}\,\cE_s(u,u)
			+\int_\Omega V\,u^2\,dx\\
			&\ge [u]_{s+\Log,+}^2-\cE_-(u,u)-\frac{|b_{n,s}|c_{n,s}}{2}\,\cE_s(u,u)
			+\big(\operatorname{essinf}_\Omega V\big)\|u\|_{L^2(\Omega)}^2.
		\end{align*}
		Using Proposition~\ref{embedding-1} (i) we have
		\[
		\cE_-(u,u)\le \frac{c_{n,s}|\mathbb S^{n-1}|}{s^2}\,\|u\|_{L^2(\Omega)}^2,
		\]
		and by Lemma~\ref{lem:Es-controlled} together with
		$\|u\|_{\mathcal H^{s+\Log}_0(\Omega)}=[u]_{s+\Log,+}$ we obtain
		\[
		\cE_s(u,u)
		\le \frac{1}{c_{n,s}\ln(1/r)}\,\|u\|_{\mathcal H^{s+\Log}_0(\Omega)}^2
		+\frac{2}{s}|\mathbb S^{n-1}|\,r^{-2s}\,\|u\|_{L^2(\Omega)}^2.
		\]
		Plugging these estimates into the previous inequality yields
		\begin{align*}
			a(u,u)
			&\ge
			\Bigl(1-\frac{|b_{n,s}|}{2\,\ln(1/r)}\Bigr)\|u\|_{\mathcal H^{s+\Log}_0(\Omega)}^2\\
			&\quad +\Bigl(\operatorname{essinf}_\Omega V
			-\frac{c_{n,s}|\mathbb S^{n-1}|}{s^2}
			-\frac{|b_{n,s}|c_{n,s}}{s}\,|\mathbb S^{n-1}|\,r^{-2s}\Bigr)\|u\|_{L^2(\Omega)}^2\\
			&=
			\alpha_r\,\|u\|_{\mathcal H^{s+\Log}_0(\Omega)}^2\\
	&\hspace{1cm}+\Bigl(\operatorname{essinf}_\Omega V
			-\frac{c_{n,s}|\mathbb S^{n-1}|}{s^2}
			-\frac{|b_{n,s}|c_{n,s}}{s}\,|\mathbb S^{n-1}|\,r^{-2s}\Bigr)\|u\|_{L^2(\Omega)}^2.
		\end{align*}
		By assumption \eqref{eq:w-cond-refined}, 
		\[
		a(u,u)\ge \alpha_r
		\|u\|_{\mathcal H^{s+\Log}_0(\Omega)}^2.
		\]
		Therefore, the Lax--Milgram theorem yields a unique $u\in \mathcal H^{s+\Log}_0(\Omega)$ such that
		$a(u,\varphi)=\langle f,\varphi\rangle$ for all $\varphi\in \mathcal H^{s+\Log}_0(\Omega)$,
		which is exactly the weak formulation of \eqref{eq:poisson}.
		The a priori bound follows from the standard Lax--Milgram estimate.

\medskip
\noindent

(ii)
Let \(u\in\mathcal H^{s+\Log}_0(\Omega)\) be a weak solution of
\eqref{eq:poisson}. Lemma~\ref{lem:Es-controlled} implies that $u\in\mathcal H^s(\mathbb R^n)$.
If $2s>n$, the Sobolev embedding into $L^\infty$ proves the assertion directly.
If $2s=n$, the embedding into every finite $L^p(\Omega)$ follows by applying the fractional Sobolev inequality at orders below $s$; choose a finite $p>2q'$.
If $2s<n$, set $p:=2_s^*$, so that again $p>2q'$.
Here $q'$ is the conjugate exponent of $q$, with $q'=1$ when $q=\infty$.
In both remaining cases, $\|v\|_{L^p(\Omega)}^2\le C\mathcal E_+(v,v)$ for $v\in\mathcal H^{s+\Log}_0(\Omega)$.
For \(\alpha>1\) and \(R>0\), define
\[
\varphi_R(t)
=
\begin{cases}
t^\alpha,
&0\le t<R,\\[2mm]
\alpha R^{\alpha-1}t-(\alpha-1)R^\alpha,
&t\ge R.
\end{cases}
\]
Then \(\varphi_R\) is nonnegative, convex, nondecreasing and Lipschitz,
with \(\varphi_R(0)=0\). Define $\Phi_R(t):=\varphi_R(|t|),\,t\in\mathbb R,$ and
\[
\Psi_R(t)
:=
\begin{cases}
\operatorname{sgn}(t)\,
\varphi_R(|t|)\varphi_R'(|t|),
&t\neq0,\\[1mm]
0,
&t=0.
\end{cases}
\]
Equivalently,
\[
\Psi_R(t)=\Phi_R(t)\Phi_R'(t)
\qquad\text{for a.e. }t\in\mathbb R.
\]
Both \(\Phi_R\) and \(\Psi_R\) are Lipschitz continuous and vanish at
the origin. Therefore,
\[
\Phi_R(u),\ \Psi_R(u)
\in\mathcal H^{s+\Log}_0(\Omega),
\]
and \(\Psi_R(u)\) is an admissible test function in the weak
formulation.

We first record some elementary estimates. Since \(\varphi_R\) is the
tangent-line truncation of the convex function \(t\mapsto t^\alpha\),
we have
\begin{equation}\label{eq:Phi-Psi-pointwise}
0\le \Phi_R(t)\le |t|^\alpha,
\qquad
|\Psi_R(t)|\le \alpha |t|^{2\alpha-1},
\qquad
0\le t\Psi_R(t)\le \alpha |t|^{2\alpha}
\end{equation}
for every \(t\in\mathbb R\).

Moreover, for every \(a,b\in\mathbb R\),
\begin{equation}\label{eq:chain-moser-signed}
\bigl(\Phi_R(a)-\Phi_R(b)\bigr)^2
\le
(a-b)\bigl(\Psi_R(a)-\Psi_R(b)\bigr).
\end{equation}
Indeed, assume that \(a>b\). Since \(\Phi_R\) and \(\Psi_R\) are
absolutely continuous, the Cauchy--Schwarz inequality gives
\[
\bigl(\Phi_R(a)-\Phi_R(b)\bigr)^2
=
\left(\int_b^a\Phi_R'(t)\,dt\right)^2
\le
(a-b)\int_b^a|\Phi_R'(t)|^2\,dt.
\]
For a.e. \(t\in\mathbb R\),
\[
\Psi_R'(t)
=
|\Phi_R'(t)|^2
+
\varphi_R(|t|)\varphi_R''(|t|)
\ge
|\Phi_R'(t)|^2,
\]
where \(\varphi_R''\ge0\) a.e. by convexity. Hence
\[
\int_b^a|\Phi_R'(t)|^2\,dt
\le
\Psi_R(a)-\Psi_R(b),
\]
which proves \eqref{eq:chain-moser-signed}. The other cases follow by
symmetry. Recall that
\[
\mathcal E_{s+\Log}(v,w)
=
\frac12
\iint_{\mathbb R^n\times\mathbb R^n}
\bigl(v(x)-v(y)\bigr)
\bigl(w(x)-w(y)\bigr)
{\bf K}_{s+\Log}(|x-y|)\,dx\,dy,
\]
where
\[
{\bf K}_{s+\Log}(r)
=
c_{n,s}\bigl(b_{n,s}-2\ln r\bigr)r^{-n-2s},
\qquad r>0.
\]

Set $\rho:=e^{-(|b_{n,s}|+1)}\in(0,1).$
For \(0<r<\rho\), we have \(-\ln r>|b_{n,s}|\), and therefore
\begin{align*}
\frac12{\bf K}_{s+\Log}(r)
&=
c_{n,s}
\left(-\ln r+\frac{b_{n,s}}2\right)
r^{-n-2s}\\
&\ge
\frac{c_{n,s}}2(-\ln r)r^{-n-2s}
=
\frac12{\bf k}_{s+\Log,+}(r).
\end{align*}
Thus
\begin{equation}\label{eq:kernel-local-positive}
{\bf K}_{s+\Log}(r)
\ge
{\bf k}_{s+\Log,+}(r),
\qquad 0<r<\rho.
\end{equation}
Furthermore,
\begin{equation}\label{eq:kernel-far-integrability}
\int_{\{|z|\ge\rho\}}
\left(
\left|{\bf K}_{s+\Log}(|z|)\right|
+
{\bf k}_{s+\Log,+}(z)
\right)\,dz
<\infty.
\end{equation}

Using \eqref{eq:chain-moser-signed} and
\eqref{eq:kernel-local-positive}, we obtain
\begin{align*}
&\iint_{\{|x-y|<\rho\}}
\bigl(\Phi_R(u(x))-\Phi_R(u(y))\bigr)^2
{\bf k}_{s+\Log,+}(x-y)\,dx\,dy\\
&\qquad\le
\iint_{\{|x-y|<\rho\}}
\bigl(u(x)-u(y)\bigr)
\bigl(\Psi_R(u(x))-\Psi_R(u(y))\bigr)
{\bf K}_{s+\Log}(|x-y|)\,dx\,dy.
\end{align*}
By adding and subtracting the integral over
\(\{|x-y|\ge\rho\}\), we find
\begin{align}
\mathcal E_+(\Phi_R(u),\Phi_R(u))
&\le
2\mathcal E_{s+\Log}(u,\Psi_R(u))
+\mathcal R_{1,R}+\mathcal R_{2,R},
\label{eq:Eplus-Phi-decomposition}
\end{align}
where
\begin{align*}
\mathcal R_{1,R}
&:=
\left|
\iint_{\{|x-y|\ge\rho\}}
\bigl(u(x)-u(y)\bigr)
\bigl(\Psi_R(u(x))-\Psi_R(u(y))\bigr)
{\bf K}_{s+\Log}(|x-y|)\,dx\,dy
\right|,
\\
\mathcal R_{2,R}
&:=
\iint_{\{|x-y|\ge\rho\}}
\bigl(\Phi_R(u(x))-\Phi_R(u(y))\bigr)^2
{\bf k}_{s+\Log,+}(x-y)\,dx\,dy.
\end{align*}

By \eqref{eq:Phi-Psi-pointwise},
\[
\left|
(a-b)\bigl(\Psi_R(a)-\Psi_R(b)\bigr)
\right|
\le
C\alpha\bigl(|a|^{2\alpha}+|b|^{2\alpha}\bigr),
\qquad a,b\in\mathbb R,
\]
where \(C>0\) is independent of \(R\) and \(\alpha\). Hence,
using \eqref{eq:kernel-far-integrability},
\[
\mathcal R_{1,R}
\le
C\alpha\|u\|_{L^{2\alpha}(\Omega)}^{2\alpha}.
\]
Similarly,
\[
\mathcal R_{2,R}
\le
C\|\Phi_R(u)\|_{L^2(\Omega)}^2
\le
C\|u\|_{L^{2\alpha}(\Omega)}^{2\alpha}.
\]
Consequently,
\begin{equation}\label{eq:Eplus-Phi-weak-energy}
\mathcal E_+(\Phi_R(u),\Phi_R(u))
\le
C\mathcal E_{s+\Log}(u,\Psi_R(u))
+
C\alpha\|u\|_{L^{2\alpha}(\Omega)}^{2\alpha},
\end{equation}
where \(C>0\) is independent of \(R\) and \(\alpha\).

Testing the weak formulation of \eqref{eq:poisson} with
\(\Psi_R(u)\), we obtain
\[
\mathcal E_{s+\Log}(u,\Psi_R(u))
+
\int_\Omega V(x)u\Psi_R(u)\,dx
=
\int_\Omega f\Psi_R(u)\,dx.
\]
By \eqref{eq:Phi-Psi-pointwise} and \(V\in L^\infty(\Omega)\),
\[
\left|
\int_\Omega V(x)u\Psi_R(u)\,dx
\right|
\le
\alpha\|V\|_{L^\infty(\Omega)}
\|u\|_{L^{2\alpha}(\Omega)}^{2\alpha},
\]
and
\[
\left|
\int_\Omega f\Psi_R(u)\,dx
\right|
\le
\alpha\int_\Omega |f|\,|u|^{2\alpha-1}\,dx.
\]
Combining these estimates with
\eqref{eq:Eplus-Phi-weak-energy}, we obtain
\begin{equation}\label{eq:Eplus-Phi-estimate}
\mathcal E_+(\Phi_R(u),\Phi_R(u))
\le
C\alpha\int_\Omega |f|\,|u|^{2\alpha-1}\,dx
+
C\alpha\|u\|_{L^{2\alpha}(\Omega)}^{2\alpha}.
\end{equation}

By the Sobolev embedding for
\(\mathcal H^{s+\Log}_0(\Omega)\),
\[
\|\Phi_R(u)\|_{L^p(\Omega)}^2
\le
C\mathcal E_+(\Phi_R(u),\Phi_R(u)),
\qquad
p\ \text{as chosen above}.
\]
It follows that
\begin{equation}\label{eq:moser-basic-signed}
\|\Phi_R(u)\|_{L^p(\Omega)}^2
\le
C\alpha\int_\Omega |f|\,|u|^{2\alpha-1}\,dx
+
C\alpha\|u\|_{L^{2\alpha}(\Omega)}^{2\alpha}.
\end{equation}

Let \(q'\) denote the conjugate exponent of \(q\), with the convention
\(q'=1\) when \(q=\infty\), and set $k:=2q'.$
Since \(q>\frac{n}{2s}\), we have $2\le k<p.$
Fix \(\alpha>1\) such that $u\in L^{\alpha k}(\Omega).$
Since \(k\ge2\) and \(\Omega\) is bounded, it follows that
\(u\in L^{2\alpha}(\Omega)\). Moreover,
\[
(2\alpha-1)q'
\le
2\alpha q'
=
\alpha k.
\]
H\"older's inequality and the boundedness of \(\Omega\) therefore give
\begin{align*}
\int_\Omega |f|\,|u|^{2\alpha-1}\,dx
&\le
\|f\|_{L^q(\Omega)}
\|u\|_{L^{(2\alpha-1)q'}(\Omega)}^{2\alpha-1}\\
&\le
C\|f\|_{L^q(\Omega)}
\|u\|_{L^{\alpha k}(\Omega)}^{2\alpha-1},
\end{align*}
where \(C>0\) is independent of \(\alpha\). Also, $\|u\|_{L^{2\alpha}(\Omega)}^{2\alpha}
\le
C\|u\|_{L^{\alpha k}(\Omega)}^{2\alpha}.$

For every fixed \(t\ge0\), the function
\(R\mapsto\varphi_R(t)\) is nondecreasing and $\varphi_R(t)\rightarrow t^\alpha$ as $R\to\infty.$
Letting \(R\to\infty\) in
\eqref{eq:moser-basic-signed} and using the monotone convergence
theorem, we obtain
\[
\|u\|_{L^{\alpha p}(\Omega)}^{2\alpha}
\le
C\alpha
\left(
\|u\|_{L^{\alpha k}(\Omega)}^{2\alpha-1}
+
\|u\|_{L^{\alpha k}(\Omega)}^{2\alpha}
\right).
\]
Using $t^{2\alpha-1}\le1+t^{2\alpha},t\ge 0,$
we arrive at
\begin{equation}\label{eq:moser-power-improvement}
\|u\|_{L^{\alpha p}(\Omega)}^{2\alpha}
\le
C\alpha
\left(
1+\|u\|_{L^{\alpha k}(\Omega)}^{2\alpha}
\right),
\end{equation}
where \(C>0\) is independent of \(\alpha\). Consequently,
\begin{equation}\label{eq:moser-iteration-step}
\|u\|_{L^{\alpha p}(\Omega)}
\le
(2C\alpha)^{\frac1{2\alpha}}
\max\left\{
1,\|u\|_{L^{\alpha k}(\Omega)}
\right\}.
\end{equation}

We now iterate \eqref{eq:moser-iteration-step}. Define
\[
\alpha_0:=1,
\qquad
\alpha_{j+1}:=\frac{p}{k}\alpha_j,
\qquad j\ge0.
\]
Then $\alpha_j=\left(\frac pk\right)^j\longrightarrow\infty$
and
\[
\alpha_jk=\alpha_{j-1}p,
\qquad j\ge1.
\]
Since \(u\in\mathcal H^{s+\Log}_0(\Omega)\), the Sobolev embedding
implies $u\in L^p(\Omega).$
Thus \eqref{eq:moser-iteration-step} can be applied successively with
\(\alpha=\alpha_j\), \(j\ge1\).

Set
\[
A_j
:=
\max\left\{
1,\|u\|_{L^{\alpha_jp}(\Omega)}
\right\},
\qquad j\ge0.
\]
Using \(\alpha_jk=\alpha_{j-1}p\), we obtain $A_j
\le
(2C\alpha_j)^{\frac1{2\alpha_j}}A_{j-1},\,j\ge 1,$
after increasing \(C\), if necessary, so that \(2C\ge1\). Iterating,
\[
A_j
\le
A_0
\prod_{m=1}^j
(2C\alpha_m)^{\frac1{2\alpha_m}}.
\]
Since $\alpha_m=\left(\frac pk\right)^m$
and \(p/k>1\), we have
\begin{align*}
\sum_{m=1}^\infty
\frac{\ln(2C\alpha_m)}{2\alpha_m}
&=
\frac12
\sum_{m=1}^\infty
\left(\frac kp\right)^m
\left[
\ln(2C)
+
m\ln\left(\frac pk\right)
\right]
<\infty.
\end{align*}
Hence $\prod_{m=1}^\infty
(2C\alpha_m)^{\frac1{2\alpha_m}}
<\infty,$
and therefore $\sup_{j\ge0}
\|u\|_{L^{\alpha_jp}(\Omega)}
<\infty.$

Let $M
:=
\sup_{j\ge0}
\|u\|_{L^{\alpha_jp}(\Omega)}
<\infty.$
For every \(L>0\), Chebyshev's inequality gives
\[
L
\bigl|\{x\in\Omega:|u(x)|>L\}\bigr|^{\frac1{\alpha_jp}}
\le
\|u\|_{L^{\alpha_jp}(\Omega)}
\le M.
\]
If $\bigl|\{x\in\Omega:|u(x)|>L\}\bigr|>0,$
then letting \(j\to\infty\) yields \(L\le M\). Therefore,
\[
\bigl|\{x\in\Omega:|u(x)|>L\}\bigr|=0
\qquad\text{for every }L>M.
\]
Consequently, $u\in L^\infty(\Omega)$ and $\|u\|_{L^\infty(\Omega)}\le M.$
This completes the proof.
	\end{proof}

	\setcounter{equation}{0}
	\section{The Dirichlet eigenvalue problem}
    In this section, we investigate the Dirichlet eigenvalue problem associated with the fractional-logarithmic operator. We establish the variational framework and derive the min--max characterization of eigenvalues, together with basic spectral properties such as orthogonality and completeness of eigenfunctions, and divergence of the eigenvalue sequence. We then turn to high-frequency behavior and prove a Weyl-type asymptotic law for the counting function and for the \(k\)-th eigenvalue.\medskip

We first establish the basic properties of the Dirichlet eigenvalues.

\begin{proof}[\bf Proof of Theorem \ref{thm:eigen-spectrum}.]

By Lemma~\ref{lem:Es-controlled}, the embedding
$\mathcal H^{s+\Log}_0(\Omega)\hookrightarrow\mathcal H^s_0(\Omega)$ is continuous for every $0<s<1$.
Lemma~\ref{lem:tran} and the bounded support condition then imply compactness in $L^2(\Omega)$.
If $u_j\rightharpoonup u$ in $\mathcal H^{s+\Log}_0(\Omega)$, this compactness and~\eqref{eq:Es-split-control}, first with fixed $r$ and then with $r\downarrow0$, give
$\mathcal E_s(u_j-u,u_j-u)\to0$.
Also, $\mathcal E_-(u_j-u,u_j-u)\to0$ by~\eqref{eq:estimate-of-E-}.
Since $\mathcal E_+$ is the squared Hilbert norm, the functional
\[
\Phi_{s+\Log}(u):=\mathcal E_{s+\Log}(u,u)
\]
is therefore weakly lower semicontinuous.
The remaining argument follows the variational construction in \cite[Theorem~1.4]{CW18}; we provide the details.

(i) Let  $\cM_1:= \{u\in \cH^{s+\Log}_0(\Omega),\, \norm{u}_{L^2(\Omega)}=1\}.$
Choose $r\in(0,1)$ so small that $\ln(1/r)\ge\max\{1,|b_{n,s}|\}$.
Lemma~\ref{lem:Es-controlled} and Proposition~\ref{embedding-1}\textup{(i)} yield a constant $C>0$ such that
\begin{equation}\label{eq:relative-functional-est}
\mathcal E_{s+\Log}(u,u)
\ge \frac12\|u\|_{\mathcal H^{s+\Log}_0(\Omega)}^2
-C\|u\|_{L^2(\Omega)}^2,
\qquad u\in\mathcal H^{s+\Log}_0(\Omega).
\end{equation}
Thus $\lambda_1^{s+\Log}(\Omega):=\inf_{\mathcal M_1}\Phi_{s+\Log}>-\infty$, and every minimizing sequence is bounded in the energy space.
Weak compactness, strong $L^2$ convergence and weak lower semicontinuity imply that this infimum is attained at some $\xi_1\in\mathcal M_1$.
Consequently, there exists a Lagrange multiplier $\lambda\in\mathbb R$ such that
\[
\mathcal E_{s+\Log}(\xi_1,\varphi)
=\lambda\int_\Omega\xi_1\varphi\,dx
\qquad\text{for all }\varphi\in\mathcal H^{s+\Log}_0(\Omega).
\]
	Choosing $\varphi = \xi_1$ yields $\lambda= \lambda_1^{s+\Log}(\Omega)$, hence $\xi_1$ is an eigenfunction of (\ref{eq:eigen}) corresponding to $\lambda_1^{s+\Log}(\Omega)$. 
	
	Next we proceed inductively and assume that $\xi_2,\dots,\xi_k \in \cH^{s+\Log}_0(\Omega)$ and $$\lambda_2^{s+\Log}(\Omega) \le \dots \le \lambda_k^{s+\Log}(\Omega)$$ are already given for some $k \in \N$ with the properties that for $i=2,\dots,k$, the function $\xi_i$ is a minimizer
	of $\cE_{s+\Log}(\cdot,\cdot)$ within the set
	\begin{align*}
		\cM_i&:= \big\{u\in \cH_i(\Omega)\::\: \norm{u}_{L^2(\Omega)}=1\big\} 
		\\[1mm]&= \Big\{u\in \cH^{s+\Log}_0(\Omega)\::\: \norm{u}_{L^2(\Omega)}=1,\: \text{$\int_{\Omega} u \xi_j \,dx =0$ for $j=1,\dots i-1$} \Big\}.
	\end{align*}
    Moreover,
	$$\lambda_i^{s+\Log}(\Omega)= \inf_{u\in \cM_i} \Phi_{s+\Log}(u) = \Phi_{s+\Log}(\xi_i),$$
	and
	\begin{equation}
		\label{eq:inductive-eigenvalue}
		\cE_{s+\Log}(\xi_i, \varphi)=\lambda_i^{s+\Log}(\Omega)\int_\Omega \xi_i \varphi\, dx \qquad \text{for all $\varphi\in  \cH^{s+\Log}_0(\Omega)$.}
	\end{equation}
	We then set
	\begin{align*}
		&\cH_{k+1}(\Omega):= \left\{u\in \cH^{s+\Log}_0(\Omega)\::\: \text{$\int_{\Omega} u \xi_i \,dx =0$ for $i=1,\dots k$} \right\},
	\end{align*}
    and 
    \[\cM_{k+1}:= \big\{u\in \cH_{k+1}(\Omega),\, \norm{u}_{L^2(\Omega)}=1\big\},\quad\lambda_{k+1}^{s+\Log}(\Omega):= \inf_{u\in \cM_{k+1}}\Phi_{s+\Log}(u) .\]
	By the same weak lower semicontinuity argument as above, the value $\lambda_{k+1}^{s+\Log}(\Omega)$ is attained by a function $\xi_{k+1} \in \cM_{k+1}$. Consequently, there exists a Lagrange multiplier $\lambda \in \R$ with the property that
	\begin{equation}
		\label{eq:inductive-eigenvalue-k+1}
		\cE_{s+\Log}(\xi_{k+1}, \varphi)=\lambda \int_\Omega \xi_{k+1} \varphi\, dx \qquad \text{for all $\varphi \in \cH_{k+1}(\Omega)$.}
	\end{equation}
	Choosing $\varphi = \xi_{k+1}$, it yields $\lambda= \lambda_{k+1}^{s+\Log}(\Omega)$. Moreover, for $i=1,\dots,k$, we have, by (\ref{eq:inductive-eigenvalue}) and the definition of $\cH_{k+1}(\Omega)$,
	\begin{align*}
		\cE_{s+\Log}(\xi_{k+1}, \xi_i)= \lambda_i^{s+\Log}(\Omega) \int_{\Omega} \xi_i \xi_{k+1}\,dx 
		= 0 =\lambda_{k+1}^{s+\Log}(\Omega)\int_{\Omega} \xi_{k+1} \xi_i \,dx.
	\end{align*}
	Hence (\ref{eq:inductive-eigenvalue-k+1}) holds with $\lambda= \lambda_{k+1}^{s+\Log}(\Omega)$ for all $\varphi\in \cH^{s+\Log}_0(\Omega)$. 
	
	Therefore, we have constructed an $L^2$-normalized sequence $(\xi_k)_k$ in $\cH^{s+\Log}_0(\Omega)$ and a nondecreasing sequence $(\lambda_k^{s+\Log}(\Omega))_k$ in $\R$ such that property (i) holds and such that $\xi_k$ is an eigenfunction of (\ref{eq:eigen}) corresponding to $\lambda = \lambda_k^{s+\Log}(\Omega)$.
    
    (ii) Supposing by contradiction that $c:= \lim \limits_{k \to \infty}\lambda_k^{s+\Log}(\Omega) <+\infty$, we deduce that
	$\Phi_{s+\Log}(\xi_k) \le c$ for every $k \in \N$. By~\eqref{eq:relative-functional-est}, the sequence $(\xi_k)$ is bounded in $\cH^{s+\Log}_0(\Omega)$, and therefore it
	contains a convergent subsequence $(\xi_{k_j})_j$ in $L^2(\Omega)$ by the compactness of embedding.
	This however is impossible since the functions $\{\xi_{k_j}\}_{j \in \N}$ are $L^2$-orthonormal. Hence (ii) is proved.
	
(iii) First, no nonzero $v\in\mathcal H^{s+\Log}_0(\Omega)$ is $L^2$-orthogonal to every $\xi_k$.
Indeed, after normalizing $\|v\|_{L^2}=1$, such a function belongs to every $\mathcal M_k$, so that $\Phi_{s+\Log}(v)\ge\lambda_k^{s+\Log}(\Omega)$ for all $k$, contradicting (ii).
To prove completeness in $L^2(\Omega)$, let $f\in L^2(\Omega)$ be orthogonal to every $\xi_k$.
Choose $\mu>C$ in~\eqref{eq:relative-functional-est}.
By Lax--Milgram there is a unique $v\in\mathcal H^{s+\Log}_0(\Omega)$ satisfying
\[
\mathcal E_{s+\Log}(v,\varphi)+\mu\int_\Omega v\varphi\,dx
=\int_\Omega f\varphi\,dx
\qquad\text{for all }\varphi\in\mathcal H^{s+\Log}_0(\Omega).
\]
Testing with $\xi_k$ gives $(\lambda_k^{s+\Log}(\Omega)+\mu)\int_\Omega v\xi_k\,dx=0$.
The coefficient is positive by~\eqref{eq:relative-functional-est}, so the preceding observation gives $v=0$.
It follows that $f$ is orthogonal to $\mathcal H^{s+\Log}_0(\Omega)$, which contains $C_c^\infty(\Omega)$ and is dense in $L^2(\Omega)$.
Hence $f=0$, proving that $\{\xi_k\}_{k\in\mathbb N}$ is an orthonormal basis.

	(iv) Let $w \in \cH^{s+\Log}_0(\Omega)$ be a $L^2$-normalized eigenfunction of (\ref{eq:eigen}) corresponding to the eigenvalue $\lambda_1^{s+\Log}(\Omega)$, i.e. we have
	\begin{equation}
		\label{eq:inductive-eigenvalue-k=1}
		\cE_{s+\Log}(w, \varphi)=\lambda_1^{s+\Log}(\Omega) \int_\Omega w \varphi\, dx \qquad \text{for all $\varphi \in \cH^{s+\Log}_0(\Omega)$.}
	\end{equation}
	We show that $w$ does not change sign. Indeed, choosing $\phi= w$ in (\ref{eq:inductive-eigenvalue-k=1}), we see that $w$ is a minimizer of $ \cE_{s+\Log}\big|_{\cM_1}$. On the other hand, we also have $|w| \in \cM_1$ and
	$$
	\cE_{s+\Log}(|w|,|w|)\le \cE_{s+\Log}(w,w) =\lambda_1^{s+\Log}(\Omega). 
	$$
	Hence equality holds by definition of $\lambda_1^{s+\Log}(\Omega)$, and then Proposition~\ref{embedding-1} (iv) implies that $w$ does not change sign. Therefore, we may assume that $\xi_1$ is nonnegative, by Proposition \ref{embedding-1}, we obtain $\lambda_1^{s+\Log}(\Omega)>0.$
   \end{proof}

We now prove the pointwise maximum principle stated above.

\begin{proof}[\bf Proof of Proposition \ref{prop:pointwise-MP}.]
For \(x\in\Omega\), since \(u=0\) in \(\mathbb R^n\setminus\Omega\), the
pointwise representation gives
\[
(-\Delta)^{s+\Log}u(x)
=
\PV \int_{\Omega}
\bigl(u(x)-u(y)\bigr){\bf K}_{s+\Log}(x-y)\,dy
+
\kappa_\Omega(x)u(x),
\]
where
\[	\kappa_\Omega(x)
:=
\int_{\mathbb R^n\setminus\Omega}
{\bf K}_{s+\Log}(x-y)\,dy,\quad x\in\Omega.\]
Moreover, by the diameter assumption, for every \(x,y\in\Omega\) with
\(x\neq y\), we have
\[
0<|x-y|\le \operatorname{diam}(\Omega)<e^{b_{n,s}/2}.
\]
Hence, using the explicit form of \({\bf K}_{s+\Log}\), we obtain
\[
{\bf K}_{s+\Log}(x-y)>0,
\qquad x,y\in\Omega,\ x\neq y.
\]

We next show that the same diameter condition implies $\kappa_\Omega(x)>0,x\in\Omega.$
Indeed, for every \(x\in\Omega\), one has $\Omega\subset B_D(x), \,D:=\operatorname{diam}(\Omega).$
and therefore
\[
\kappa_\Omega(x)
=
\int_{\mathbb R^n\setminus\Omega}
{\bf K}_{s+\Log}(x-y)\,dy
\ge
\int_{\mathbb R^n\setminus B_D(x)}
{\bf K}_{s+\Log}(x-y)\,dy .
\]
Using polar coordinates and the explicit formula
\[
{\bf K}_{s+\Log}(z)
=
c_{n,s}|z|^{-n-2s}\bigl(b_{n,s}-2\ln|z|\bigr),
\]
we compute
\[
\begin{aligned}
	\int_{\mathbb R^n\setminus B_D(x)}
	{\bf K}_{s+\Log}(x-y)\,dy=
	\frac{c_{n,s}|\mathbb S^{n-1}|}{2s}
	D^{-2s}
	\Big(
	b_{n,s}-2\ln D-\frac1s
	\Big).
\end{aligned}
\]
By the diameter condition, $b_{n,s}-2\ln D-\frac1s>0.$
Hence
\[
\int_{\mathbb R^n\setminus B_D(x)}
{\bf K}_{s+\Log}(x-y)\,dy>0.
\]
Therefore, $\kappa_\Omega(x)>0,
\, x\in\Omega.$

We first prove the weak maximum principle. Suppose, by contradiction, that
\(u\) is negative somewhere in \(\Omega\). Since \(u=0\) in
\(\mathbb R^n\setminus\Omega\) and \(u\in C(\mathbb R^n)\), the function \(u\)
attains a negative minimum at some point \(x_0\in\Omega\). Thus $u(x_0)=\min_{\mathbb R^n}u<0.$
In particular, $u(x_0)-u(y)\le0$ for all $y\in\Omega.$
Since \({\bf K}_{s+\Log}(x_0-y)>0\) for \(y\in\Omega\), \(y\neq x_0\), it follows that
\[
\PV \int_{\Omega}
\bigl(u(x_0)-u(y)\bigr){\bf K}_{s+\Log}(x_0-y)\,dy
\le0.
\]
On the other hand,  \(\kappa_\Omega(x_0)>0\) and \(u(x_0)<0\) by the choice of
the negative minimum point, we obtain $\kappa_\Omega(x_0)u(x_0)<0.$
Consequently, $(-\Delta)^{s+\Log}u(x_0)<0,$
which contradicts the assumption $(-\Delta)^{s+\Log}u\ge0$ in $\Omega.$
Therefore, $u\ge0$ in $\Omega.$

We now prove the strong maximum principle. Assume that \(u\not\equiv0\) in
\(\Omega\). By the weak maximum principle just proved, \(u\ge0\) in \(\Omega\).
Suppose, by contradiction, that there exists \(x_0\in\Omega\) such that $u(x_0)=0.$
Using the decomposition above and the fact that \(u(x_0)=0\), we get
\[
(-\Delta)^{s+\Log}u(x_0)
=
-\PV\int_{\Omega}
u(y){\bf K}_{s+\Log}(x_0-y)\,dy.
\]
Since \(u\ge0\) in \(\Omega\) and ${\bf K}_{s+\Log}(x_0-y)>0$ for all $y\in\Omega,\ y\neq x_0,$
we obtain $(-\Delta)^{s+\Log}u(x_0)\le0.$
Together with the assumption $(-\Delta)^{s+\Log}u(x_0)\ge0,$
this implies
\[
\int_{\Omega}
u(y){\bf K}_{s+\Log}(x_0-y)\,dy=0.
\]
Since the integrand is nonnegative and the kernel is strictly positive for
\(y\in\Omega\setminus\{x_0\}\), we conclude that $u=0$ a.e. in $\Omega.$
By continuity, this gives \(u\equiv0\) in \(\Omega\), contradicting the
assumption \(u\not\equiv0\). Hence no such point \(x_0\) exists, and therefore $u>0$ in $\Omega
$, which completes the proof.
\end{proof}

Next, we introduce the Riesz mean of the shifted operator (see \cite{hormander1968riesz}):
\[
{\rm Tr}\big((-\Delta)^{s+\Log}-\Lambda\big)_-
=\sum_{k\in\N}\big(\lambda_k^{s+\Log}(\Omega)-\Lambda\big)_-,
\qquad \Lambda\in\R.
\]

	We first derive an asymptotic upper bound for the Riesz mean of the Dirichlet spectrum.

	\begin{lemma}\label{lm eigenlimit 1}
		Let $\Omega$ be a bounded domain. Then one has that 
		\begin{equation}\label{upp bound-ei1}
			{\rm Tr}\big((-\Delta)^{s+\Log}-\Lambda\big)_-\le   \frac{|\Omega|}{(2\pi)^{n}} \frac{2|\mathbb{S}^{n-1}|}{n(n+2s)}   (s \Lambda)^{1+\frac{n}{2s}} \big( \ln s\Lambda\big)^{-\frac{n}{2s}}\big(1+o(1)\big)
		\end{equation} 
       as $\Lambda \rightarrow \infty$. 
	\end{lemma}

    \begin{proof}
    It follows by \cite[Theorem 2]{Ge} that 
	$${\rm Tr}\big((-\Delta)^{s+\Log}-\Lambda\big)_-\leq \frac{|\Omega|}{(2\pi)^{n}} \int_{\R^n} \big(|\xi|^{2s}\ln|\xi|^2 -\Lambda\big)_-d\xi.  
	$$
	By the theory of Lambert $W$ function, a direct computation shows that for given $\Lambda>0$,  
	$$r_\Lambda^{2s}\ln r_\Lambda^2 =\Lambda \qquad \Longrightarrow \qquad r_\Lambda\leq (s \Lambda)^{\frac1{2s}} 
	\big(\ln \Lambda\big)^{-\frac1{2s}} \Big(1+\frac{1}{2s}\frac{\ln \ln \Lambda}{\ln \Lambda}\big(1+o(1)\big)\Big), $$
	which implies that  for sufficiently large $\Lambda$
	\begin{align*}
		\int_{\R^n} \big(|\xi|^{2s}\ln|\xi|^2 -\Lambda\big)_-d\xi&= \int_{B_{r_\Lambda}} \big(\Lambda-|\xi|^{2s} \ln|\xi|^2 \big)d\xi 
		\\[2mm]&=|\mathbb{S}^{n-1}|\Big(\frac1n \Lambda r_\Lambda^{n}   - \frac2{n+2s} r_\Lambda^{2s+n} \big(\ln r_\Lambda-\frac1{n+2s}\big) \Big)	\\[2mm]&=|\mathbb{S}^{n-1}|\Big(\frac{2s \Lambda r_{\Lambda}^n}{n\left(n+2s\right)}+\frac{2}{\left(n+2s\right)^2}r_{\Lambda}^{2s+n}\Big)
		\\[2mm]&=  \frac{2|\mathbb{S}^{n-1}|}{n(n+2s)}   (s \Lambda)^{1+\frac{n}{2s}} \big( \ln \Lambda\big)^{-\frac{n}{2s}}\big(1+o(1)\big).
	\end{align*}
	The proof ends. 
\end{proof}

With this lemma at hand, we can now complete the proof.

\begin{proof}[\bf Proof of Theorem \ref{thm:weyl-fraclog}.]

 Let $g\in C_c^\infty(\R^n)$ be a real-valued, $L^2$-normalized function with support in 
	$\{x\in \R^n:|x|\leq \frac{\delta}2\}$. For $p\in\R^n$ and $q\in \Omega_\delta$, define the coherent states 
	$$F_{p,q}(x):=e^{{\rm i} p\cdot x } g(x-q), $$
	where $\Omega_\delta=\{x\in\Omega:\,  {\rm dist}(x,\R^n\setminus\Omega)>\delta  \}$.

	From \cite[Theorem 12.8]{LL}, there holds
	\begin{align*}
		{\rm Tr}((-\Delta)^{s+\Log}-\Lambda )_-& \geq (2\pi)^{-n} \int\int_{\R^n\times \Omega_\delta} \big\langle F_{p,q}, \big((-\Delta)^{s+\Log}-\Lambda\big)_-  F_{p,q} \big\rangle \,dp dq
		\\[1mm]&\geq (2\pi)^{-n} \int\int_{\R^n\times \Omega_\delta} \Big(\big\langle F_{p,q}, (-\Delta)^{s+\Log}  F_{p,q} \big\rangle -\Lambda\Big)_-\, dp dq. 
	\end{align*}
	Similar to the proof of \cite[(8)]{Ge},  there exists $C>0$ such that 
	$$\langle F_{p,q}, (-\Delta)^{s+\Log}  F_{p,q} \rangle \leq |p|^{2s}\ln|p|^2+C,$$
	thus
	\begin{align*}
		{\rm Tr}((-\Delta)^{s+\Log}-\Lambda )_-\geq&  \frac{|\Omega_\delta|}{(2\pi)^{n}}\int_{\R^n} \Big(|p|^{2s}\ln|p|^2+C-\Lambda\Big)_- dp\\=&\frac{|\Omega_\delta|}{(2\pi)^{n}} \frac{2|\mathbb{S}^{n-1}|}{n(n+2s)}   (s \Lambda)^{1+\frac{n}{2s}} \big( \ln \Lambda\big)^{-\frac{n}{2s}}\big(1+o(1)\big), 
	\end{align*}
	which, together with (\ref{upp bound-ei1}), implies that 
	\begin{align*} 
		\lim_{\Lambda\to+\infty} {\rm Tr}\big((-\Delta)^{s+\Log}-\Lambda\big)_- \Big(\Lambda ^{-1-\frac{n}{2s}} 
		\big( \ln \Lambda\big)^{\frac{n}{2s}}\Big) 
		=  \frac{|\Omega|}{(2\pi)^{n}} \frac{2|\mathbb{S}^{n-1}|}{n(n+2s)} s^{1+\frac{n}{2s}} =:A_s. 
	\end{align*}
	Let $\cS_0(\Lambda)= \sum_k(\Lambda-\lambda_k^{s+\Log}(\Omega))_+={\rm Tr}\big((-\Delta)^{s+\Log}-\Lambda\big)_-$, then for any $\kappa\in(0, \frac\Lambda2)$, 
	$$\big(\Lambda+\kappa-\lambda_k^{s+\Log}(\Omega)\big)_+-\big(\Lambda -\lambda_k^{s+\Log}(\Omega)\big)_+\geq \kappa \big(\Lambda -\lambda_k^{s+\Log}(\Omega)\big)_+ ^0 =\kappa \cN(\Lambda)$$
	and
	$$\big(\Lambda-\lambda_k^{s+\Log}(\Omega)\big)_+-\big(\Lambda-\kappa -\lambda_k^{s+\Log}(\Omega)\big)_+\leq \kappa \big(\Lambda  -\lambda_k^{s+\Log}(\Omega)\big)_+ ^0 =\kappa \cN(\Lambda),$$
	that is, 
	\begin{align} \label{bnd--1}
		\cS_0(\Lambda)-\cS_0(\Lambda-\kappa) \leq \kappa \cN(\Lambda)\leq \cS_0(\Lambda+\kappa)-\cS_0(\Lambda). 
	\end{align} 

Note that the function $F(\Lambda):=\Lambda^{1+\frac{n}{2s}}(\ln\Lambda)^{-\frac{n}{2s}},\Lambda>1$
is increasing and convex on $(\Lambda_1,\infty)$ for some sufficiently large $\Lambda_1$.
Fix $\epsilon\in(0,\tfrac12)$. By the asymptotic relation in \eqref{bnd--1}, there exists
$\Lambda_2>2\Lambda_1$ such that for all $\Lambda\ge \Lambda_2$,
\[
\Big|\big(\Lambda ^{-1-\frac{n}{2s}}(\ln \Lambda)^{\frac{n}{2s}}\big)\cS_0(\Lambda)-A_s\Big|
\le \epsilon,
\]
equivalently, $(A_s-\epsilon)\,F(\Lambda)\le \cS_0(\Lambda)\le (A_s+\epsilon)\,F(\Lambda).$
Take $\kappa=\sqrt{\epsilon}\,\Lambda$. Then $\Lambda+\kappa=\Lambda(1+\sqrt{\epsilon})\ge \Lambda_1$,
so that $F$ is increasing and convex on the interval $[\Lambda,\Lambda+\kappa]$.
Applying \eqref{bnd--1} and the above bounds yields
\[
\cN(\Lambda)\le \frac1\kappa\Big\{A_s\big(F(\Lambda+\kappa)-F(\Lambda)\big)
+\epsilon\big(F(\Lambda+\kappa)+F(\Lambda)\big)\Big\}.
\]
Since $F$ is convex, we have for any $\kappa>0$,
\[
\frac{F(\Lambda+\kappa)-F(\Lambda)}{\kappa}\le F'(\Lambda+\kappa),
\]
while the monotonicity of $F$ implies $F(\Lambda)\le F(\Lambda+\kappa)$ and hence
\[
\frac{\epsilon}{\kappa}\big(F(\Lambda+\kappa)+F(\Lambda)\big)
\le \frac{2\epsilon}{\kappa}F(\Lambda+\kappa)
=\frac{2\sqrt{\epsilon}}{\Lambda}\,F(\Lambda+\kappa).
\]
Moreover, a direct computation gives
\[
F'(\Lambda)=\frac{n+2s}{2s}\,\Lambda^{\frac{n}{2s}}(\ln\Lambda)^{-\frac{n}{2s}-1}
\Bigl(\ln\Lambda-\frac{n}{n+2s}\Bigr).
\]
Combining these estimates and using $\Lambda+\kappa=\Lambda(1+\sqrt{\epsilon})$, we obtain
\begin{align*}
\cN(\Lambda)
&\le A_s\,F'(\Lambda+\kappa)+\frac{2\sqrt{\epsilon}}{\Lambda}\,F(\Lambda+\kappa)\\[2mm]
&\le \frac{n+2s}{2s}A_s\,[\Lambda(1+\sqrt{\epsilon})]^{\frac{n}{2s}}
\bigl(\ln[\Lambda(1+\sqrt{\epsilon})]\bigr)^{-\frac{n}{2s}-1}
\Bigl(\ln[\Lambda(1+\sqrt{\epsilon})]-\frac{n}{n+2s}\Bigr)\\[2mm]
&\quad +\frac{2\sqrt{\epsilon}}{\Lambda}\,\{\Lambda(1+\sqrt{\epsilon})\}^{1+\frac{n}{2s}}
\bigl(\ln[\Lambda(1+\sqrt{\epsilon})]\bigr)^{-\frac{n}{2s}}.
\end{align*}
and similarly we derive a lower bound.  With the same notation, we obtain
\[
\cN(\Lambda)\ge \frac1\kappa\Big\{A_s\big(F(\Lambda)-F(\Lambda-\kappa)\big)
-\epsilon\big(F(\Lambda)+F(\Lambda-\kappa)\big)\Big\}.
\]
By convexity of $F$, the derivative is increasing, hence
\[
\frac{F(\Lambda)-F(\Lambda-\kappa)}{\kappa}\ge F'(\Lambda-\kappa).
\]
Moreover, by monotonicity $F(\Lambda-\kappa)\le F(\Lambda)$ and thus
\[
\frac{\epsilon}{\kappa}\big(F(\Lambda)+F(\Lambda-\kappa)\big)
\le \frac{2\epsilon}{\kappa}F(\Lambda)
=2\sqrt{\epsilon}\,\Lambda^{\frac{n}{2s}}(\ln\Lambda)^{-\frac{n}{2s}}.
\]
Using the explicit expression
\[
F'(\Lambda)=\frac{n+2s}{2s}\,\Lambda^{\frac{n}{2s}}(\ln\Lambda)^{-\frac{n}{2s}-1}
\Bigl(\ln\Lambda-\frac{n}{n+2s}\Bigr),
\]
and $\Lambda-\kappa=\Lambda(1-\sqrt{\epsilon})$, we arrive at
\begin{align*}
\cN(\Lambda)
&\ge A_s\,F'(\Lambda-\kappa)-2\sqrt{\epsilon}\,\Lambda^{\frac{n}{2s}}(\ln\Lambda)^{-\frac{n}{2s}}\\[2mm]
&\ge \frac{n+2s}{2s}A_s\,[\Lambda(1-\sqrt{\epsilon})]^{\frac{n}{2s}}
\bigl(\ln[\Lambda(1-\sqrt{\epsilon})]\bigr)^{-\frac{n}{2s}-1}
\Bigl(\ln[\Lambda(1-\sqrt{\epsilon})]-\frac{n}{n+2s}\Bigr)\\[2mm]
&\quad-2\sqrt{\epsilon}\,\Lambda^{\frac{n}{2s}}(\ln\Lambda)^{-\frac{n}{2s}}.
\end{align*}
Dividing the upper and lower bounds by $\Lambda^{\frac{n}{2s}}(\ln\Lambda)^{-\frac{n}{2s}}$
and letting $\Lambda\to\infty$, we obtain
\[
\limsup_{\Lambda\to\infty}\Bigl(\Lambda^{-\frac{n}{2s}}(\ln\Lambda)^{\frac{n}{2s}}\cN(\Lambda)\Bigr)
\le \frac{n+2s}{2s}A_s\,(1+o_\epsilon(1)) + C\sqrt{\epsilon},
\]
and
\[
\liminf_{\Lambda\to\infty}\Bigl(\Lambda^{-\frac{n}{2s}}(\ln\Lambda)^{\frac{n}{2s}}\cN(\Lambda)\Bigr)
\ge \frac{n+2s}{2s}A_s\,(1+o_\epsilon(1)) - C\sqrt{\epsilon},
\]
where $o_\epsilon(1)\to0$ as $\Lambda\to\infty$ for each fixed $\epsilon$, and $C$ is independent
of $\Lambda$ and $\epsilon$.  Since $\epsilon\in(0,\tfrac12)$ is arbitrary, letting first
$\Lambda\to\infty$ and then $\epsilon\downarrow0$ yields
\[\lim_{\Lambda\to+\infty}\Bigl(\Lambda^{-\frac{n}{2s}}(\ln\Lambda)^{\frac{n}{2s}}\cN(\Lambda)\Bigr)
=\frac{n+2s}{2s}A_s
=(2\pi)^{-n}s^{\frac{n}{2s}}\omega_n\,|\Omega|.\]

	Taking $\Lambda=\lambda_k^{s+\Log}(\Omega)$, implies that  
\begin{equation}\label{eq:inv-asymp}
\lim_{k\to\infty} k\,
\big(\lambda_k^{s+\Log}(\Omega)\big)^{-\frac{n}{2s}}
\big(\ln \lambda_k^{s+\Log}(\Omega)\big)^{\frac{n}{2s}}
=(2\pi)^{-n}s^{\frac{n}{2s}}\omega_n\,|\Omega|=:C_0.
\end{equation}

Directly from \eqref{eq:inv-asymp}, for any $\delta\in(0,1)$ there exists $k_\delta$ such that
for all $k\ge k_\delta$,
\[
(1-\delta)C_0
\le
k\,
\big(\lambda_k^{s+\Log}(\Omega)\big)^{-\frac{n}{2s}}
\big(\ln \lambda_k^{s+\Log}(\Omega)\big)^{\frac{n}{2s}}
\le
(1+\delta)C_0.
\]
Equivalently,
\[
\Big(\frac{k}{(1+\delta)C_0}\Big)^{\frac{2s}{n}}
\le
\frac{\lambda_k^{s+\Log}(\Omega)}{\ln \lambda_k^{s+\Log}(\Omega)}
\le
\Big(\frac{k}{(1-\delta)C_0}\Big)^{\frac{2s}{n}}.
\]
Taking logarithms in the preceding bounds gives
\[
\ln\lambda_k^{s+\Log}(\Omega)-\ln\ln\lambda_k^{s+\Log}(\Omega)
=\frac{2s}{n}\ln k+O(1).
\]
Since $\ln\ln\lambda_k/\ln\lambda_k\to0$, it follows that
\[
\frac{\ln\lambda_k^{s+\Log}(\Omega)}{\ln k}\longrightarrow\frac{2s}{n}.
\]
Thus~\eqref{eq:inv-asymp} yields
\[
\lim_{k\to\infty}
\frac{\lambda_k^{s+\Log}(\Omega)}{k^{2s/n}\ln k}
=\frac{2s}{n}\,C_0^{-2s/n}.
\]
Finally, since $C_0=(2\pi)^{-n}s^{\frac{n}{2s}}\omega_n|\Omega|$, we compute
\[
C_0^{-\frac{2s}{n}}
=
\Big((2\pi)^{-n}s^{\frac{n}{2s}}\omega_n|\Omega|\Big)^{-\frac{2s}{n}}
=(2\pi)^{2s}\,s^{-1}\,(\omega_n|\Omega|)^{-\frac{2s}{n}}
\]
and hence
\[
\lim_{k\to\infty}\frac{\lambda_k^{s+\Log}(\Omega)\,k^{-\frac{2s}{n}}}{\ln k}
=
\frac{2s}{n}\,C_0^{-\frac{2s}{n}}
=
\frac{2}{n}\,(2\pi)^{2s}\,(\omega_n|\Omega|)^{-\frac{2s}{n}}.
\]
The proof ends. 
\end{proof}

 \bigskip
 \smallskip

\noindent{\bf Conflict of interest:}  {\small On behalf of all authors, the corresponding author states that there is no conflict of interest.}

\medskip

\noindent{\bf Acknowledgements:}  {\small We thank the referees for their detailed reports and valuable suggestions,
which greatly improved our writing. \smallskip

 H. Chen is supported by  the National Natural Science Foundation of China, No.  12361043.\smallskip
 
R. Chen is supported by China Scholarship Council, Liujinxuan [2025] no. 37.
}
    


\end{document}